\documentclass[11pt,letterpaper,twoside]{article}
\usepackage[latin1]{inputenc}
\usepackage[top=1.5in, bottom=1.5in, left=1.25in, right=1.25in]{geometry}

\usepackage{jmlr2esmall}

\usepackage{paralist}
\usepackage{color}
\usepackage{tikz}
\usepackage{pgfplots}
\usetikzlibrary{shapes,arrows,patterns,positioning,backgrounds}

\usepackage{graphicx}
\usepackage{caption}
\usepackage{subcaption}
\usepackage{tabularx}

\usepackage{siunitx} 

\usepackage[textsize=tiny]{todonotes}

\usepackage{amsmath,empheq}
\usepackage{ntheorem} 
\usepackage{amsfonts}
\usepackage{amssymb}
\usepackage[normalem]{ulem}
\usepackage{bm}
\usepackage{mathrsfs}
\usepackage{setspace}
\usepackage{xcolor}
\usepackage{mathtools,leftidx}
\usepackage{tcolorbox}
\usepackage{placeins}

\usepackage{algorithm}
\usepackage{algorithmicx}
\usepackage{algpseudocode}

\newtheorem{theorem}{Theorem}[section]
\newtheorem{lemma}[theorem]{Lemma}
\newtheorem{corollary}[theorem]{Corollary}
\newtheorem{proposition}[theorem]{Proposition}

\newtheorem{property}{Property}[section]
\newtheorem{definition}{Definition}[section]

\newtheorem{remark}{Remark}[section]

\newcommand{\qed}{\nobreak \ifvmode \relax \else
      \ifdim\lastskip<1.5em \hskip-\lastskip
      \hskip1.5em plus0em minus0.5em \fi \nobreak
      \vrule height0.75em width0.5em depth0.25em\fi}

\newenvironment{proof_without_of_suffix}[1][Proof]{\begin{trivlist}
\item[\hskip \labelsep {\bfseries #1}]}{\hfill $\blacksquare$ \end{trivlist}}

\newenvironment{proof}[1][Proof of]{\begin{trivlist}
\item[\hskip \labelsep {\bfseries #1}]}{\hfill $\blacksquare$ \end{trivlist}}

\DeclareMathOperator*{\argmin}{arg\,min}
\DeclareMathOperator*{\argmax}{arg\,max}

\DeclarePairedDelimiter\abs{\lvert}{\rvert}
\DeclarePairedDelimiter\norm{\lVert}{\rVert}
\newcommand{\lr}[1]{ \left( #1 \right)}
\newcommand{\tr}[1]{ \mathrm{trace}\left( #1 \right)}

\DeclarePairedDelimiter\floor{\lfloor}{\rfloor}
  
\newcommand{\upperRomanNum}[1]{\textup{\uppercase\expandafter{\romannumeral#1}}}
\newcommand{\lowerRomanNum}[1]{\textup{\lowercase\expandafter{\romannumeral#1}}}

\numberwithin{table}{section}
\numberwithin{figure}{section}

\usepackage{titlesec}
\titleformat{\paragraph}
{\normalfont\normalsize\bfseries}{\theparagraph}{1em}{}
\titlespacing*{\paragraph}
{0pt}{3.25ex plus 1ex minus .2ex}{1.5ex plus .2ex}

\usepackage[pdfauthor={Jayanth Jagalur-Mohan},
            pdftitle={BatchGreedy},
            pdfsubject={},
            pdfkeywords={},
            breaklinks=true,
            ]{hyperref}

\usepackage{cleveref}
\crefname{theorem}{theorem}{theorems}
\Crefname{theorem}{Theorem}{Theorems}
\crefname{lemma}{lemma}{lemmas}
\Crefname{lemma}{Lemma}{Lemmas}
\crefname{proposition}{proposition}{propositions}
\Crefname{proposition}{Proposition}{Propositions}
\crefname{corollary}{corollary}{corollaries}
\Crefname{corollary}{Corollary}{Corollaries}
\crefname{fact}{fact}{facts}
\Crefname{fact}{Fact}{Facts}
\crefname{property}{property}{properties}
\Crefname{property}{Property}{Properties}
\crefname{definition}{definition}{definitions}
\Crefname{definition}{Definition}{Definitions}
\crefname{equation}{}{}
\Crefname{equation}{Eq.}{Eqs.}
\crefname{figure}{figure}{figures}
\Crefname{figure}{Figure}{Figures}

\begin{document}

\title{Batch greedy maximization of non-submodular functions: Guarantees and applications to experimental design}

\author{\name Jayanth Jagalur-Mohan \email jagalur@mit.edu \\
        \name Youssef Marzouk  \email ymarz@mit.edu\\
        \addr Massachusetts Institute of Technology\\
              Cambridge, MA 02139 USA}

\maketitle

 \begingroup
 \leftskip6em
 \rightskip\leftskip
 \noindent\textit{{\small Dedicated to Thara Jagalur, a magnificent spirit who brought tremendous strength to her parents and showed them the transcendental ways of love. She is a shining star and will reverberate in the hearts of all those that knew her.}}
 \par
 \endgroup

\vspace{.3in}

\begin{abstract}
We propose and analyze batch greedy heuristics for cardinality constrained maximization of non-submodular non-decreasing set functions. 
We consider the standard greedy paradigm, along with its distributed greedy and stochastic greedy variants.
Our theoretical guarantees are characterized by the combination of submodularity and supermodularity ratios. We argue how these parameters define tight modular bounds based on incremental gains, and provide a novel reinterpretation of the classical greedy algorithm using the minorize--maximize (MM) principle. Based on that analogy, we propose a new class of methods exploiting any plausible modular bound.
In the context of optimal experimental design for linear Bayesian inverse problems, we bound the submodularity and supermodularity ratios when the underlying objective is based on mutual information. We also develop novel modular bounds for the mutual information in this setting, and describe certain connections to polyhedral combinatorics.
We discuss how algorithms using these modular bounds relate to established statistical notions such as leverage scores and to more recent efforts such as volume sampling. We demonstrate our theoretical findings on synthetic problems and on a real-world climate monitoring example.
\end{abstract}

\vskip 0.2in


	
\tableofcontents

\section{Introduction}

Many design problems in engineering and science are inherently combinatorial. A typical goal in these problems is to select a subset of indices from a larger candidate set, by maximizing or minimizing an objective subject to certain constraints. In resource allocation, for instance, a goal may be to assign tasks to different actors to ensure their expedited and efficient completion. 
In the climate sciences, a goal might be to create the most informative network of monitoring stations by identifying a set of locations from a larger list of plausible candidate sites.

Problems such as these can be formulated as the optimization of set functions. They have a rich mathematical structure and have been studied extensively \cite{Wolsey_Nemhauser_book,lovasz_book,Schrijver_book,Papadimitriou_Steiglitz_book}. Often the set functions exhibit properties such as \emph{submodularity} (\Cref{def:submodularity_using_firt_order_diff}), which allow for efficient optimization solutions \cite{fujishige2005}. The optimization of submodular functions has been a topic of intense research in the past decade owing to these functions' significance for many problems in  machine learning. The maximization and minimization of such functions under very general matroid constraints has been a focus of several studies \cite{Vondrak_2008,Vondrak_Chekuri_Zenklusen_2011,Jegelka_Bilmes_2011,Horel_Singer_2016,Sviridenko_etal_2017,Khanna_etal_2017_a}. For a catalog of historical efforts, as well as more recent investigations on these topics, we refer to the monographs \cite{Bach2013,Krause_Golovin_2014}.

While submodularity naturally arises in many problems, there are important problems where the objective lacks this property. Our interest is one such case, where the core task is to maximize a non-submodular but non-decreasing objective. Such objectives are ubiquitous in \emph{optimal experimental design} (OED), a fundamental problem in statistics which involves the specification of all aspects of an experiment. The mathematical foundations of experimental design have a history spanning almost a century \cite{Fisher_1936,Wald_1943}. Methods developed in this field are relevant to engineering, the social and medical sciences, and econometrics. In broad terms, a typical goal is to seek a design that provides the best ``return'' (suitably defined) for the least amount of experimental effort.
More specifically, we will consider optimal experimental design under cardinality constraints in the Bayesian setting, with an objective that reflects an end goal of parameter inference.
For a classical overview of experimental design, we refer the reader to \cite{Pukelsheim_oed_book,Fedorov_oed_book}, and for a perspective on Bayesian optimal experimental design, we highlight the review in \cite{Chaloner_etal_1995}.

When maximizing a submodular function under cardinality constraints, the greedy heuristic of successively picking the best candidate performs remarkably well despite its simplicity \cite{Nemhauser1978,Nemhauser1978b}.
The same technique works well in many cases with a non-submodular objective \cite{Das_Kempe_2011,Bian_etal_2017,Khanna_etal_2017_b,Elenberg_etal_2018,Qian_Singer_2019}.
Yet the repeated function evaluations necessary to update the incremental gains (see \Cref{def:incremental_gain}) at every step
may comprise
a staggering computational cost if each function evaluation is computationally intensive.
It is natural, then, to ask if there are favorable properties of the function that circumvent the need to repeatedly update the incremental gains. To what extent do such properties hold for a given function, and how can we exploit them? And how does the worst case bound suffer if the incremental gains are only periodically updated?

Motivated by these questions, we propose and analyze a variant of the greedy heuristic for non-submodular functions which we refer to as \emph{batch greedy}. We also analyze distributed and stochastic versions of this batch greedy approach.
Every step in these algorithms selects a batch of indices, thus reducing the overall number of function evaluations. In many practical problems---for instance Bayesian optimal experimental design---this heuristic still yields excellent solutions, as we shall demonstrate. 

The solution index set of the greedy algorithm results from a sequence of locally optimal choices.  The choice in every step is determined  by the incremental gains associated with each index, or---in the batch greedy setting---by the \emph{sum} of incremental gains corresponding to a fixed-size subset of indices. The sum of incremental gains is simply a modular function (\Cref{def:modular_function}) and the individual incremental gains define a subgradient  (\Cref{def:sub_grad_diff_SM}).
Analogous to results in the convex analysis for the optimization of continuous functions, the performance of the greedy algorithm can be tied to properties of the subgradient.
The choice of subgradient is not unique for any real valued continuous function, and the same is true for set functions. This observation motivates us to recast sequential greedy algorithms in a more abstract framework using minorize--maximize (MM) optimization principles. Using this framework, we propose several alternative modular bounds 
and use them to develop novel techniques for linear Bayesian optimal experimental design.

\subsection{Outline and summary of key contributions}

Much of the manuscript is written with Bayesian optimal experimental design as a motivation.
Yet our theoretical results are rather general: applicable to any monotone set function.
In describing these results (\Cref{sec:batch_greedy} and \Cref{subsec:greedy_as_mm}), we have favored a more abstract notation to help reach a broader readership. These sections can be read more or less independently of the rest of the paper.  With the same intention, we disperse much of the relevant literature survey into the sections that follow.
\begin{itemize}
\item In \Cref{sec:preliminaries_notation} we introduce some notation and background on set functions, and give an overview of the Bayesian experimental design problem. We formulate the design objectives we study, and comment on some of their set theoretic properties.

\item In \Cref{sec:batch_greedy} we discuss our primary results. We provide approximation guarantees for the batch greedy heuristic in the context of maximizing monotone but non-submodular objectives.
We analyze the standard batch greedy paradigm (\Cref{thm:batch_std_greedy}), along with distributed batch greedy (\Cref{thm:parallel_batch_std_greedy}) and stochastic batch greedy (\Cref{thm:stochastic_batch_std_greedy}) variants.
Our results portray the joint expressive power of the submodularity and supermodularity ratios. These two parameters quantify how much a function deviates from submodularity and supermodularity, respectively. We then bound these parameters for the information theoretic objective in linear Bayesian experimental design.

\item  Viewing the sub/super-modularity ratios as parameters that define tight modular bounds, we argue in \Cref{sec:mm_alg} that the classical greedy heuristic can be viewed as one instance of optimization based on the MM principle. Building on that insight, we put forth a general framework for optimizing monotone set functions using any modular bound (\Cref{thm:greedy_ascent}). 

\item In \Cref{sec:mm_alg_lbip}, for the problem of linear Bayesian experimental design, we explicitly discuss one such MM-based optimization approach exploiting modular bounds on the mutual information objectives. These modular bounds are constructed using standard inequalities but, as we show, they have surprising connections to notions in polyhedral combinatorics. 

\item In \Cref{sec:num_results} we investigate the performance of our algorithms on random instances of structured inverse problems, and on the problem of designing sensor networks for improving climate models. We end with a broader discussion of future directions in \Cref{sec:conclusions}. 

\item \Cref{appendix:def} contains some background to help with the reading, and all the technical results are collected in \Cref{appendix:techResults}.
In \Cref{appendix:extra_discussion_and_num_results} we provide some supplementary numerical results and more discussion that helps contextualize the core ideas explored in this work.

\end{itemize}


\section{Preliminaries and notation} \label{sec:preliminaries_notation}

\subsection{Set function properties and subset selection} \label{subsec:indexSubsetSelection}

Consider an index set $\mathscr{V} \coloneqq \lbrace 1, \ldots, m \rbrace, \ m \in\mathbb{Z}_{>0}$, and let its power set (i.e., set of all subsets) be denoted by $2^{\mathscr{V}}$. We will refer to  $\mathscr{V}$ as the candidate set. Any real-valued set-function $F : 2^{\mathscr{V}} \rightarrow \mathbb{R}$ such that $F(\emptyset)=0 $  is submodular \cite{fujishige2005,Bach2013} if and only if, for all subsets $\mathscr{A},\mathscr{B} \subseteq \mathscr{V}$, we have:
\begin{equation} \label{eqn:submodularity}
F(\mathscr{A}) + F(\mathscr{B}) \geq  F\lr{\mathscr{A} \cup \mathscr{B}} +  F\lr{\mathscr{A} \cap \mathscr{B}}.
\end{equation}
A function is \emph{supermodular} if its negation is submodular, and it is \emph{modular} (see also \Cref{def:modular_function}) if it is both supermodular and submodular.

An alternative but equivalent  \cite[Proposition~2.2]{Bach2013} definition of submodularity  highlights the \emph{diminishing returns} property 
and is often easier to demonstrate in practice:
\begin{definition}[Submodular set function defined using first-order differences \cite{fujishige2005,Bach2013}] \label{def:submodularity_using_firt_order_diff} The set function $F$ is submodular if and only if, for all $\mathscr{A},\mathscr{B} \subseteq  \mathscr{V}$ and $\nu \in \mathscr{V}$
such that $\mathscr{A} \subseteq \mathscr{B}$ and $\nu \notin \mathscr{B}$, we have
\begin{displaymath}
  F(\mathscr{A} \cup \lbrace \nu \rbrace) - F(\mathscr{A}) \geq F(\mathscr{B} \cup \lbrace \nu \rbrace) - F(\mathscr{B}).
\end{displaymath}
\end{definition}

In the experimental design problems we investigate, the indices of the candidate set will correspond to individual components of a multivariate random variable $Y \in \mathbb{R}^{m}$. For the purpose of selecting $k<m$ components we introduce the notion of a selection operator.
\begin{definition}[Selection operator] \label{def:selectionOperator} We refer to $\mathcal{P} \in \mathbb{R}^{m \times k},\  k < m$, as a \emph{selection operator} with index set $\mathscr{I}(\mathcal{P}) \equiv \mathscr{I}_{\mathcal{P}} = \{i_1,\ldots,i_{k}\} \subset \mathscr{V}$, when $\mathcal{P} = \begin{bmatrix} e_{i_1},  \ldots, e_{i_{k}}\end{bmatrix}$. The $e_{i_j}$ are distinct canonical unit vectors from the $m$-dimensional identity matrix $I_{m}$.
\end{definition}
Two selection operators that are unique up to permutation of their columns will have identical index sets, and are equivalent for our purposes. We will denote by $\mathcal{S}(k)$ the set of all permutation invariant selection operators $\mathcal{P}$ with $\abs{\mathscr{I}(\mathcal{P})} = k$. It is clear that $\abs{\mathcal{S}(k)} = \binom{m}{k}$. As an abstract operator, $\mathcal{P}$ can be understood as a full-column rank matrix, and an isometry on a $k$-dimensional subspace of $\mathbb{R}^{m}$.  Applying $\mathcal{P}$ on $Y$ allows us to select $k$ components out of $m$,
\begin{equation} \label{eqn:subset_of_Y}
 Y_{\mathcal{P}} \coloneqq \mathcal{P}^{\top}Y = \left[ Y_{i_1},\ldots,Y_{i_{k}} \right] ^{\top} \in \mathbb{R}^{k}.
\end{equation}

\subsection{Bayesian inference and optimal experimental design} \label{subsec:BIP_and_exp_design}
In Bayesian parameter inference, we seek to characterize the distribution of some parameters of interest, $X \in \mathbb{R}^{n}$, given a realization of some (related) observations $Y \in \mathbb{R}^{m}$. More specifically, having endowed $X$ with a prior distribution (whose density we denote by $\pi_{X}$) and knowing the conditional density of the data $\pi_{Y|X}$, we wish to characterize the posterior distribution $\pi_{X \vert Y = y^*}$ for some realization $y^*$ of $Y$.\footnote{To simplify this exposition, we assume all random variables to have densities with respect to a suitable base measure.}
%
%
The associated Bayesian optimal experimental design problem is to find the subset of observations $ Y_{\mathcal{P}}$ that is most informative about the parameters of interest $X$. Solving the inference problem with $Y_{\mathcal{P}}$, or equivalently those observations specified by the index set $\mathscr{I}_{\mathcal{P}}$, will result in a posterior  $\pi_{X | Y_{\mathcal{P}}}$ that in general differs from $\pi_{X | Y}$. We refer to the latter as the ``full posterior'' to indicate it is obtained by conditioning on all candidate observations.


To assess the quality of selected observation subset, we will evaluate the \emph{mutual information} between the inference parameter $X$ and the selected observations $Y_{\mathcal{P}}$. Of course, many other experimental design criteria could be considered and have been employed in the literature. Here we focus on mutual information as it is broadly applicable, and well-founded in Bayesian decision theory as measure of the utility of an experiment \cite{Lindley_1956} in fully non-Gaussian/nonlinear settings \cite{Chaloner_etal_1995}.
\begin{definition}[Mutual information \cite{cover_thomas_2005}] \label{def:mutual_info} Let two random variables $X$ and $Y$ have joint density $\pi_{X,Y}$ and let $\pi_{X}$ and $\pi_{Y}$ denote the densities of their respective marginals. The mutual information $\mathcal{I}\lr{X;Y}$ is the relative entropy or Kullback--Leibler divergence (\Cref{def:KL_divergence}) between the joint density and the product density $\pi_{X}\pi_{Y}$.
\begin{displaymath}
	\mathcal{I}\lr{X;Y}\coloneqq D_{KL}\lr{ \pi_{X,Y} \| \pi_{X} \pi_{Y} } = \mathbb{E}_{\pi_{X,Y}} \log \lr{ \frac{\pi_{X,Y}}{\pi_{X} \pi_{Y}} }.
\end{displaymath}
\end{definition}
Bayes' rule allows us to rewrite mutual information as an \emph{expected information gain} from prior to posterior, i.e.,  $\mathcal{I}\lr{X;Y} = \mathbb{E}_{\pi_{Y}}  D_{KL}\lr{ \pi_{X|Y} \| \pi_{X} }$.
This interpretation is quite intuitive: a larger mutual information or expected information gain means that the posterior differs more strongly from the prior, on average.

Formally, we can now state the problem of interest as follows: Given a desired number of observations $k < m$, we seek a selection operator $\mathcal{P}_{\text{opt}} \in \mathcal{S}(k)$ (\Cref{def:selectionOperator}) such that the mutual information between the inference parameter $X$ and the selected observations $Y_{\mathcal{P}_{\text{opt}}}\coloneqq\mathcal{P}_{\text{opt}}^{\top}Y$ is maximized:
\begin{equation} \label{eqn:exp_design_prob_statement_max} \tag{\texttt{Prob-Max}}
	\mathcal{P}_{\text{opt}} = \argmax_{\mathcal{P} \in \mathcal{S}(k) \subset \mathbb{R}^{m \times k}}  \mathcal{I} \lr{ X;\mathcal{P}^{\top} Y }\, . 
\end{equation}

Alternatively, we could indirectly determine the selection operator by first finding a complementary set of observations that do not significantly inform the inference parameters. While this approach may seem convoluted, its value is easier to appreciate when the goal is to remove a small fraction of observations from the parent set while retaining the bulk.
To the best of our knowledge this has not  been studied in the context of Bayesian inference, but the underlying \emph{reverse/backward} principle can be found in earlier investigations concerning subset and feature selection \cite{Couvreur_Bresler_2000_backward,Zhang_2011_backward}, graph cut approximation \cite{Bian_etal_2015_backward}, optimization of a certain class of set functions \cite{Ilev_2001_backward}, and classical experimental design using elementary symmetric polynomials \cite{Zelda_Sra_2017}. 

For any selection operator $\mathcal{P} \in \mathcal{S}(k)$, let  $\mathcal{P}^{c} \in \mathcal{S}(m-k)$ denote its complement, meaning their corresponding index sets are such that $\mathscr{I}_{\mathcal{P}} \cup \mathscr{I}_{\mathcal{P}^c} = \mathscr{V}$ and $\mathscr{I}_{\mathcal{P}} \cap \mathscr{I}_{\mathcal{P}^c} = \emptyset.$
Our objective now is to minimize the loss of mutual information by discarding an optimal subset of observations $Y_{\mathcal{P}^c_{\text{opt}}}\coloneqq\mathcal{P}^c_{\text{opt}} {}^{\top}Y$. The loss is measured with respect to mutual information as determined by the set of all candidate observations.
\begin{equation} \label{eqn:exp_design_prob_statement_min} \tag{\texttt{Prob-Min}}
	\mathcal{P}^c_{\text{opt}} = \argmin_{\mathcal{P}^c \in \mathcal{S}(m - k) \subset \mathbb{R}^{m \times m- k}}  \mathcal{I} \lr{ X; Y } - \mathcal{I} \lr{ X;Y \setminus {\mathcal{P}^c}^{\top} Y }.
\end{equation}
The optimal index sets pertinent to \cref{eqn:exp_design_prob_statement_max} and \cref{eqn:exp_design_prob_statement_min} are identical; hence we do not distinguish between them notationally.
Note, however, that the objectives differ in how they vary as functions of $\mathscr{I}_{\mathcal{P}}$.
The objective in \cref{eqn:exp_design_prob_statement_max} is a non-decreasing function with respect to cardinality of $\mathscr{I}_{\mathcal{P}}$, and consequently  non-increasing with respect to $\mathscr{I}_{\mathcal{P}^{c}}$. In contrast, the objective in \cref{eqn:exp_design_prob_statement_min} has opposite relationships with respect to the same cardinalities.
More importantly, the two approaches have philosophical and practical differences.

We show in \Cref{prop:MI_SM_uncorrelatedCase} that the objective in \cref{eqn:exp_design_prob_statement_max}, 
is a submodular function when the observations are conditionally independent. 
\begin{proposition}\label{prop:MI_SM_uncorrelatedCase} 
Given random variables $X \in \mathbb{R}^{n}$ and $Y \in \mathbb{R}^{m}$, let $\mathcal{P} \in \mathbb{R}^{m \times k }$ be a selection operator such that $\mathcal{P}^{\top}Y = \left[ Y_{i_1},\ldots,Y_{i_{k}} \right] ^{\top}$, with $\mathscr{I}_{\mathcal{P}}\subset \mathscr{V}$. The mutual information $\mathcal{I}(X;\mathcal{P}^{\top}Y)$,  between $X$ and $\mathcal{P}^{\top}Y$ is submodular if $Y_{i_j} | X$ are independent.
\end{proposition}
\hyperref[proof:prop_MI_SM_uncorrelatedCase]{The proof is given in \Cref{appendix:techResults}.}\medskip

It is worth emphasizing that the submodularity of the mutual information $\mathcal{I}(X;\mathcal{P}^{\top}Y)$ for conditionally independent observations holds even when the underlying joint distribution is non-Gaussian.
As a simple corollary of \Cref{prop:MI_SM_uncorrelatedCase}, the objective in \cref{eqn:exp_design_prob_statement_min} can be shown to be supermodular  under the same assumptions. 
More precisely, it is the \emph{supermodular dual}, but we will defer more discussion on the topic of {duals} to \Cref{sec:mm_alg_lbip}.
	
\section{Batch greedy algorithms for maximizing monotone set functions} \label{sec:batch_greedy}

We now  focus solely on cardinality-constrained maximization of monotone set functions. We will not restrict these functions to be submodular. We begin with a summary of various existing greedy heuristics, to contextualize the results that follow.

\subsection{Greedy algorithms: a brief history} \label{subsec:batch_greedy_intro}

In the case of cardinality-constrained maximization of non-decreasing submodular functions, the greedy heuristic of successively picking the candidate corresponding to the highest incremental gain (\Cref{def:incremental_gain}) performs well despite its simplicity. It has a constant factor ($1-1/e$) approximation guarantee \cite{Nemhauser1978}, which cannot be improved in general by any other polynomial time algorithm \cite{Nemhauser1978b}. If the function can be shown to have small curvature $c \in [0,1]$ (\Cref{def:total_curvature}), then the greedy algorithm possesses a more refined guarantee, $\frac{1}{c}(1-e^{-c})$ \cite{Conforti_Cornuejols_1984}. If the function is not submodular, then one can still provide an approximation guarantee by incorporating a submodularity ratio $\gamma \in [0,1]$ (see \Cref{def:sm_ratio_das_kempe}) \cite{Das_Kempe_2011} and a generalized curvature $\alpha \in [0,1]$ (\Cref{def:gen_curvature}) to obtain a similar factor, $\frac{1}{\alpha}(1-e^{-\alpha \gamma})$ \cite{Bian_etal_2017}.
 
\begin{definition}[Incremental gain] \label{def:incremental_gain} We denote the incremental or marginal gain of a set $ \mathscr{A} \subset \mathscr{V}$ given a set $\mathscr{B} \subset \mathscr{V}$ as $\rho_{\mathscr{A}}(\mathscr{B}) \coloneqq  F(\mathscr{A}\cup\mathscr{B}) -F(\mathscr{B})$. For $\nu  \in \mathscr{V}$, we use
the shorthand $\rho_{\nu}(\mathscr{B})$ for $\rho_{\{\nu\}}(\mathscr{B})$.
\end{definition}
Closely related variants of the greedy heuristic can be better choices depending on the context and needs.
\begin{itemize}
	\item In \cite{Robertazzi_Schwartz_1989}, an accelerated version was explored wherein the computed incremental gains are stored and exploited in the successive step, possibly reducing the overall number of function evaluations.
	\item In \cite{Mirzasoleiman_etal_2013}, the authors propose a two-stage parallelized version which reduces the number of function evaluations per parallel process.  The approximation guarantee for the algorithm, however, in general depends on the size of the candidate set and cardinality constraint. That dependence can only be overcome in special cases.
	\item In \cite{lazier_than_lazy_greedy} the authors analyze a randomized version of the greedy heuristic, termed stochastic greedy. This algorithm achieves, in expectation, a ($1-1/e -\epsilon$) approximation guarantee relative to the optimum solution. The number of function evaluations does not depend on the cardinality constraint, but linearly on the size of the candidate set, thus reducing the complexity substantially.
	\item In \cite{Liu_etal_2016} the authors analyzed the greedy heuristic wherein the locally optimal decision involved selecting the best possible set of $q>1$ indices.
	This necessarily requires evaluating incremental gains associated with all combinatorial possibilities, a potentially severe overhead but one which offers better guarantees. The algorithm was referred to as batch greedy in \cite{Liu_etal_2016}, where the batch size is the cardinality of the \emph{locally combinatorially optimal} set chosen in each step.
\end{itemize}

\subsection{Batch greedy algorithm and its analysis} \label{subsec:alg_analysis_batch_std_greedy}

The variants of greedy discussed in \Cref{subsec:batch_greedy_intro} were mostly analyzed only in the context of non-decreasing submodular functions. Our approach  (\Cref{alg:batch_std_greedy,alg:parallel_batch_std_greedy,alg:stochastic_batch_std_greedy}) can be understood to be yet another distinct variant of the greedy heuristic, but one which we analyze for the more general case of monotone non-submodular objectives. While we label our approach as batch greedy, it is unlike the algorithm in \cite{Liu_etal_2016} and in some sense its polar opposite.

In particular, we investigate the greedy strategy of picking multiple candidates in each step but relying solely on the incremental gains associated with individual candidates. This naturally reduces the computational overhead by avoiding combinatorial combinations, but at the expense of inferior approximation guarantees. It is in this sense the exact opposite of the algorithm in \cite{Liu_etal_2016}, since we are at the other end of the trade-off spectrum. 

\begin{algorithm}
\begin{algorithmic}[1]
\caption{Standard batch greedy algorithm}  \label{alg:batch_std_greedy}
\State {Input  $F,\mathscr{V},l,\{q_1,\ldots,q_l\}$}
\State {Initialize $\mathscr{A}=\emptyset$}
 \For {$i = 1 \text{ to }  l$}
  \State {Determine $\rho_{a}(\mathscr{A}) \quad \forall a \in \mathscr{V} \setminus \mathscr{A}$.}
  \State {Find $\mathscr{Q} \subseteq \mathscr{V} \setminus \mathscr{A}, \ \abs{\mathscr{Q}} = q_i$, comprising the indices with the highest incremental gains.}
  \State {$\mathscr{A} \leftarrow \mathscr{A} \cup \mathscr{Q}$}
\EndFor

\Return {Index set $\mathscr{A}$}
\end{algorithmic}
\end{algorithm}

\Cref{alg:batch_std_greedy} describes the standard batch greedy algorithm. 
The total cardinality constraint $k$ is necessarily the sum of batch sizes across all steps, $k = \sum_i q_i$.
We seek approximation guarantees for \Cref{alg:batch_std_greedy}  pertinent to the maximization of any  non-decreasing set function. To aid our arguments we introduce the \emph{supermodularity ratio}, which has very recently been used in other contexts too \cite{Tzoumas_etal_2017_SupModRatio,Bogunovic_etal_2018_SupModRatio,Karaca_etal_2018_SupModRatio}.
\begin{definition} [Supermodularity ratio] \label{def:supermodularity_ratio}  The supermodularity ratio of a non-negative set function $F$ with respect to a set $\mathscr{V}$ and a parameter $k \geq 1$ is
\begin{displaymath}
\eta_{\mathscr{V},k}(F)  =  \min_{\mathscr{B} \subseteq \mathscr{V}, \mathscr{A}: \abs{\mathscr{A}} \leq k, \mathscr{A} \cap \mathscr{B} = \emptyset} \frac{\rho_{\mathscr{A}}(\mathscr{B})}{ \sum_{\nu \in \mathscr{A}} \rho_{\nu}(\mathscr{B})}.
\end{displaymath}
\end{definition}
The supermodularity ratio is inspired by and related to the submodularity ratio (\Cref{def:sm_ratio_das_kempe}), originally introduced in \cite{Das_Kempe_2011}. It is a lower bound on the ratio of the incremental gain associated with any set compared against the sum of incremental gains associated with its elements. In contrast, the reciprocal of submodularity ratio is an upper bound on the same quantity.
In \cite{Bian_etal_2017} the authors define the submodularity ratio without the cardinality parameter $k$ (\Cref{def:sm_ratio_bian}) by taking a minimum across all possibilities. 
In the same way, we can define the supermodularity ratio without the cardinality parameter as the largest scalar $\eta$ such that
\begin{displaymath}
\frac{\rho_{\mathscr{A}} (\mathscr{B})}{\sum_{\nu \in \mathscr{A} \setminus \mathscr{B}} \rho_{\nu}(\mathscr{B})} \geq \eta ,
\qquad  \forall \mathscr{A}, \mathscr{B} \subseteq \mathscr{V}.
\end{displaymath}
We will refer to both $\eta_{\mathscr{V},k}$ and $\eta$ as supermodularity ratio, preferring one over the other depending on the context. Informally,  the supermodularity ratio quantifies how close a set function is to being supermodular, while the submodularity ratio performs the same task for  submodularity. More formally, we can prove that a function $F$  is supermodular \emph{iff} the supermodularity ratio $\eta(F) = 1$
(\hyperref[proof:prop_supermodularity_ratio]{see proof in \Cref{appendix_subsec:greedy_stuff}}).
For all set functions which are not supermodular (but which may or may not be submodular) we have the condition $\eta < 1$ as a direct corollary. 

Now we state the result corresponding to \Cref{alg:batch_std_greedy}.
\begin{theorem} \label{thm:batch_std_greedy}
Let $F$ be a non-decreasing  function with $F(\emptyset)=0$.  The batch greedy algorithm for maximizing $F(\mathscr{A})$ subject to $\abs{\mathscr{A}} \leq k$ outputs a set $\mathscr{A}$ such that
\begin{displaymath}
	F(\mathscr{A}) \geq \lr{ 1 - \prod_{i=1}^{l} \lr{ 1 - \frac{q_i  \eta_{\mathscr{V},q_i} \gamma_{\mathscr{V},k}}{k}} } \max_{\mathscr{B} \subset \mathscr{V}, \abs{\mathscr{B}} \leq k} F(\mathscr{B})
\end{displaymath}
where $\gamma_{\mathscr{V},k}$ is the submodularity ratio and $\eta_{\mathscr{V},q_i}$ is the supermodularity ratio.
\end{theorem}
\hyperref[proof:theorem_batch_std_greedy]{The proof is given in \Cref{appendix_subsec:greedy_stuff}.}\medskip

If we chose $q$ indices during each step with the total number of indices $k=q l$, the approximation guarantee in  \Cref{thm:batch_std_greedy} can be simplified using a standard logarithmic inequality:
\begin{equation} \label{eqn:batch_std_greedy_approx_guarantee}
	F(\mathscr{A}) \geq  \lr{1 - e^{- \eta_{\mathscr{V},q} \gamma_{\mathscr{V},k} }} \max_{\mathscr{B} \subset \mathscr{V}, \abs{\mathscr{B}} \leq k} F(\mathscr{B}).
\end{equation}
Alternatively $q$ can also be the maximum batch size across all steps, $q = \max q_i$, $i=1,\ldots,l$.
The approximation guarantee as given in \cref{eqn:batch_std_greedy_approx_guarantee} can be viewed as a straightforward generalization of that in \cite{Das_Kempe_2011} pertinent to the variant of greedy with batch selection.
Letting the function be submodular, $\gamma_{\mathscr{V},k} \geq 1$, and choosing one index in every step, $q=1$, reduces the guarantee to the classical result by \cite{Nemhauser1978} since $ \eta_{\mathscr{V},1} = 1$ for any set function. Selecting more than one index during each step worsens the guarantee since $\forall k_1,k_2$ with $k_1 \geq k_2 \geq 1$ we have $\eta_{\mathscr{V},k_1} \leq \eta_{\mathscr{V},k_2}$; this claim holds by definition because the constraint set of the latter is contained in the former.  

In \Cref{fig:approx_ratio} we visualize the approximation factor $1 - e^{- \eta_{\mathscr{V},q} \gamma_{\mathscr{V},k} }$ in \cref{eqn:batch_std_greedy_approx_guarantee} for the range of submodularity and supermodularity ratios.
\begin{figure}[!ht]
  	\centering \includegraphics[width=.75\textwidth,angle=0]{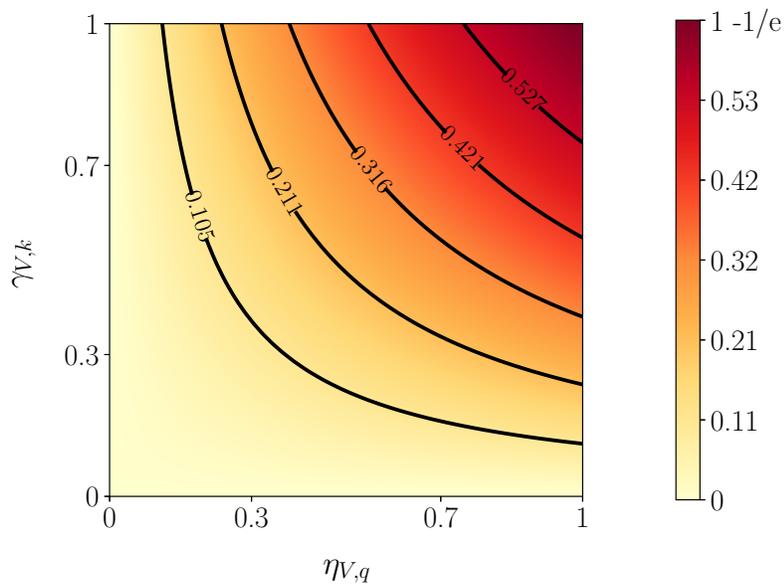}
  	 \caption{Visualization of the approximation factor $1 - e^{- \eta_{\mathscr{V},q} \gamma_{\mathscr{V},k} }$ in \cref{eqn:batch_std_greedy_approx_guarantee}.}
  	\label{fig:approx_ratio}
\end{figure}
If the function is modular, the approximation factor $1 - e^{- \eta_{\mathscr{V},q} \gamma_{\mathscr{V},k}}$ reduces to $1-1/e$, highlighting a gap  in our analysis similar to \cite{Das_Kempe_2011}. The supermodularity ratio $\eta_{\mathscr{V},q}$ is necessary to characterize the batch greedy heuristic, while the submodularity ratio $\gamma_{\mathscr{V},k}$ characterizes the non-submodularity of the function. Together, as they appear in \cref{eqn:batch_std_greedy_approx_guarantee} they are not sufficient to refine the worst case bound accounting for modularity. The notions of curvature (\Cref{def:total_curvature}) \cite{Conforti_Cornuejols_1984} for submodular functions and generalized curvature (\Cref{def:gen_curvature}) \cite{Bian_etal_2017} for non-submodular functions have been used in non-batch settings to give refined approximation guarantees avoiding such gaps. It should be possible to incorporate curvature in an analysis of the batch greedy heuristic to generalize the result in \Cref{thm:batch_std_greedy}, but we defer this investigation to the future.

If we consider the result in \Cref{thm:batch_std_greedy} directly and simplify it for the case when we choose all the indices in one step, we get a tight bound: 
\begin{equation} \label{eqn:batch_std_greedy_approx_guarantee_special_case}
	F(\mathscr{A}) \geq   \eta_{\mathscr{V},k} \gamma_{\mathscr{V},k} \max_{\mathscr{B} \subset \mathscr{V}, \abs{\mathscr{B}} \leq k} F(\mathscr{B}).
\end{equation}
The product $\eta_{\mathscr{V},k} \gamma_{\mathscr{V},k} \leq 1$ by definition. The algorithm returns an optimal index set if $\eta_{\mathscr{V},k} \gamma_{\mathscr{V},k}= 1$. 
One instance of this scenario is when the function is modular, in which case it is both supermodular and submodular, meaning $\gamma_{\mathscr{V},k}=\eta_{\mathscr{V},k}=1, \forall k$ and hence  $\eta_{\mathscr{V},k} \gamma_{\mathscr{V},k} = 1$,  $\forall k$. Of course,  we can always optimize a modular function exactly, and hence this observation is not entirely useful on its own. If we view the product of supermodularity and submodularity ratios as a measure of deviation from modularity, however, then they together prescribe favorable circumstances for a function that is not necessarily modular to be \emph{almost} maximized exactly. 


\subsection{Distributed batch greedy algorithm and its analysis} \label{subsec:alg_analysis_parallel_batch_std_greedy}

The distributed greedy algorithm for submodular maximization was first proposed in \cite{Mirzasoleiman_etal_2013}. 
Our distributed \emph{batch} greedy approach is described in  \Cref{alg:parallel_batch_std_greedy}.
The parameters $\widehat{l}$ and $\widetilde{l}$ are the number of iterations for which the batch greedy algorithm is run in each round. We define $\widehat{k}\coloneqq q \widehat{l}$ and $\widetilde{k}  \coloneqq q  \widetilde{l}$ as the cumulative number of indices returned at the end round. Similar to the setup in \cite{Mirzasoleiman_etal_2013}, we allow for  index sets larger than the cardinality constraint to be returned, meaning $k \leq \min ( \widehat{k},\widetilde{k} )$.
\begin{algorithm}
\begin{algorithmic}[1]
\caption{Distributed batch greedy algorithm}  \label{alg:parallel_batch_std_greedy}
\State {Input  $F,\mathscr{V}$, $n_p$, $\widehat{l},\widetilde{l},q$}
\State {Partition the set $\mathscr{V}$ into $n_p$ sets $\mathscr{V}_1,\ldots,\mathscr{V}_{n_p}$}
\State {Run the standard batch greedy algorithm, with batch size $q$,
for $\widehat{l}$ iterations within each set $\mathscr{V}_i$ to yield the corresponding solution set $\mathscr{A}_{i,[n_p,\widehat{k}]}^{\text{bg}}$}
\State {Merge the result sets: $\bigcup_i \mathscr{A}_{i,[n_p,\widehat{k}]}^{\text{bg}} \eqqcolon  \mathscr{M}$}
\State {Run the standard batch greedy algorithm, with batch size $q$,
for $\widetilde{l}$ iterations on $\mathscr{M}$ to yield the solution set $\mathscr{A}_{[n_p,\widetilde{k}]}^{\text{d-bg}}$}

\Return {Index set $\mathscr{A}_{[n_p,\widetilde{k}]}^{\text{d-bg}}$}
\end{algorithmic}
\end{algorithm}
 
In the most general case, the input to \Cref{alg:parallel_batch_std_greedy} could include different batch sizes in each round, across parallel processes, and across different iterations. For the purpose of exposition and analysis, we have favored a simplified version of the algorithm with a fixed/uniform batch size $q$.
Our theoretical analysis, summarized as \Cref{thm:parallel_batch_std_greedy} below, extends the work of  \cite{Mirzasoleiman_etal_2013} not only to batch settings but also to non-submodular functions.
\begin{theorem} \label{thm:parallel_batch_std_greedy}
Let $F$ be a non-decreasing  function with $F(\emptyset)=0$.  The distributed batch greedy algorithm for maximizing $F(\mathscr{A})$ subject to $\abs{\mathscr{A}} \leq k=ql$ outputs a set $\mathscr{A}_{[n_p,\widetilde{k}]}^{\text{d-bg}}$ such that
\begin{displaymath}
	F\left (\mathscr{A}_{[n_p,\widetilde{k}]}^{\text{d-bg}} \right ) \geq 
\lr{1 - e^{- \eta_{\mathscr{V},q} \, \gamma_{\mathscr{V},\widehat{k}} \, ( \widetilde{l} / \widehat{l} ) }}  \lr{1 - e^{-  \eta_{\mathscr{V},q} \, \gamma_{\mathscr{V},k} \, (\widehat{l} / l ) }} \frac{\gamma_{\mathscr{V},k}}{k}  \max_{\mathscr{B} \subset \mathscr{V}, \abs{\mathscr{B}} \leq k} F(\mathscr{B}),
\end{displaymath}
where $\gamma_{\mathscr{V},k}$, $\gamma_{\mathscr{V},\widehat{k}}$ are submodularity ratios and $\eta_{\mathscr{V},q_i}$ is the supermodularity ratio.
\end{theorem}
\hyperref[proof:theorem_parallel_batch_std_greedy]{The proof is given in \Cref{appendix_subsec:greedy_stuff}.} 
\medskip

If $F$ is submodular ($\gamma=1$), and further if $\widehat{l} = \widetilde{l} = l$, then \Cref{thm:parallel_batch_std_greedy} can be simplified as follows:
\begin{equation} \label{eqn:parallel_batch_std_greedy_submodular_case}
	F(\mathscr{A}^{\text{d-bg}}) \geq 
\frac { \lr{1 - e^{- \eta_{\mathscr{V},q}}}^2}{\min(n_p,k)}   \max_{\mathscr{B} \subset \mathscr{V}, \abs{\mathscr{B}} \leq k} F(\mathscr{B}).
\end{equation}
Note that the term appearing in the denominator of \cref{eqn:parallel_batch_std_greedy_submodular_case} is $\min(n_p,k)$ and not $k$, the natural expected term as per \Cref{thm:parallel_batch_std_greedy}. In this case of $F$ being submodular, we rely on \cite[Theorem~4.1]{Mirzasoleiman_etal_2013} as opposed to \Cref{lemma:parallel_greedy_helpful_lemma_1} to prove the slightly improved result in  \cref{eqn:parallel_batch_std_greedy_submodular_case}. 
The dependence of the distributed solution on $\min(n_p,k)$ is in general unavoidable, but as shown in \cite{Mirzasoleiman_etal_2013}, the ground set $\mathscr{V}$ and function $F$ can exhibit rich geometrical structure that can be used to prove stronger results.

\subsection{Stochastic batch greedy algorithm and its analysis} \label{subsec:alg_analysis_stochastic_batch_std_greedy}

The stochastic greedy algorithm for submodular functions was first proposed in \cite{lazier_than_lazy_greedy}.
Here we present a batch variant of the stochastic greedy algorithm, in \Cref{alg:stochastic_batch_std_greedy} below. 
\begin{algorithm}
\begin{algorithmic}[1]
\caption{Stochastic batch greedy algorithm}  \label{alg:stochastic_batch_std_greedy}
\State {Input  $F,\mathscr{V},l,\{q_1,\ldots,q_l\},s$}
\State {Initialize $\mathscr{A}=\emptyset$}
 \For {$i = 1 \text{ to }  l$}
  \State {Construct a randomly sampled set $\mathscr{R}$ by sampling $s$ random elements from $ \mathscr{V} \setminus \mathscr{A}$.}
  \State {Determine $\rho_{a}(\mathscr{A}), \ \forall a \in \mathscr{R}$.}
  \State {Find $\mathscr{Q} \subseteq \mathscr{R},\ \abs{\mathscr{Q}} = q_i$, comprising the indices with the highest incremental gains.}
  \State {$\mathscr{A} \leftarrow \mathscr{A} \cup \mathscr{Q}$}
\EndFor

\Return {Index set $\mathscr{A}$}
\end{algorithmic}
\end{algorithm}

In the most general case, the input to \Cref{alg:stochastic_batch_std_greedy} includes batch sizes $q_i,\  i=1,\ldots,l$, indicating the number of indices selected at each step. We will analyze the algorithm (wlog) for the case of fixed batch size $q$ with the cardinality constraint $k = q l$. Furthermore we will restrict our analysis to the regime when the cardinality constraint is at most $\mathcal{O}(\sqrt{m})$;  this is very much reflective of practical scenarios when stochastic greedy algorithms are warranted. 
Our theoretical analysis extends the work of \cite{lazier_than_lazy_greedy} not only to batch settings but also to non-submodular functions; the main result is stated in \Cref{thm:stochastic_batch_std_greedy}.
\begin{theorem} \label{thm:stochastic_batch_std_greedy}
Let $F$ be a non-decreasing  function with $F(\emptyset)=0$, and $0<\epsilon<1$ be a tolerance parameter. The stochastic batch greedy algorithm for maximizing $F(\mathscr{A})$ with $s =\frac{m}{k} \log\frac{q}{\epsilon}$, $\abs{\mathscr{A}} \leq k \leq \floor{\sqrt{m/e} -1/e}$, and $\frac{m -2k}{2 e k^2} \geq  \frac{q-1}{\log^2 \frac{q}{\epsilon}}$ outputs a set $\mathscr{A}$ such that
\begin{displaymath}
	\mathbb{E}[F(\mathscr{A})]	
 \geq \lr{1- e^{- \lr{1 - \epsilon} \gamma_{\mathscr{V},k} \eta_{\mathscr{V},q}}} \max_{\mathscr{B} \subset \mathscr{V}, \abs{\mathscr{B}} \leq k} F(\mathscr{B}),
\end{displaymath}
where  $\gamma_{\mathscr{V},k}$ is the submodularity ratio and $\eta_{\mathscr{V},q}$ is the supermodularity ratio.
\end{theorem}
\hyperref[proof:theorem_stochastic_batch_std_greedy]{The proof is given in \Cref{appendix_subsec:greedy_stuff}.}\medskip

If the function $F$ is submodular ($\gamma=1$) and if the batch size is equal to one, then our result reduces to that shown in \cite{lazier_than_lazy_greedy}. 
It is worth highlighting that the assumptions of \Cref{thm:stochastic_batch_std_greedy}---specifically the last inequality relating $m$, $k$, $q$, and $\epsilon$---stem primarily from the complexity of a required probability bound, detailed as \Cref{lemma:stochastic_batch_std_greedy_prob_bound}. 
The complexities arise specifically when the batch size is greater than one. The corresponding constraints, however, should be easily satisfied in practice. 
For instance, suppose that $k \leq \frac{1}{\kappa} \sqrt{m/e}$ for some $\kappa > 1$. Then it is easy to see that $\frac{m -2k}{2 e k^2} \geq  \frac{q-1}{\log^2 \frac{q}{\epsilon}}$ holds when $\epsilon \leq q e^{-\frac{\sqrt{2(q-1)}}{\kappa}}$. When $q=1$, this constraint on the tolerance parameter is trivial since we know $0<\epsilon<1$. When $q>1$, a sufficiently large $\kappa$ will ensure that the condition can be satisfied.

\subsection{Complexity of the batch greedy algorithms} \label{subsec:batch_greedy_complexity}

Recall that we denote the size of the candidate set $\mathscr{V}$ as $m$ and the desired cardinality as $k$.  The number of function evaluations needed to find the standard greedy solution is $\mathcal{O}(m k)$.

\begin{itemize}
\item In the batch heuristic (\Cref{alg:batch_std_greedy}) with a uniform batch size $q$, the number of function evaluations reduces to $\mathcal{O}( \frac{m k}{q})$. Thus a suitably chosen batch size reduces the computational overhead of a large cardinality constraint, but the linear dependence on the size of the candidate set remains the same. 
\item For the distributed batch greedy heuristic (\Cref{alg:parallel_batch_std_greedy}), if $\widehat{k} = \widetilde{k} = k$, then 
 the number of function evaluations per parallel process
to construct the set $\mathscr{M}$ is $\mathcal{O}(\frac{m k}{q n_p})$. To determine the distributed batch greedy solution using $\mathscr{M}$ we will further need $\mathcal{O}(\frac{k^2 n_p}{q})$ function evaluations.
\item For the stochastic batch greedy heuristic (\Cref{alg:stochastic_batch_std_greedy}),  the number of function evaluations needed to find the solution is $\widetilde{\mathcal{O}}(\frac{m}{q})$; here $\widetilde{\mathcal{O}}$ absorbs the $\log q$ dependency that arises since the random set $\mathscr{R}$ drawn at each step scales logarithmically with the batch size. The complexity does not depend on the cardinality constraint due to the probabilistic nature of the algorithm.
\end{itemize}

In all three instances, the batch variant of the greedy heuristic reduces the complexity by a factor $1/q$.

\subsection{Theoretical guarantees for linear Bayesian optimal experimental design} \label{subsec:theoretical_guarantees_LBIP_std_greedy}

The Bayesian linear--Gaussian model arises in numerous practical applications, and is a building block for countless others. It is particularly important in \emph{inverse problems} \cite{stuart2010inverse,Kaipio2005}, where the observations may depend indirectly on the parameters through the action of a smoothing forward operator. Examples of such problems include computerized tomography and electromagnetic source inversion. 
%
%
Here we discuss theoretical guarantees for the performance of \Cref{alg:batch_std_greedy,alg:parallel_batch_std_greedy,alg:stochastic_batch_std_greedy} when used in optimal experimental design for these problems. First, we set up the relevant notation and discuss some important features of linear Bayesian experimental design.

\subsubsection{Linear Bayesian inverse problem} \label{subsubsec:lbip_notation_setup}

Without loss of generality, we model the parameters $X$ and observation noise $\epsilon$ as zero-mean, normally distributed random variables with covariance matrices $\Gamma_{X}$ and $\Gamma_{Y|X}$, respectively. Here $Y$ denotes the observed data. We assume that $X$ and $\epsilon$ are independent of each other. The linear forward model that maps parameters to data is represented by $G \in \mathbb{R}^{m \times n}$. Hence
\begin{equation}\label{eqn:linearFwdModel}
	 Y = G X + \epsilon
\end{equation}
serves as our statistical model for the data and specifies the likelihood, i.e., $Y \vert x \sim \mathcal{N}( Gx, \Gamma_{Y \vert X})$.  The data $Y$ thus have a marginal distribution $\mathcal{N}\lr{0,\Gamma_Y}$, where the covariance $\Gamma_{Y}$ is
\begin{equation} \label{eqn:dataMargCovariance}
	\Gamma_{Y}  \coloneqq  G \Gamma_{X} G^{\top} + \Gamma_{Y|X}.
\end{equation}
The linearity of the forward model, along with Gaussianity of the prior and observation noise, allows us to characterize the posterior in closed form: $X | Y \sim \mathcal{N}(\mu_{X|Y} ,\Gamma_{X|Y})$. Here $\Gamma_{X|Y}$ is the posterior covariance matrix and $\mu_{X|Y}$ is the posterior mean, which is a function of the actual realization of the data $y$:
\begin{subequations}\label{eqn:posterior_moments}\begin{align}
\label{eqn:posterior_cov}	\Gamma_{X|Y} & \coloneqq   \lr{ \Gamma_{X}^{-1} + G^{\top}	\Gamma_{Y|X}^{-1} G }^{-1} ,\\
\label{eqn:posterior_mean}	\mu_{X|Y}(y)       & \coloneqq   \Gamma_{X|Y} G^{\top} \Gamma_{Y|X}^{-1} y.
\end{align} \end{subequations}
In \cref{eqn:posterior_cov}, the term $G^{\top}	\Gamma_{Y|X}^{-1} G $ is the Hessian of the negative log-likelihood (i.e., the Fisher information matrix).

The task of optimal experimental design involves selecting a \emph{subset} of observations $Y_{\mathcal{P}} = \mathcal{P}^{\top} Y \in \mathbb{R}^{k}$, corresponding to some selection operator $\mathcal{P} \in \mathbb{R}^{m\times k}$, such that $X | Y_{\mathcal{P}} \sim \mathcal{N}(\mu_{X|Y_{\mathcal{P}}} ,\Gamma_{X|Y_{\mathcal{P}}})$ has minimal uncertainty (given constraints on $k$). We make this goal precise by using the mutual information between $X$ and $Y_{\mathcal{P}}$, $\mathcal{I}\lr{X ; Y_{\mathcal{P}}}$, as our design objective. Analogous to \cref{eqn:posterior_moments} we can write the statistical moments of the posterior $X|Y_{\mathcal{P}}$ as
\begin{subequations}  \label{eqn:reduced_data_posterior_moments} \begin{align}
\label{eqn:reduced_data_posterior_cov}\Gamma_{X|Y_{\mathcal{P}}} & \coloneqq   \lr{ \Gamma_{X}^{-1} + G_{\mathcal{P}}^{\top}	\Gamma_{Y_{\mathcal{P}}|X}^{-1} G_{\mathcal{P}}}^{-1} ,\\
\label{eqn:reduced_data_posterior_mean}	\mu_{X|Y_{\mathcal{P}}}(y) & \coloneqq   \Gamma_{X|Y_{\mathcal{P}}} G_{\mathcal{P}}^{\top} \Gamma_{Y_{\mathcal{P}}|X}^{-1} \mathcal{P}^{\top}y.
\end{align} \end{subequations}
In \cref{eqn:reduced_data_posterior_moments} $G_{\mathcal{P}} : =\mathcal{P}^{\top}G \in \mathbb{R}^{k \times n}$
and $\Gamma_{Y_{\mathcal{P}}|X}  \coloneqq \mathcal{P}^{\top} \Gamma_{Y|X} \mathcal{P} : \mathbb{R}^{k \times k}$ is the compression of $\Gamma_{Y|X}$ by $\mathcal{P}$ to $\mathbb{R}^{k}$. In  linear algebraic terms, $\Gamma_{Y_{\mathcal{P}}|X}$ is a principal submatrix of $\Gamma_{Y|X}$. In a similar manner, the marginal covariance of $Y_{\mathcal{P}}$ is
\begin{equation} \label{eqn:ReducedDataMarginal}
\Gamma_{Y_{\mathcal{P}}}  \coloneqq  G_{\mathcal{P}} \Gamma_{X} G_{\mathcal{P}}^{\top} + \Gamma_{Y_{\mathcal{P}}|X} = \mathcal{P}^{\top} \Gamma_{Y} \mathcal{P} : \mathbb{R}^{k \times k}.
\end{equation}
From \cref{eqn:reduced_data_posterior_cov} we can identify $G_{\mathcal{P}}^{\top}	\Gamma_{Y_{\mathcal{P}}|X}^{-1} G_{\mathcal{P}}$ as the relevant Hessian term when the likelihood is specified using any subset of observations $Y_{\mathcal{P}}$. It is clear that this term is \emph{not} a compression of the full Hessian $G^{\top}	\Gamma_{Y|X}^{-1} G$ unless $\Gamma_{Y|X}$ is diagonal, meaning that the observation errors are uncorrelated. 

\subsubsection{Spectral properties of the mutual information} \label{subsubsec:lbip_exp_design_prop}

Since the inference parameters $X$ and the data $Y$ are jointly Gaussian random variables, the mutual information between them, $\mathcal{I}\lr{X;Y}$, can be written as
\begin{equation} \label{eqn:mutualInfoGaussians_using_logDet}
	\mathcal{I}(X;Y) =  \mathcal{I}(Y;X) = \frac{1}{2}\log { \frac{\det (\Gamma_{Y})}{\det (\Gamma_{Y|X})}} = \frac{1}{2} \log { \frac{\det (\Gamma_{X})}{\det (\Gamma_{X|Y})}}.
\end{equation}
The above expression can be easily verified from first principles. From \cref{eqn:mutualInfoGaussians_using_logDet}, it is clear that we could alternatively express mutual information as a function of the generalized eigenvalues of the definite pencils $(\Gamma_{Y},\Gamma_{Y|X})$ or  $(\Gamma_{X},\Gamma_{X|Y})$ (\Cref{def:definiteGEV}). The first can be viewed as a ``data space'' pencil, while the second is the corresponding ``parameter space'' pencil. Since $\Gamma_{Y} \succeq \Gamma_{Y|X} \succ 0$ and $\Gamma_{X} \succeq \Gamma_{X|Y} \succ 0$, the eigenvalues of both definite pencils are lower bounded by one. Here the symbol `$\succeq$' denotes the L\"{o}wner ordering, or the positive semi-definite ordering of Hermitian matrices (\Cref{def:lowner_ordering}). Furthermore, the two generalized eigenvalue problems, while differing in dimension, have identical generalized eigenvalues that are strictly greater than one. We make this fact precise through the following proposition.
\begin{proposition} \label{prop:eigen_prob_equivalence}
Let $X \in \mathbb{R}^{n}$ and $Y \in \mathbb{R}^{m}$ be jointly Gaussian random variables as defined in \Cref{subsubsec:lbip_notation_setup}. Then the following conditions are equivalent:
\begin{enumerate}
	\item \label{enumerate:prop_eigen_prob_equivalence_1} The definite pairs $(\Gamma_{Y} - \Gamma_{Y|X},\Gamma_{Y|X})$ and $(\Gamma_{X} - \Gamma_{X|Y},\Gamma_{X|Y})$ have identical non-trivial generalized eigenvalues, $\sigma_j > 0$.
	\item \label{enumerate:prop_eigen_prob_equivalence_2} The definite pairs $(\Gamma_{Y},\Gamma_{Y|X})$ and $(\Gamma_{X},\Gamma_{X|Y})$ have identical generalized eigenvalues that are strictly greater than $1$, $1+\sigma_j > 1$.
\end{enumerate}
\end{proposition}
%
\hyperref[proof:prop_eigen_prob_equivalence]{The proof is given in \Cref{appendix_subsec:lbip}.} \medskip

Using the notation of \Cref{prop:eigen_prob_equivalence}, we can now write $\mathcal{I}(X;Y)$ as follows:
\begin{equation} \label{eqn:mutualInfoGaussians_using_gev}
	\mathcal{I}(X;Y) = \frac{1}{2} \sum_{j} \log \lr{1 + \sigma_j}.
\end{equation}
Such an expression for mutual information in terms of the generalized eigenvalues $\sigma_j$ has been highlighted in many  works \cite{Alexanderian_etal_2016,giraldi_etal_2017} that adopt a Bayesian formalism. From a more classical statistics perspective, mutual information can be written using the squared canonical correlation scores \cite{Bach_Jordan_2002} between the concerned random variables.  These scores can be computed using generalized eigenvalue problems that are different from those in  \Cref{prop:eigen_prob_equivalence}, but not surprisingly have a similar dual representation.

If we were to determine $\mathcal{I}(X;Y_{\mathcal{P}})$, the mutual information between $X$ and a subset of observations $Y_{\mathcal{P}}$, then the relevant generalized eigenvalues are those of the definite pairs $(\Gamma_{Y_\mathcal{P}},\Gamma_{Y_\mathcal{P}|X})$ or  equivalently $(\Gamma_{X},\Gamma_{X|Y_\mathcal{P}})$. The equivalence holds since \Cref{prop:eigen_prob_equivalence} applies to these new definite pairs constructed using the compressed covariance operators. If we denote by $\widehat{\sigma}_{j}$ the eigenvalues for the case when we use a subset of observations $Y_{\mathcal{P}}$, then we can  write $\mathcal{I}(X;Y_{\mathcal{P}})$ in a manner analogous to \cref{eqn:mutualInfoGaussians_using_logDet,eqn:mutualInfoGaussians_using_gev} as	\begin{equation} \label{eqn:mutualInfoGaussians_reduced_data}
	\mathcal{I}(X;Y_{\mathcal{P}})  = \frac{1}{2}\log { \frac{\det (\Gamma_{Y_{\mathcal{P}}})}{\det (\Gamma_{Y_{\mathcal{P}}|X})}} = \frac{1}{2}\log { \frac{\det (\Gamma_{X})}{\det (\Gamma_{X|Y_{\mathcal{P}}})}} = \frac{1}{2} \sum_{j} \log \lr{1 + \widehat{\sigma}_{j}}.
\end{equation}

The expression for mutual information $\mathcal{I}(X;Y_{\mathcal{P}}) $ is quite revealing upon closer scrutiny. First, notice that we can rewrite \cref{eqn:mutualInfoGaussians_reduced_data} as the difference of log principal determinant of the data marginal and observation error covariance operators,
\begin{equation} \label{eqn:mutualInfoGaussians_reduced_data_DS_function}
	\mathcal{I}(X;Y_{\mathcal{P}})  = \frac{1}{2} \lr{ \log \det \lr{\mathcal{P}^{\top}\Gamma_{Y} \mathcal{P}} - \log \det \lr{\mathcal{P}^{\top} \Gamma_{Y|X} \mathcal{P}}}.
\end{equation}
It is thus the difference between two set functions, each one defined as the log determinant of the principal submatrix of a definite matrix.  Here the set comprises the indices of rows/columns of the principal submatrix, with each index corresponding to a unique observation. Each of those set functions can be inferred to be submodular \cite{Gantmacher_Krein_1960,Kotelyanskii_1950,Fan_1967,Fan_1968,Kelmans_1983,Johnson_1985}. 
Mutual information can thus be interpreted as the difference between two submodular functions (and hence is a DS function, following the terminology in \cite{IyerBilmes2012_diff2SM}). This is not entirely surprising since it has been shown that \emph{every} set function can be decomposed as the difference between two submodular functions \cite[Lemma~4]{NarBil2005}, \cite[Lemma~3.1]{IyerBilmes2012_diff2SM}.
This fact is akin to how any continuous function, subject to weak conditions, can be expressed as the sum of a convex and concave part \cite[Theorem~1]{Yuille_Rangarajan_2003}.
Such a decomposition is not unique and can be exponentially hard to compute \cite{IyerBilmes2012_diff2SM}, but in our case it is readily apparent.

The sum of the generalized eigenvalues, $\sum \widehat{\sigma}_j$, is the expected \emph{symmetrized} Kullback--Leiber divergence between the prior $\pi_X$ and posterior $\pi_{X|Y_{\mathcal{P}}}$; this relationship is shown in \Cref{prop:expSymKLDiv}. Note that the expected symmetrized Kullback-Leiber divergence\\
\noindent $\mathbb{E}_{\pi_{Y_{\mathcal{P}}}} \left[ D_{\text{KL}}(\pi_{X|Y_{\mathcal{P}} } \| \pi_{X} ) + D_{\text{KL}}(\pi_{X} \| \pi_{ X|Y_{\mathcal{P}}} )\right] $ is nothing but the mutual information
$\mathcal{I}(X;Y_{\mathcal{P}})$ 
plus an additional term.
It is worth contrasting this design criterion with that used in traditional Bayesian A-optimal design, where the trace of the posterior covariance operator (i.e., sum of simple eigenvalues) is minimized. In the former case the generalized eigenvalues correspond to the largest reductions in posterior variance \emph{relative} to the prior \cite[Corollary~3.1]{Spantini_etal_2015}, while in the latter case the simple eigenvalues represent the largest absolute contributions to the posterior variance, without regard to the prior. On the other hand, Bayesian D-optimal design in the linear--Gaussian case, wherein we minimize the determinant of the posterior covariance operator, is identical to maximizing mutual information; this can be easily inferred from \cref{eqn:mutualInfoGaussians_reduced_data}.

\subsubsection{Bounds for the performance of the batch greedy algorithm} \label{subsubsec:bounds_batch_std_greedy}

We now provide bounds for the submodularity and supermodularity ratios of the set function  objective of \cref{eqn:exp_design_prob_statement_max} in the setting of linear Bayesian inverse problems, where the objective is given by  \cref{eqn:mutualInfoGaussians_reduced_data}. These bounds enter our theoretical guarantees for the solution of \cref{eqn:exp_design_prob_statement_max}  using \Cref{alg:batch_std_greedy}.

\begin{proposition} \label{prop:bound_sup_sub_modularity_ratio} 
In the linear--Gaussian setting defined in \Cref{subsubsec:lbip_notation_setup},
the submodularity ratio $\gamma$ and supermodularity ratio $\eta$ 
pertinent to \cref{eqn:exp_design_prob_statement_max} can be both lower bounded by $\frac{\log \zeta_{\text{min}}}{\log \zeta_{\text{max}}}$, where $\zeta$ is any generalized eigenvalue of the definite pair $\lr{\Gamma_Y, \Gamma_{Y|X}}$.
\end{proposition}
\hyperref[proof:prop_bound_sup_sub_modularity_ratio]{The proof is given in \Cref{appendix_subsec:greedy_stuff}.}\medskip

Recall from \Cref{subsubsec:lbip_notation_setup} that $\zeta = 1 +\sigma$, where $\sigma$ is any generalized eigenvalue of the definite pair  $\lr{\Gamma_Y - \Gamma_{Y|X}, \Gamma_{Y|X}}$. We know $\sigma_{\text{min}} \geq 0$, and thus $\zeta_{\text{min}} \geq 1$. 
If the least eigenvalue is $1$, its algebraic multiplicity is the difference between the number of observations and the dimension of the inference parameters (\Cref{prop:eigen_prob_equivalence}). In such cases we obtain a trivial bound on the submodularity and supermodularity ratios.
If the dimension of the inference parameters is greater than the number of observations, as is typically the case in inverse problems (for instance, consider the limit $n \to \infty$ 
analyzed in  \cite{stuart2010inverse}), we obtain a non-trivial bound on the same parameters.
The empirical performance of  \Cref{alg:batch_std_greedy} is however impressive in all cases as will be demonstrated. Such gaps between worst-case bounds and practical performance are commonplace in algorithmic analysis. In this particular case the culprit is not necessarily the analysis framework, but a loose bound on the parameters featuring in the approximation guarantee.

\subsection{Relationship to other algorithmic approaches} \label{subsec:contextualizing_batch_greedy_with_other_methods}

As we have previously remarked, our results \Cref{thm:batch_std_greedy,thm:parallel_batch_std_greedy,thm:stochastic_batch_std_greedy} can be interpreted as the natural generalization of certain existing results \cite{Das_Kempe_2011,Mirzasoleiman_etal_2013,lazier_than_lazy_greedy} to the batch---and in some cases also to the non-submodular---settings. We now contrast our results with specific aspects of these works and other relevant investigations.
\begin{itemize}
 	\item In \cite{Das_Kempe_2011} the authors describe the so-called ``oblivious'' algorithm, wherein indices are selected ignoring any possible dependence among them. 
 	In the batch greedy algorithm we propose, at each step we also ignore any dependence among the indices that are selected. However we \emph{do} account for dependencies with the indices already selected.
 	When the batch size is equal to the cardinality constraint, i.e., $q=k$, or equivalently if all indices are chosen in one step, i.e., $l=1$, the batch greedy heuristic is the oblivious algorithm.
 	When we have a batch size  greater than one but strictly less than the cardinality constraint, then the batch heuristic can be understood as the interleaving of the oblivious and classical greedy algorithms.
 	
 	\item Our ideas involving batch greedy heuristics are similar in spirit to adaptive sampling algorithms \cite{balkanski_etal_2018,Balkanski_singer_2018_b,Balkanski_singer_2018_a,
 	Chekuri_Quanrud_2019_b,Fahrbach_etal_2019,Chekuri_Quanrud_2019,Chen_etal_2019,Balkanski_etal_2019}, in that both try to add a large set of elements at each step as opposed to growing the solution set incrementally. In \cite{Qian_Singer_2019} the authors analyze an adaptive sampling algorithm for the maximization of non-submodular functions. 
 	That work introduces the notion of differential submodularity ($\alpha_{\text{DF}} \in [0,1]$), which captures how tightly incremental gains corresponding to a given non-submodular function can be bounded by incremental gains corresponding to submodular functions. The corresponding approximation guarantee for cardinality constraint maximization was shown to be $(1  -1 /e^{\alpha_{\text{DF}}^2} -\epsilon)$, where $\epsilon > 0$ is a tolerance parameter in the adaptive sampling algorithm. 
 	 \cite{Qian_Singer_2019} shows how feature selection for regression and Bayesian A-optimal experimental design are $\gamma^{2}$-differentially submodular, where $\gamma$ is the submodularity ratio.
 	Using similar strategies, it can be seen that the mutual information-based Bayesian experimental design criterion is also $\gamma^{2}$-differentially submodular.  
 	 Consequently, the adaptive algorithm in \cite{Qian_Singer_2019} corresponds to a $(1  -1 /e^{\gamma^4} -\epsilon)$  performance guarantee.
 	 In our case, however, we have a strictly better
$(1  -1 /e^{\gamma^2})$ guarantee as described in \Cref{thm:batch_std_greedy}.  The approximation factor $(1  -1 /e^{\gamma^2})$ is obtained by recognizing that the submodularity and supermodularity ratios are lower bounded by the same term (see \Cref{prop:bound_sup_sub_modularity_ratio}) and hence we can replace $\eta$ by $\gamma$ in \cref{eqn:batch_std_greedy_approx_guarantee}.
It is worth emphasizing that our approximation factor $(1  -1 /e^{\gamma^2})$ is independent of batch size, and has no tolerance parameter.
	
\end{itemize}
\section{Sequential greedy algorithms based on the MM principle} \label{sec:mm_alg}

We now explore algorithms of a slightly different flavor, which are still greedy in a sense but do not necessarily make the locally optimal choice corresponding to incremental gain. They are based on the MM  (minorize--maximize or majorize--minimize)  principle \cite{Lange_2016_MMOpt}, and involve maximizing (resp.\ minimizing) at every step a minorizing (resp.\ majorizing) counterpart function to eventually obtain the $\argmax$ (resp.\ $\argmin$). Every instance of applying the MM principle is prompted by the need to transform a hard optimization problem into a sequence of simpler ones. Here simplicity can refer to one of many desirable attributes, such as reduction in overall computational complexity, linearization or convexification of the problem, easing the handling of complex constraints, and so on. For a broad survey of techniques and algorithms based on the MM principle we refer the readers to the excellent text \cite{Lange_2016_MMOpt}.

While most of the work in optimization using the MM principle has been in the setting of continuous functions, in the recent past these have been adapted to the world of set functions \cite{IyerJegelkaBilmes2012_mirror,IyerJegBil2013}. In \cite{NarBil2005,IyerBilmes2012_diff2SM}, the authors explored these ideas for the most general case of  difference between two submodular functions. Concrete hardness results on multiplicative inapproximability of such optimization problems were also shown in the same work \cite{IyerBilmes2012_diff2SM}. The MM principle proves handy in these scenarios; it is deployed by first seeking tight modular bounds for one (or both) of the terms in the difference, and subsequently optimizing the resulting submodular (or modular) objective using existing methods. The key property of submodular functions that enables such approaches is that one can define a subgradient \cite{fujishige2005} (\Cref{def:sub_grad_diff_SM}) and supergradient \cite{Jegelka_Bilmes_2011} (\Cref{def:super_grad_diff_SM}) at every point, meaning for every subset in the power set. Calculation of those semigradients, however, will require repeated function evaluations of appropriate subsets, and the semidifferential sets (\Cref{def:sub_grad_diff_SM,def:super_grad_diff_SM}) that contain them constitute a polyhedral partitioning of $\mathbb{R}^{m}$ \cite{Iyer_Bilmes_2015_polyhedral}.

Our efforts utilizing the MM principle differ from the approaches summarized above and  perhaps to a certain degree from how MM algorithms are generally thought of. The techniques in \cite{IyerJegelkaBilmes2012_mirror,IyerJegBil2013,IyerBilmes2012_diff2SM} are applicable for fairly general constraints and can be best understood as an iterative framework where one continues until local convergence is achieved.  The algorithms we consider, however, are sequential in a `bonafide' sense, meaning an individual element or a batch of elements are selected at every step and appended to the previously chosen set of elements. Our focus is restricted to cardinality-constrained problems. 
We begin by providing a reinterpretation of the classical greedy heuristic. 

\subsection{Reinterpreting the classical greedy heuristic for maximization} \label{subsec:greedy_as_mm}

 The definitions of the submodularity and supermodularity ratios allow us to write the following bounds for the incremental gain $\rho_{\mathscr{A}}(\mathscr{B}) $,
\begin{equation}
	 \eta  \sum_{\nu \in \mathscr{A}} \rho_{\nu}(\mathscr{B}) \leq \rho_{\mathscr{A}}(\mathscr{B})  \leq \frac{1}{\gamma} \sum_{\nu \in \mathscr{A}}\rho_{\nu}(\mathscr{B}).
\end{equation}
If we view $\rho_{\mathscr{A}}(\mathscr{B})$ as a function of $\mathscr{A}$, then what we have above are modular bounds on $\rho_{\mathscr{A}}(\mathscr{B}) $ defined using the incremental gains $ \rho_{\nu}(\mathscr{B})$, $\nu \in \mathscr{A}$. Alternatively if we view the set $\mathscr{B}$ as one individual entity and the elements of $\mathscr{A}$  as  distinct entities, we can write modular bounds on $ F( \mathscr{A} \cup  \mathscr{B} )$ in terms of $F(\mathscr{B})$ and $ \rho_{\nu}(\mathscr{B})$, $\nu \in \mathscr{A}$:
\begin{equation} \label{eqn:modular_bound_for_FAuB_using_incremental_gain}
	F(\mathscr{B}) + \eta \sum_{\nu \in \mathscr{A}} \rho_{\nu}(\mathscr{B}) \leq F( \mathscr{A} \cup  \mathscr{B} )  \leq F(\mathscr{B}) + \frac{1}{\gamma} \sum_{\nu \in \mathscr{A}}\rho_{\nu}(\mathscr{B}).
\end{equation}
Thus we can view the submodularity and supermodularity ratios as parameters that allow us to define the tightest possible modular bounds of the above form for any $\mathscr{A},\mathscr{B} \subseteq \mathscr{V}$. Note that the bounds are also tight in the sense that if $\mathscr{A} \setminus \mathscr{B} = \emptyset$ then \cref{eqn:modular_bound_for_FAuB_using_incremental_gain} reduces to a trivial statement. Viewed through the lens of MM algorithms, the modular lower/upper bound serves as the minorizing/majorizing counterpart to $ F( \mathscr{A} \cup  \mathscr{B} )$.

With these perspectives, the greedy heuristic for maximization in \Cref{alg:batch_std_greedy} can be easily paraphrased in the language of MM algorithms. Having selected the set $\mathscr{A}_{j-1}$ after $j-1$ steps, we  seek $ \mathscr{A}_{j-1 \rightarrow j}  \subseteq \mathscr{V} \setminus \mathscr{A}_{j-1}$, $\abs{\mathscr{A}_{j-1 \rightarrow j} } = q_j$ that maximizes $F( \mathscr{A}_{j-1} \cup \mathscr{A}_{j-1 \rightarrow j} )$. We achieve this by maximizing the minorizing counterpart: the modular lower bound as in \cref{eqn:modular_bound_for_FAuB_using_incremental_gain} defined using the incremental gains $\rho_{\nu}(\mathscr{A}_{j-1})$, $\nu \in \mathscr{V} \setminus \mathscr{A}_{j-1} $.

Incremental gains associated with individual elements, coupled with the submodularity or  supermodularity ratios, are an obvious combination for defining modular bounds for a function. Many problems, however, may have additional attributes that enable other ways of defining modular bounds. The greedy heuristic can be adapted appropriately to utilize these special modular bounds. Such bounds could be easier to compute and perhaps perform better empirically. In \Cref{alg:greedy_ascent} we outline a greedy procedure built on the ansatz of an abstract modular lower  bound $\mathcal{M}_{\uparrow}\left[ \cdot \right]$ and modular upper bound  $\mathcal{M}_{\downarrow}\left[ \cdot \right]$ for the  incremental gain $\rho_{\mathscr{A}}(\mathscr{B}) $ associated with any $\mathscr{A}$, $\mathscr{B} \subseteq \mathscr{V}$:
\begin{equation}
	\mathcal{M}_{\uparrow}\left[ \rho_{\mathscr{A}}(\mathscr{B}) \right] \leq \rho_{\mathscr{A}}(\mathscr{B}) \leq \mathcal{M}_{\downarrow}\left[ \rho_{\mathscr{A}}(\mathscr{B}) \right].
\end{equation}
The symbols $\downarrow,\uparrow$ are a visual cue to pinpoint whether the bound is from above or below.

\begin{algorithm}
\begin{algorithmic}[1]
\caption{Greedy algorithm using modular lower bounds}  \label{alg:greedy_ascent}
\State {Input  $F,\mathscr{V},l,\{q_1,\ldots,q_l\}$}
\State {Initialize $\mathscr{A}=\emptyset$}
 \For {$i = 1 \text{ to }  l$}
  \State {Determine $\mathcal{M}_{\uparrow}\left[\rho_{\mathscr{V} \setminus \mathscr{A}}(\mathscr{A}) \right]$.}
  \State {Find $\mathscr{Q} \subseteq \mathscr{V} \setminus \mathscr{A}, \abs{\mathscr{Q}} = q_i$ that maximizes $\mathcal{M}_{\uparrow}\left[\rho_{\mathscr{Q}}(\mathscr{A}) \right]$}
  \State {$\mathscr{A} \leftarrow \mathscr{A} \cup \mathscr{Q}$}
\EndFor

\Return {Index set $\mathscr{A}$}
\end{algorithmic}
\end{algorithm}

Now we give performance bounds for the abstract batch greedy approach of \Cref{alg:greedy_ascent}.
\begin{theorem} \label{thm:greedy_ascent}
Let $F$ be a non-decreasing  function. \Cref{alg:greedy_ascent} for maximizing $F(\mathscr{A})$ subject to $\abs{\mathscr{A}} \leq k$ outputs a set $\mathscr{A}$ such that
\begin{displaymath}
	F(\mathscr{A}) \geq \lr{ 1 - \prod_{i=1}^{l} \lr{ 1 - \frac{q_i }{k \lr{1+\tau_i}} }} \max_{\mathscr{B} \subset \mathscr{V}, \abs{\mathscr{B}} \leq k} F(\mathscr{B})
\end{displaymath}
where $\tau_i > 0$ is a parameter that encapsulates the effectiveness of the modular lower bound $\mathcal{M}_{\uparrow}\left[ \cdot \right] $ in locally describing the function.
\end{theorem}
\hyperref[proof:theorem_greedy_ascent]{The proof is given in \Cref{appendix_subsec:greedy_stuff}.}\medskip

The analysis of \Cref{alg:greedy_ascent} is along the same lines as that of \Cref{alg:batch_std_greedy}. The effectiveness of the modular lower bound in locally describing the function is baked into the approximation guarantee via the parameter $\tau_i$. If $\mathscr{A}^{*}$ is the optimal set of the maximization problem, then for the $i^{\text{th}}$ step, $\tau_i$ is defined as the smallest scalar satisfying
\begin{equation} \label{eqn:tau_i_def}
  \rho_{\mathscr{A}^{*}} \lr{ \mathscr{A}_{i-1}} -  \mathcal{M}_{\uparrow}  \left[  \rho_{\mathscr{A}^{*}} \lr{ \mathscr{A}_{i-1}} \right] \leq  \frac{ \tau_i k \rho_{\mathscr{A}_{i}} \lr{ \mathscr{A}_{i-1}}}{q_i}.
\end{equation}
The dependence of $\tau_i$ is strictly on the  modular lower bound, but in many cases estimating it is easier when it is defined as the smallest scalar satisfying
\begin{equation} \label{eqn:tau_i_def_looser}
   \mathcal{M}_{\downarrow}  \left[  \rho_{\mathscr{A}^{*}} \lr{ \mathscr{A}_{i-1}} \right] -  \mathcal{M}_{\uparrow}  \left[  \rho_{\mathscr{A}^{*}} \lr{ \mathscr{A}_{i-1}} \right] \leq  \frac{ \tau_i  k  \mathcal{M}_{\uparrow}  \left[  \rho_{\mathscr{A}_{i}} \lr{ \mathscr{A}_{i-1}} \right]}{q_i}.
\end{equation}
Notice that $\tau_i$ estimated using the definition in \cref{eqn:tau_i_def_looser}  is always larger than what is necessary to satisfy \cref{eqn:tau_i_def}.

One instance of $\tau_i$ which we have previously seen, indirectly, is when the modular bounds are defined using incremental gains and submodularity/supermodularity ratios. In that case, with a little effort it is easy to see that $\tau_i$ according to \cref{eqn:tau_i_def_looser} will be
 $\frac{1 - \gamma_{\mathscr{V},k} \eta_{\mathscr{V},k} }{\gamma_{\mathscr{V},k} \eta_{\mathscr{V},q_i} }$. But this value is sub-optimal when compared to the $\tau_i $ one can implicitly infer from the result in \Cref{thm:batch_std_greedy}, which is $\frac{1 - \gamma_{\mathscr{V},k} \eta_{\mathscr{V},q_i} }{\gamma_{\mathscr{V},k} \eta_{\mathscr{V},q_i} }$. The sub-optimality stems from the fact that $\eta_{\mathscr{V},q_i} \geq \eta_{\mathscr{V},k}$, $\forall q_i \leq k$. The discrepancy is easily explained since in the proof of \Cref{thm:greedy_ascent} we allow for the modular lower bound $\mathcal{M}_{\uparrow}\left[ \cdot \right] $ to be negative. If we assume non-negativity of the modular lower bound, the proof can be accordingly modified to obtain the optimal result.
 

\subsection{Adapting MM techniques for linear Bayesian optimal experimental design} \label{sec:mm_alg_lbip}

In this section we develop specialized MM-based algorithms for the linear Bayesian optimal experimental design problem formulated in \Cref{subsubsec:lbip_exp_design_prop}.

\subsubsection{Modular bounds for mutual information design criteria} 

In \Cref{subsubsec:lbip_exp_design_prop} we discussed some key spectral properties of design criteria based on mutual information. Now  we derive helpful majorizations and minorizations using tight modular bounds on these information theoretic objectives.  

\begin{proposition}\label{prop:MBound_MI}
Let $X \in \mathbb{R}^{n}$ and $Y \in \mathbb{R}^{m}$ be jointly Gaussian random variables as defined in \Cref{subsubsec:lbip_notation_setup}. Let $\mathcal{P} \in \mathbb{R}^{m \times k }$ be a selection operator such that $Y_{\mathcal{P}} \coloneqq  \mathcal{P}^{\top}Y = \left[ Y_{i_1},\ldots,Y_{i_{k}} \right] ^{\top}$. The mutual information between $X$ and $Y_{\mathcal{P}}$ can be bounded on both sides by the following modular functions:
\begin{displaymath}
\begin{aligned}
 \mathcal{I}\lr{X;Y_{\mathcal{P}}} & \geq 
   \frac{1}{2} \tr{\mathcal{P}^{\top} \lr{ \log\lr{ \Gamma_{Y}} - \log \mathrm{diag}\lr{ \Gamma_{Y|X}}} \mathcal{P}},\\
 \mathcal{I}\lr{X;Y_{\mathcal{P}}} & \leq 
   \frac{1}{2} \tr{\mathcal{P}^{\top} \lr{ \log \mathrm{diag}\lr{\Gamma_{Y}} - \log \lr{ \Gamma_{Y|X}} } \mathcal{P}}   .
\end{aligned}
\end{displaymath}
\end{proposition}
\hyperref[proof:prop_MBound_MI]{The proof is given in \Cref{appendix_subsec:lbip_mm}.}\medskip

Note that the lower bound in the proposition above is always positive when mutual information is submodular (\Cref{corr:MI_SM_case_bounding_BP}), but in general it need not be.
The machinery used to obtain \Cref{prop:MBound_MI} exploits the linear algebraic structure of the information theoretic objectives. Using the same arsenal of tools we can provide the following alternative bounds:
\begin{subequations} \label{eqn:alt_MBound_MI}
\begin{align} 
 \mathcal{I}\lr{X;Y_{\mathcal{P}}} & \geq 
   \frac{1}{2} \tr{\mathcal{P}^{\top} \lr{ \log\lr{ \Gamma_{Y}} - \log \lr{ \Gamma_{Y|X}} + I_{m}\log\lr{\varrho_1} } \mathcal{P}} ,\\
 \mathcal{I}\lr{X;Y_{\mathcal{P}}} & \leq 
   \frac{1}{2} \tr{\mathcal{P}^{\top} \lr{ \log\lr{ \Gamma_{Y}} - I_{m}\log\lr{\varrho_2} - \log \lr{ \Gamma_{Y|X}} } \mathcal{P}}.
 \end{align} \end{subequations}
We include arguments for these bounds \cref{eqn:alt_MBound_MI} in the proof of \Cref{prop:MBound_MI} as well.
The constants $\varrho_1,\varrho_2 \in (0,1]$ depend on the range of the spectrum of covariance operators $\Gamma_{Y|X}, \Gamma_{Y}$ respectively. If $\text{eig}(\Gamma_{Y|X}) \in [\alpha_s,\alpha_l] \subset \mathbb{R}_{>0}$ and $\text{eig}(\Gamma_{Y}) \in [\beta_s,\beta_l] \subset \mathbb{R}_{>0}$, then $\varrho_1 \coloneqq  \frac{4 \alpha_{s} \alpha_{l}}{\lr{\alpha_{s} + \alpha_{l}}^{2}}$ and  $\varrho_2 \coloneqq  \frac{4 \beta_{s} \beta_{l}}{\lr{\beta_{s} + \beta_{l}}^{2}}$.
Note that lower bound in \cref{eqn:alt_MBound_MI} is not guaranteed to be positive even when mutual information is submodular, and the bounds in \cref{eqn:alt_MBound_MI} are tighter when the cardinality of $\mathscr{I}\lr{\mathcal{P}}$ is smaller.


To understand the modular bounds  in \Cref{prop:MBound_MI} and their significance, it is helpful to discuss two simple corollaries. We begin with the more restricted of the two, which is valid only when the observations are conditionally independent.
\begin{corollary}\label{corr:MI_SM_case_bounding_BP}
For random variables  $X \in \mathbb{R}^{n}$ and  $Y \in \mathbb{R}^{m}$ as defined in \Cref{prop:MBound_MI}, if $Y_{i_k} | X$ are additionally independent, then the following bounds hold for any selection operator $\mathcal{P} \in \mathbb{R}^{m \times k }$:
\begin{displaymath}
\mathcal{I}(X;Y)  - \mathcal{I}(X;Y\setminus Y_{\mathcal{P}}) 
\leq  \frac{1}{2} \tr{\mathcal{P}^{\top} \lr{\log \lr{ \Gamma_{Y}} - \log \lr{\Gamma_{Y|X}}} \mathcal{P}}
\leq  \mathcal{I}\lr{X;Y_{\mathcal{P}}}.
 \end{displaymath}
\end{corollary}
\hyperref[proof:corr_MI_SM_case_bounding_BP]{The proof is given in \Cref{appendix_subsec:lbip_mm}.}\medskip

Recall that  $\mathcal{I}\lr{X;Y_{\mathcal{P}}}$ is the objective corresponding to \cref{eqn:exp_design_prob_statement_max}, and $\mathcal{I}(X;Y)  - \mathcal{I}(X;Y\setminus Y_{\mathcal{P}})$ is the objective corresponding to \cref{eqn:exp_design_prob_statement_min}.
If $\mathcal{P}$ is square with its corresponding index set $\mathscr{I}_{\mathcal{P}} = \mathscr{V}$, then the corollary statement reduces to a simple identity. In all other circumstances the inequalities are strict.

We now highlight a few properties encapsulated by \Cref{corr:MI_SM_case_bounding_BP} which relate to concepts in polyhedral combinatorics.
\begin{itemize}
\item We know the mutual information $ \mathcal{I}\lr{X;Y_{\mathcal{P}}}$ is submodular when observations are conditionally independent (\Cref{prop:MI_SM_uncorrelatedCase}). The term $\mathcal{I}(X;Y)  - \mathcal{I}(X;Y\setminus Y_{\mathcal{P}})$  is the supermodular dual (\Cref{def:supermodular_dual}) of  $\mathcal{I}\lr{X;Y_{\mathcal{P}}}$, and both the functions have the same base polytope (\Cref{def:submodular_base_polyhedra}).

\item  The diagonal of the matrix appearing in the middle term deserves special attention.  We denote the diagonal of $\frac{1}{2} \lr{\log \lr{ \Gamma_{Y}} - \log \lr{\Gamma_{Y|X}}}$ by the vector $d \in \mathbb{R}^{m}$. Note that the covariance  $\Gamma_{Y|X}$ in this context is diagonal since we are presently discussing the case of uncorrelated observation error.  The vector $d$ can be inferred to be a point in the base polytope (\Cref{def:submodular_base_polyhedra}) of $ \mathcal{I}\lr{X;Y_{\mathcal{P}}}$ or equivalently of $\mathcal{I}(X;Y)  - \mathcal{I}(X;Y\setminus Y_{\mathcal{P}})$. 
As a result, it is naturally a subgradient (\Cref{def:sub_grad_diff_SM}) of $ \mathcal{I}\lr{X;Y_{\mathcal{P}}}$ and a supergradient of  $\mathcal{I}(X;Y)  - \mathcal{I}(X;Y\setminus Y_{\mathcal{P}})$ defined at $\mathscr{V}$ or $\emptyset$. The individual components of $d$ are all positive, which must be the case since the base polytope lies in the positive orthant for all non-decreasing submodular functions. 

\item Most interestingly, $d$ is not necessarily one of the extreme points of the base polytope but rather can lie in its interior. Note that there are at most $m!$ extreme points and these are at least in theory determinable using the `greedy algorithm' (Rado--Edmonds theorem \cite{Edmonds1971}). Typically subgradients defined using extremal points are used to  drive the optimization steps  \cite{IyerJegelkaBilmes2012_mirror,IyerJegBil2013,IyerBilmes2012_diff2SM}. Using a non-extremal point to drive optimization, as we do here, is a novel approach, albeit specific to the setting of linear--Gaussian Bayesian problems.

\end{itemize}

We now consider a generalization of \Cref{corr:MI_SM_case_bounding_BP} pertinent to the case of correlated observation errors. Recall that in this case the mutual information $\mathcal{I}\lr{X;Y_{\mathcal{P}}}$ is neither submodular nor supermodular, but simply a non-decreasing set function.
\begin{corollary} \label{corr:MI_deviation_from_SM}
For random variables  $X \in \mathbb{R}^{n}, Y \in \mathbb{R}^{m}$ as defined in \Cref{prop:MBound_MI}, the following bounds hold for any selection operator $\mathcal{P} \in \mathbb{R}^{m \times k}$:
\begin{align}
\mathcal{I}(X;Y)  - \mathcal{I}(X;Y\setminus Y_{\mathcal{P}}) 
\label{eqn:corr_MI_deviation_from_SM_eq1}    & \leq   \frac{1}{2}  \tr{\mathcal{P}^{\top} \lr{ \log\lr{ \Gamma_{Y}} -   \log\lr{ \Gamma_{Y|X}}} \mathcal{P}}  +\mathfrak{D}_{\mathcal{P}^{c}}, \\
\Longleftrightarrow
\mathcal{I}(X; Y_{\mathcal{P}}) 
\label{eqn:corr_MI_deviation_from_SM_eq2}    & \geq   \frac{1}{2}  \tr{\mathcal{P}^{\top} \lr{ \log\lr{ \Gamma_{Y}} -   \log\lr{ \Gamma_{Y|X}}} \mathcal{P}} - \mathfrak{D}_{\mathcal{P}},
\end{align}
where $\mathfrak{D}_{\mathcal{P}}  \coloneqq  \frac{1}{2} \tr{\mathcal{P}^{\top} \lr{\log \mathrm{diag} \lr{\Gamma_{Y|X}} - \log\lr{ \Gamma_{Y|X}}} \mathcal{P}}$.
\end{corollary}
\hyperref[proof:corr_MI_deviation_from_SM]{The proof is given in \Cref{appendix_subsec:lbip_mm}.}\medskip

The term $\mathfrak{D}_{\mathcal{P}}$ is non-negative and  strictly increases with increasing cardinality of the index set $\mathscr{I}(\mathcal{P})$.  When the observation errors are uncorrelated, meaning that the covariance operator $\Gamma_{Y|X}$ is diagonal, $\mathfrak{D}_{\mathcal{P}}$ evaluates to zero and we recover the statement in \Cref{corr:MI_SM_case_bounding_BP} by combining \cref{eqn:corr_MI_deviation_from_SM_eq1,eqn:corr_MI_deviation_from_SM_eq2}. Adding \cref{eqn:corr_MI_deviation_from_SM_eq1,eqn:corr_MI_deviation_from_SM_eq2} we obtain,
\begin{equation}
\mathcal{I}(X;Y)  - \mathcal{I}(X;Y\setminus Y_{\mathcal{P}})
 \leq  \mathcal{I}\lr{X;Y_{\mathcal{P}}}
  + \frac{1}{2}\tr{  \log \mathrm{diag} \lr{\Gamma_{Y|X}} - \log\lr{ \Gamma_{Y|X}}}.
 \end{equation}
Loosely speaking, $\mathfrak{D} \coloneqq \frac{1}{2}\tr{  \log \mathrm{diag} \lr{\Gamma_{Y|X}} - \log\lr{ \Gamma_{Y|X}}}$ reflects the deviation from submodularity of the mutual information $\mathcal{I}\lr{X;Y_{\mathcal{P}}}$: the smaller its magnitude, the more $\mathcal{I}\lr{X;Y_{\mathcal{P}}}$ behaves as a submodular function.
If $\mathfrak{D}_{\mathcal{P}^{c}}$ and $\mathfrak{D}_{\mathcal{P}}$ in \Cref{corr:MI_deviation_from_SM} are replaced by $\mathfrak{D}$, the inequalities are akin to affine modular bounds on the set functions $\mathcal{I}(X;Y)  - \mathcal{I}(X;Y\setminus Y_{\mathcal{P}}) $ and $\mathcal{I}(X; Y_{\mathcal{P}}) $.
The pair $\lr{d,\mathfrak{D}}$, where the vector $d \in \mathbb{R}^{m}$  is the diagonal of $\frac{1}{2} \lr{\log \lr{ \Gamma_{Y}} - \log \lr{\Gamma_{Y|X}}}$, characterizes the affine modular bound in \Cref{corr:MI_deviation_from_SM}. Such a pair in general defines generalized lower and upper  polyhedra (\Cref{def:generalized_polyhedra}) for arbitrary set functions that need not  be sub/super-modular. \cite{Iyer_Bilmes_2015_polyhedral} uses those theoretical constructs to study polyhedral aspects of submodular functions.

\subsubsection{MM algorithms for maximizing information gain} \label{subsec:maximize_info_gain}

For the problem of maximizing information gain \cref{eqn:exp_design_prob_statement_max}, wherein the objective is non-decreasing, we have previously discussed the performance of the standard batch greedy algorithm (\Cref{alg:batch_std_greedy}) in \Cref{subsec:theoretical_guarantees_LBIP_std_greedy}. We now consider \Cref{alg:greedy_ascent}, implemented with the modular lower bounds for incremental gain given 
in \Cref{prop:MBound_MI}. 
We will refer to this algorithm as \texttt{MMGreedy} in the comparative discussion of \Cref{sec:num_results}.
Note that incremental gain in the present context corresponds to information gain, which in turn can be written as conditional mutual information. The inequalities in \Cref{prop:MBound_MI} can be suitably adapted by replacing the covariance operators with the appropriate conditional covariances. We describe this process below.

Suppose our goal is to select $k<m$ observations altogether, and we have already selected $k_1< k$. Let the indices of the selected $k_1$ observations correspond to the selection operator $\mathcal{P}_1$,  $\mathscr{I}_{\mathcal{P}_1} \subset \mathscr{V}$. We now seek the remaining $k_2 = k - k_1$ observations whose corresponding index selection operator $\mathcal{P}_2$, with $\mathscr{I}_{\mathcal{P}_2} \subset \mathscr{V} \setminus \mathscr{I}_{\mathcal{P}_1}$, is  $\argmax_{\mathcal{P}, \abs{\mathscr{I}_{\mathcal{P}}}=k_2} \mathcal{I}(X;Y_{\mathcal{P}} | Y_{\mathcal{P}_1})$. In the proposed MM framework, this is achieved by maximizing the minorizing modular lower bound, 
\begin{equation} \label{eqn:max_incremental_gain_lowerbound}
\frac{1}{2} \tr{\widehat{\mathcal{P}}^{\top} \lr{ \log\lr{ \Gamma_{Y_{\mathcal{P}_1^c}|Y_{\mathcal{P}_1}}} - \log \mathrm{diag}\lr{ \Gamma_{Y_{\mathcal{P}^c_1}|X,Y_{\mathcal{P}_1}}}} \widehat{\mathcal{P}}} \leq \mathcal{I}(X;Y_{\mathcal{P}} | Y_{\mathcal{P}_1}).
\end{equation}
In \cref{eqn:max_incremental_gain_lowerbound}, $\mathcal{P}^c_1$ is the complement of $\mathcal{P}_1$ with $\mathscr{I}_{\mathcal{P}^c_1} = \mathscr{V} \setminus \mathscr{I}_{\mathcal{P}_1}$. The selection operator $\widehat{\mathcal{P}}$ in \cref{eqn:max_incremental_gain_lowerbound} has the same index set as $\mathcal{P}$ and is simply its counterpart when selecting from the reduced dimension random variable $Y_{\mathcal{P}^c_1} \in \mathbb{R}^{m-k_1}$. The approximation guarantee in this case follows from \Cref{thm:greedy_ascent}.   Bounding the parameter $\tau_i$ (see \Cref{thm:greedy_ascent}) that encapsulates the effectiveness of the modular lower bound is a non-trivial exercise which we will not pursue. Our empirical results (see \Cref{sec:num_results}), however, demonstrate that the performance of the algorithm is comparable to that of standard batch greedy (\Cref{alg:batch_std_greedy}) and in some instances slightly better.

\subsubsection{MM algorithms for minimizing information loss} \label{subsec:minimize_info_loss}

Our motivation to independently study minimizing information loss is triggered by a combination of factors. The asymmetry that exists between incrementally choosing good observations---as opposed to discarding bad observations---is perhaps the most intriguing reason.
In the context of linear Bayesian inverse problems, \Cref{corr:MI_SM_case_bounding_BP} encapsulates our motivation  perfectly. Observe how the same modular bound proves useful in selecting the best observation and discarding the worst. Yet each procedure carried out sequentially, by updating the modular bound, yields different answers in general. While \Cref{corr:MI_SM_case_bounding_BP} is restricted to the submodular case with conditionally independent observations, we have the generalization in \Cref{corr:MI_deviation_from_SM} with slightly different modular bounds.

\cref{eqn:exp_design_prob_statement_min} is the formal representation of the problem of minimizing information loss. 
Exploiting the previously discussed modular bounds for mutual information, we propose the following approach for its solution.
Suppose our goal is to discard $k<m$ observations, and we have already removed $k_1< k$. Let the  indices of the discarded $k_1$ observations correspond to the selection operator $\mathcal{P}_1$,  $\mathscr{I}_{\mathcal{P}_1} \subset \mathscr{V}$. We now seek $k_2  \coloneqq  k-k_1$ observations whose corresponding index selection operator $\mathcal{P}_2$,  with $\mathscr{I}_{\mathcal{P}_2} \subset \mathscr{V} \setminus \mathscr{I}_{\mathcal{P}_1}$, is the solution of the following problem:
\begin{equation} \label{eqn:min_info_loss_lbip}
\mathcal{P}_2 = \argmin_{\mathcal{P}, \abs{\mathscr{I}_{\mathcal{P}}}=k_2} \mathcal{I}(X;Y) - \mathcal{I}(X;Y \setminus Y_{\mathcal{P}_1}, Y_{\mathcal{P}}) = \argmin_{\mathcal{P}, \abs{\mathscr{I}_{\mathcal{P}}}=k_2} \mathcal{I}(X;Y|Y_{\mathcal{P}_1}) - \mathcal{I}(X;Y \setminus Y_{\mathcal{P}_1}, Y_{\mathcal{P}}).
\end{equation}

We solve \cref{eqn:min_info_loss_lbip} using \Cref{alg:min_info_loss_lbip} by minimizing the majorizing modular upper bound, as given in \Cref{prop:min_MI_loss_upper_bound}. 
We refer to this algorithm as \texttt{MMReverseGreedy}, alluding to the fact that we remove indices sequentially from the candidate set until we reach the desired cardinality.
While it is unclear if this approach can be supplemented with an  approximation guarantee, numerical  results indicate excellent empirical performance.
\begin{algorithm}
\begin{algorithmic}[1]
\caption{Sequential greedy algorithm for minimizing information loss}  \label{alg:min_info_loss_lbip}
\State {Input  $G, \Gamma_X, \Gamma_{Y|X},\mathscr{V},l,\{q_1,\ldots,q_l\}$}
\State {Initialize $\mathscr{I}(\mathcal{P}_{r})=\emptyset$}
 \For {$i = 1 \text{ to }  l$}
  \State {Determine $\mathcal{M}^{\downarrow}\left[ \mathcal{I}(X;Y|Y_{\mathcal{P}_{r}}) - \mathcal{I}(X;Y \setminus Y_{\mathcal{P}_{r}}, Y_{\mathcal{P}})  \right]$ as per \Cref{prop:min_MI_loss_upper_bound}.}
  \State {Find $\mathcal{P}$ with $\mathscr{I}_{\mathcal{P}} \subseteq \mathscr{V} \setminus \mathscr{I}_{\mathcal{P}_{r}}$ and $\abs{\mathscr{I}_{\mathcal{P}}} = q_i$ that minimizes $\mathcal{M}^{\downarrow}\left[ \mathcal{I}(X;Y|Y_{\mathcal{P}_{r}}) - \mathcal{I}(X;Y \setminus Y_{\mathcal{P}_{r}}, Y_{\mathcal{P}})  \right]$}
  \State {$\mathscr{I}_{\mathcal{P}_r} \leftarrow \mathscr{I}_{\mathcal{P}_r} \cup \mathscr{I}_{\mathcal{P}}$}
\EndFor

\Return {Index set $\mathscr{I}_{\mathcal{P}_r}$}
\end{algorithmic}
\end{algorithm}

\begin{proposition} \label{prop:min_MI_loss_upper_bound}
For random variables  $X \in \mathbb{R}^{n}$ and $Y \in \mathbb{R}^{m}$ as defined in \Cref{prop:MBound_MI}, and selection operators $\mathcal{P} \in \mathbb{R}^{m \times k}$ and $\mathcal{P}_1 \in \mathbb{R}^{m \times k_1}$ such that $k_1 < k$, the following bound holds:
\begin{align*}
 \mathcal{I}(X;Y|Y_{\mathcal{P}_1}) - \mathcal{I}(X;Y \setminus Y_{\mathcal{P}_1}, Y_{\mathcal{P}}) 
 & \leq \tr{ \widehat{\mathcal{P}}^{\top} \lr{ \log \lr{\Gamma_{Y|Y_{\mathcal{P}_1}}} - \log \lr{\Gamma_{Y|X,Y_{\mathcal{P}_1}}} }\widehat{\mathcal{P}}}  \\
& \quad +  \tr{ \widehat{\mathcal{P}}^{c}{}^{\top} \lr{ \log \mathrm{diag} \lr{\Gamma_{Y|X,Y_{\mathcal{P}_1}}} - \log \lr{\Gamma_{Y|X,Y_{\mathcal{P}_1}}} }\widehat{\mathcal{P}}^{c}}  \\ 
& \quad +  \tr{ \widehat{\mathcal{P}}^{c}{}^{\top} \lr{ \log \mathrm{diag} \lr{\Gamma_{Y \setminus Y_{\mathcal{P}_1}|X}} - \log \lr{\Gamma_{Y|X,Y_{\mathcal{P}_1}}} }\widehat{\mathcal{P}}^{c}} .
\end{align*}
Here $\widehat{\mathcal{P}}$ is the counterpart of $\mathcal{P}$ when selecting from the  reduced dimension random variable $Y_{\mathcal{P}^c_1} \in \mathbb{R}^{m-k_1}$.
\end{proposition}
\hyperref[proof:prop_min_MI_loss_upper_bound]{The proof is given in \Cref{appendix_subsec:lbip_mm}.}\medskip

The two trace terms involving $\widehat{\mathcal{P}}^{c}$ are non-negative and identically zero if the observations are conditionally independent. In that scenario, the proposition statement reduces to a slightly modified version of  \Cref{corr:MI_SM_case_bounding_BP}. 

\subsubsection{Numerical and computational issues concerning MM algorithms} \label{subsec:complexity_mod_bounds}

Our aim here is to provide some technical information on how to compute the modular bounds, and the associated computational costs. With regard to computational complexity, our focus here is not the number of function calls to the value oracle model, but complexity in terms of number of floating point operations (FLOPs). This is a more pertinent comparative metric since the optimization is performed using a  minorizing/majorizing surrogate, not the actual objective itself.
We will use the symbol $\mathscr{O}(\cdot)$ when referring to FLOP count, rather than $\mathcal{O}(\cdot)$, which we previously used to specify the number of function evaluations.

For the algorithms outlined in \Cref{subsec:maximize_info_gain,subsec:minimize_info_loss}, one critical task is estimating the diagonal of the matrix logarithm of a definite matrix.  
In each iteration, the matrix of interest is obtained using a Schur complement operation. More precisely, it is the Schur complement corresponding to one of the principal submatrices of the (larger) matrix in the previous iteration. The computational overhead of the Schur complement is moderate compared to that of determining the diagonal entries of the matrix logarithm. The cumulative cost of the latter task across all the iterations is what determines the complexity of the algorithm.

If the objective is submodular, and when the batch size is one, it is easy to see that we only require the relative ordering of the diagonal entries as opposed to  absolute accuracy of each. However there is no direct way to obtain relative ordering and, surprisingly, nor is there a direct efficient way to estimate the diagonal of a matrix logarithm. 
At first glance, it is tempting to consider estimating each entry of the diagonal as the bilinear form $e^{\top}_{i} \log (A) e_{i}$,  but such efforts \cite{Golub_Meurant_2009_book,Bai_Fahey_Golub_1996,Golub_Meurant_1993} are quickly realized to be inefficient in this case. 
Avoiding  computation of the matrix logarithm in all but special cases in not an option. We refer to the readers to \cite[Chap.~11]{Higham_2008_book} for a comprehensive survey on computing matrix logarithms.  

In our case, since we seek the logarithm of a definite matrix, we can use its eigenvalue decomposition. Recall that for  a definite matrix $A \in \mathbb{R}^{m \times m}$ with  eigendecomposition $A = U \Sigma U^{*}$, its matrix logarithm is simply $\log A = U \log \Sigma U^{*}$. Note how any favorable decay in the spectrum of the matrix $A$ may not necessarily aid the estimation of its matrix logarithm, since eigenmodes corresponding to the smallest eigenvalues could also dominate the estimate. In special cases, though, randomized methods that provide inexpensive low rank approximations of a matrix could still be useful.

In general, however, the leading order complexity for the algorithms in \Cref{subsec:maximize_info_gain,subsec:minimize_info_loss} is $\mathscr{O}(\frac{m^3  k}{q})$. Here $m$ is the size of the candidate set, $k$ is the desired cardinality, and $q$ is the batch size. The dominating cost is that of determining the sub/super--gradient using the modular bounds of \Cref{prop:MBound_MI}. Asymptotically, we do not improve on the complexity of batch greedy based on incremental gains (\Cref{alg:batch_std_greedy}); the dominating cost there is that of evaluating log determinants.


\subsubsection{Contextualizing with other approaches} \label{subsec:contextualizing_with_other_methods}

In various simplified settings, the sequential methods outlined in \Cref{subsec:maximize_info_gain,subsec:minimize_info_loss} for the linear Bayesian optimal experimental design problem share similarities to existing techniques. Consider the case of independent and identically distributed observation errors; in this case the objective in \cref{eqn:mutualInfoGaussians_using_logDet}
is submodular and defined entirely by the data marginal $\pi_Y$. \cref{eqn:exp_design_prob_statement_max} is now equivalent to finding the mode of a fixed-size determinantal point process (DPP) \cite{Kulesza_Taskar_2011,Kulesza_Taskar_2012}, with the kernel of the DPP being 
the marginal covariance of the data, $\Gamma_Y$.

In the MM approach of \Cref{subsec:maximize_info_gain}, we define a tight modular lower bound on the objective using the trace of the logarithm of the data marginal covariance, or its appropriate conditional counterpart. The individual diagonal entries of the matrix logarithm are akin to \emph{weighted} leverage scores associated with each row/column index; here the weights are simply the logarithms of the eigenvalues. Historically, the notion of leverage scores was introduced in the context of linear regression for outlier detection, and to assess the amount of influence exerted by one observation regardless of its actual value \cite{chatterjee_hadi_1986,Hoaglin_Welsch_1978}. In the recent past such scores have been useful in performing subset selection \cite{Boutsidis_etal_2009} and in linear regression \cite{Drineas_etal_2011}, both of which directly relate to optimal experimental design.

\textbf{Sampling based on leverage scores versus DPPs.} Volume sampling in DPPs and sampling based on traditional statistical leverage scores are clearly different. But in precisely what way?
Contrasting their differences and understanding the nuances is a useful exercise.
Consider a DPP defined through an $L$-ensemble specified by $\mathcal{L} \succeq 0$; often the semi-definite matrix $\mathcal{L}$ is also termed the kernel matrix.  One popular interpretation is obtained by expressing $\mathcal{L}$ as a  Gram matrix, $\mathcal{B}^{\top} \mathcal{B}$, where the columns of $\mathcal{B}$ are feature vectors representing items in the candidate set. Furthermore, if each column $B_i$ is written as the product of a quality term $q_i \in \mathbb{R}_{>0}$ and a vector of normalized diversity features $\phi_i \in \mathbb{R}^{d}$, $\norm{\phi_i} = 1$, then we have $\mathcal{L}_{ij} = q_i \phi_i^{\top} \phi_j q_j$. 
Now suppose that $\mathcal{U} \Lambda \mathcal{U}^{\top}$ is the eigendecomposition of $\mathcal{L}$. Then it is easy to see that $ q_i^2  = {\mathcal{U} \Lambda \mathcal{U}^{\top}}_{[i,i]}$.  The diagonal entries in ${\mathcal{U} \Lambda \mathcal{U}^{\top}}$ are strictly speaking not statistical leverage scores, since we are weighing the contribution of each eigenvector by its corresponding eigenvalue. But note that when using randomized numerical linear algebra methods, the practice is always to give more importance to higher eigenmodes. In the context of this comparative discussion, it is evident that sampling based on leverage scores does not account for diversity among the candidates, but accounts only for the quality of each one. 

\section{Numerical results} \label{sec:num_results}

In this section we evaluate the performance of the batch greedy algorithms when applied to linear Bayesian optimal experimental design.
Our focus here to is to demonstrate how batch size affects performance and to evaluate the impact of different modular bounds; hence we only consider the standard greedy algorithm and the MM algorithms. The performances of distributed and stochastic variants of the greedy heuristic, compared to the standard greedy paradigm, have been showcased in earlier works \cite{Mirzasoleiman_etal_2013,lazier_than_lazy_greedy}.
In the discussion and figures that follow, we label the approach of maximizing information gain
using \Cref{alg:batch_std_greedy} as \texttt{StdGreedy}; the approach of maximizing information gain using the minorizing surrogate (\Cref{alg:greedy_ascent} and \Cref{subsec:maximize_info_gain}) as \texttt{MMGreedy};
and analogously, the approach of  minimizing information loss using the majorizing surrogate (\Cref{alg:min_info_loss_lbip} and \Cref{subsec:minimize_info_loss}) as \texttt{MMReverseGreedy}.

\subsection{An inverse problem with structured random operators} \label{subsec:num_results_random}

We consider a Bayesian inference problem in the setting of \Cref{subsubsec:lbip_notation_setup}, with a randomly generated linear forward operator $G$. The dimension of the parameters $X$ is set to $n=20$, while cardinality of the candidate set of observations $Y$ is fixed at $m=100$.
Realizations of the forward operator are constructed by independently generating random singular values and left/right singular vectors. The spectrum of $G$, while random, has a prescribed exponential rate of decay reflective of many real world inverse problems \cite{Spantini_etal_2015}. The random left and right singular vectors are obtained by taking a modified Gram-Schmidt QR factorization of a random normal matrix of the appropriate dimension. We specify the prior and observation error covariances,  $\Gamma_X$  and $\Gamma_{Y|X}$, using squared exponential kernels with correlation lengths $0.105$ and $0.021$, respectively, on a unit domain. We draw $1000$ random instances of the forward operator $G$ and solve the experimental design problem of maximizing mutual information $\mathcal{I}(X,\mathcal{P}^{\top}Y)$ in each case. The spectra of the relevant operators are shown in \Cref{fig:spectrum_corr_obs_error}. If a spectrum corresponds to a random operator, we plot the median and indicate the spread.

\begin{figure}[!ht] 	
  	\centering \includegraphics[width=\textwidth,angle=0]{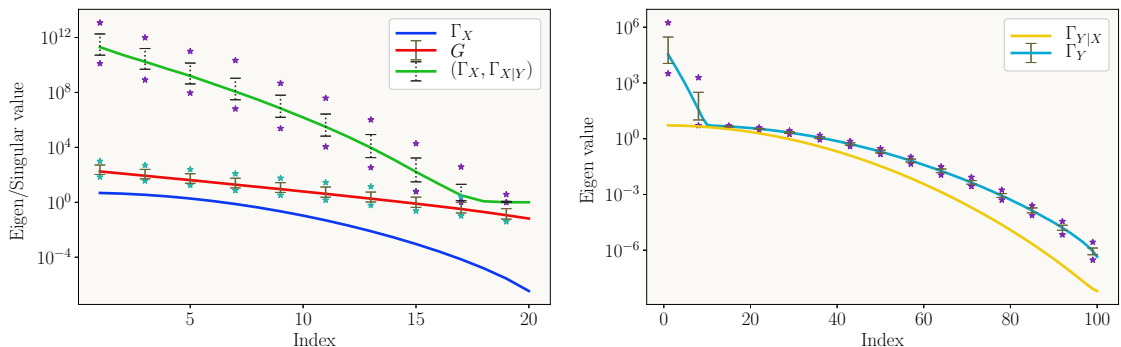}
  	 \caption{Spectrum of the relevant operators of the inverse problem with correlated observation error. The solid line is the median across 1000 random instances of the forward model. The whiskers capture the interquantile range (10\% to 90\%),  and the $\star$ symbols mark the maximum and minimum eigen/singular value. The prior and observation error covariances are not random.}
  	\label{fig:spectrum_corr_obs_error}
\end{figure}

In \Cref{fig:batch_effect_subplot_corr_obs_error} we illustrate the influence of batch size for each algorithm. On the horizontal axis we indicate the number of observations included, and on the vertical axis we plot the amount of mutual information captured at that cardinality relative to the maximum amount possible (i.e., using all the candidate observations $Y$). We consider seven different batch sizes, corresponding to $q \in  \{1\%, 10\%, 20\%, 30\%, 40\%, 50\%, 100\%\}$. 
(Here we have expressed the batch size $q$ as a fraction of the size of the candidate set.) For each algorithm and for each choice of batch size, \Cref{fig:batch_effect_subplot_corr_obs_error} shows the median performance across the random instances of the forward model. \Cref{fig:random_sampling_error_4x3subplot_corr_obs_error} illustrates the variability in performance across these random problem instances; here, in addition to the median mutual information, we show the 0.1/0.9 quantiles and range of mutual information obtained at each cardinality.
Kinks in the continuous lines of Figures~\ref{fig:batch_effect_subplot_corr_obs_error}--\ref{fig:random_sampling_error_4x3subplot_corr_obs_error} reflect the start of each new batch calculation.
Improved gains at smaller batch sizes are consistent with our theoretical claims in \Cref{thm:batch_std_greedy,thm:greedy_ascent}. 
The numerical results also corroborate our remark in \Cref{subsec:theoretical_guarantees_LBIP_std_greedy}, that the trivial approximation guarantee for $m > n$ is merely an artifact of loose bounds on the sub/super-modularity ratios. Interestingly, while we do not have a theoretical analysis of the \texttt{MMReverseGreedy} algorithm, the effect of batch size on its performance is similar to that of the other two heuristics.

In \Cref{fig:batch_effect_subplot_random_selection}, we compare the performance of the greedy algorithms, for a range of batch sizes, to that of a \emph{random} selection of indices. To obtain the latter results, we draw 1000 random index selections for each random forward model. In general, the greedy algorithms vastly outperform random selection. The median performance of even the one-shot ($q=100\%$) greedy approaches is better than the median performance of random selection. This difference is greatest for the \texttt{MMGreedy} and \texttt{MMReverseGreedy} algorithms, which show significantly better performance at large batch sizes than \texttt{StdGreedy}. As shown in the right two panels of \Cref{fig:batch_effect_subplot_random_selection}, the median performance of these two greedy heuristics is better than even the 90\% quantile of performance of random selection.

\begin{figure}[!ht] 	
  	\centering \includegraphics[width=\textwidth,angle=0]{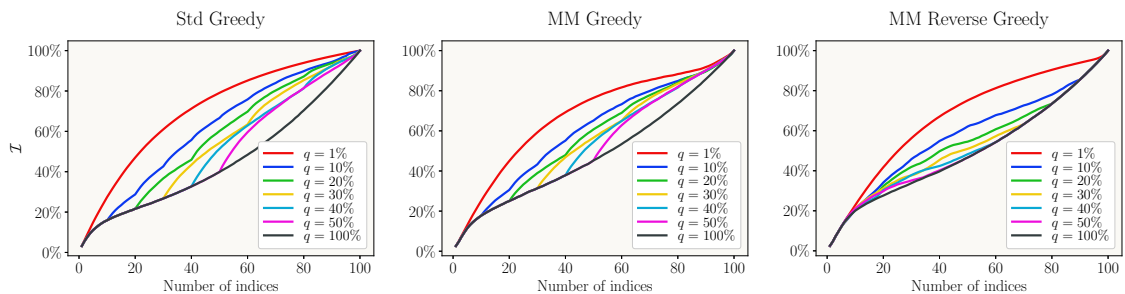}
	\caption{Performance of each greedy heuristic for different batch sizes, ranging from single index selection ($q=1\%$) to a one-shot approach ($q=100\%$). The solid line is the median across 1000 random instances of the forward model.}
  	\label{fig:batch_effect_subplot_corr_obs_error}
\end{figure}

\begin{figure}[!ht] 	
  	\centering \includegraphics[width=\textwidth,angle=0]{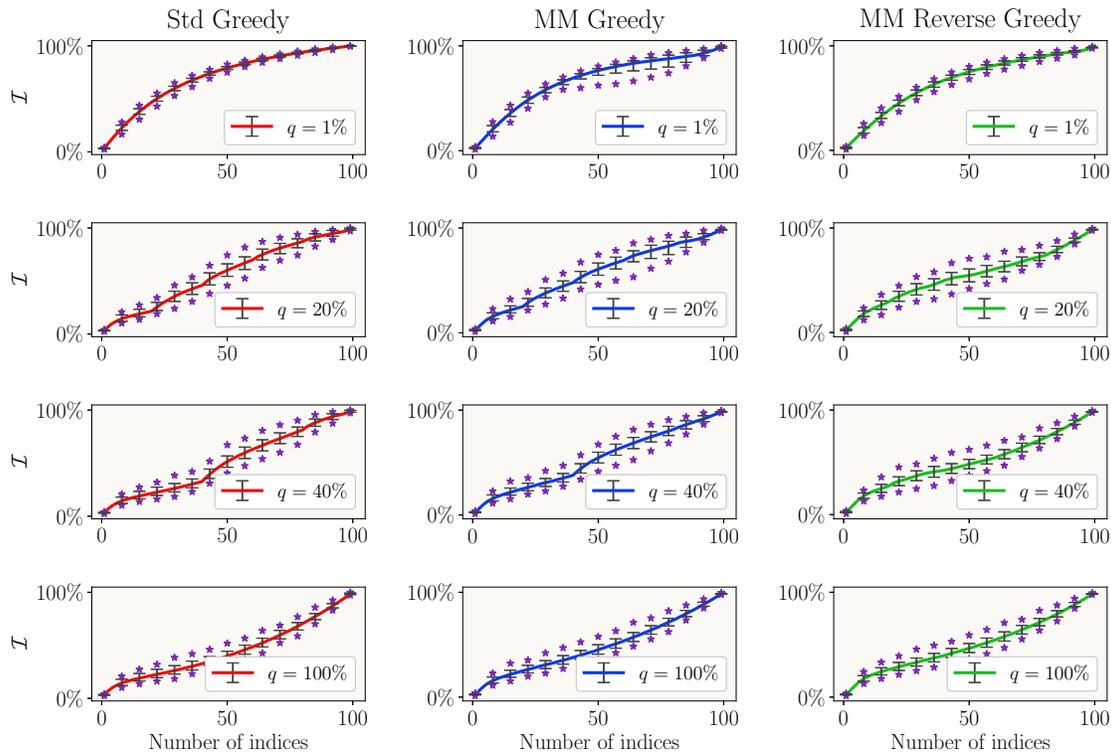}
	\caption{Performance of the greedy heuristics for different batch sizes across all 1000 random instances of the forward model. The solid line is the median; the whiskers bound the 10\% to 90\% interquantile range,  and the $\star$ marks the maximum and minimum mutual information captured.}
  	\label{fig:random_sampling_error_4x3subplot_corr_obs_error}
\end{figure}

\begin{figure}[!ht] 	
  	\centering \includegraphics[width=\textwidth,angle=0]{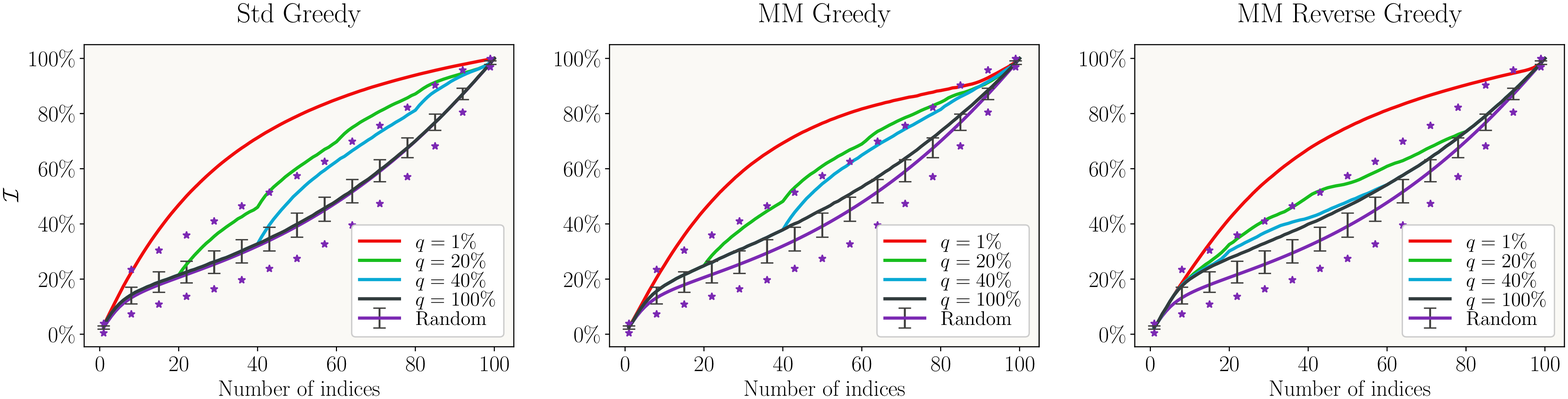}
  	 \caption{Performance of each greedy heuristic, for different batch sizes $q$, compared to a random selection of indices. The solid line corresponding to a given batch size is the median across 1000 random instances of the forward model.  The solid line labeled ``random'' is the median performance across $10^6$ cases: 1000 random instances of the forward model tensorized with 1000 random selections of indices. For the random selection results, the whiskers bound the 10\% to 90\% interquantile range and the  $\star$ symbols mark the maximum and minimum mutual information captured. Note that the random selection results are the same in each of the three panels above; only the batch greedy algorithms are different.}
  	\label{fig:batch_effect_subplot_random_selection}
\end{figure}

Fixing the batch size, we compare the three greedy algorithms' performance more directly in \Cref{fig:compare_algorithms_4subplot_corr_obs_error}. We consider the cases  $q \in  \{1\%,20\%,40\%,100\%\}$. Once again, we only plot the median performance across 1000 random instances of the forward model to retain visual clarity. When choosing indices one at a time ($q = 1\%$)  the standard greedy heuristic does marginally better than both MM greedy heuristics, but at larger batch sizes, both MM approaches provide better gains. This distinction is more readily apparent in \Cref{fig:diff_MI_random_sampling_error_4x2subplot_corr_obs_error}, where we plot the difference between the relative amount of information captured by \texttt{MMGreedy} and \texttt{MMReverseGreedy} in comparison to the standard greedy heuristic. Here we also indicate the spread due to the forward model being random. The differences in the performances of the three heuristics are stark at relatively lower cardinalities, and diminish at higher cardinality since information saturates. The upside to using the \texttt{MMGreedy} heuristic in comparison to the standard greedy heuristic is much greater than the downside, as indicated by the interquantile range and the maximum/minimum mutual information captured.

\begin{figure}[!ht] 	
  	\centering \includegraphics[width=\textwidth,angle=0]{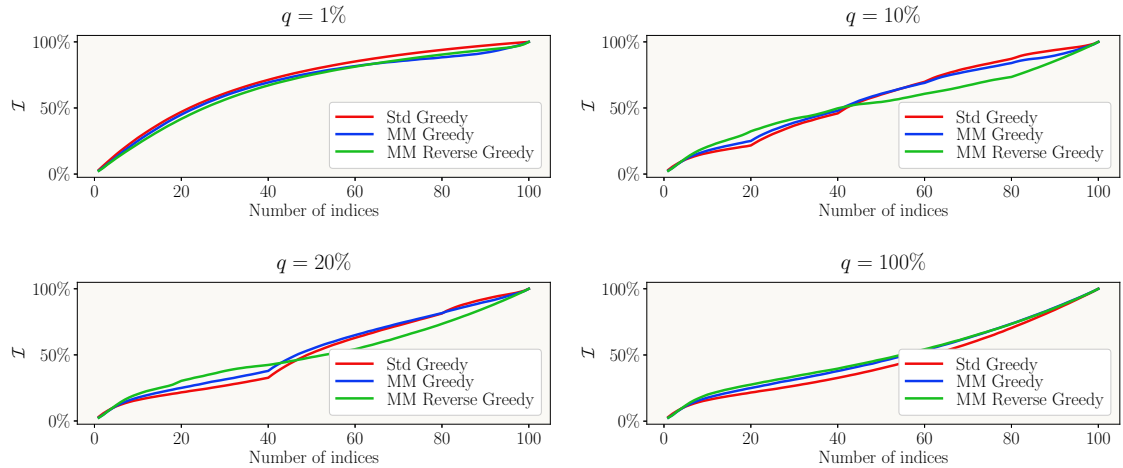}
	\caption{A comparative study of the greedy heuristics for four different batch sizes. The solid line is the median across 1000 random instances of the forward model.}
  	\label{fig:compare_algorithms_4subplot_corr_obs_error}
  	\end{figure}

\begin{figure}[!ht] 	
  	\centering \includegraphics[width=\textwidth,angle=0]{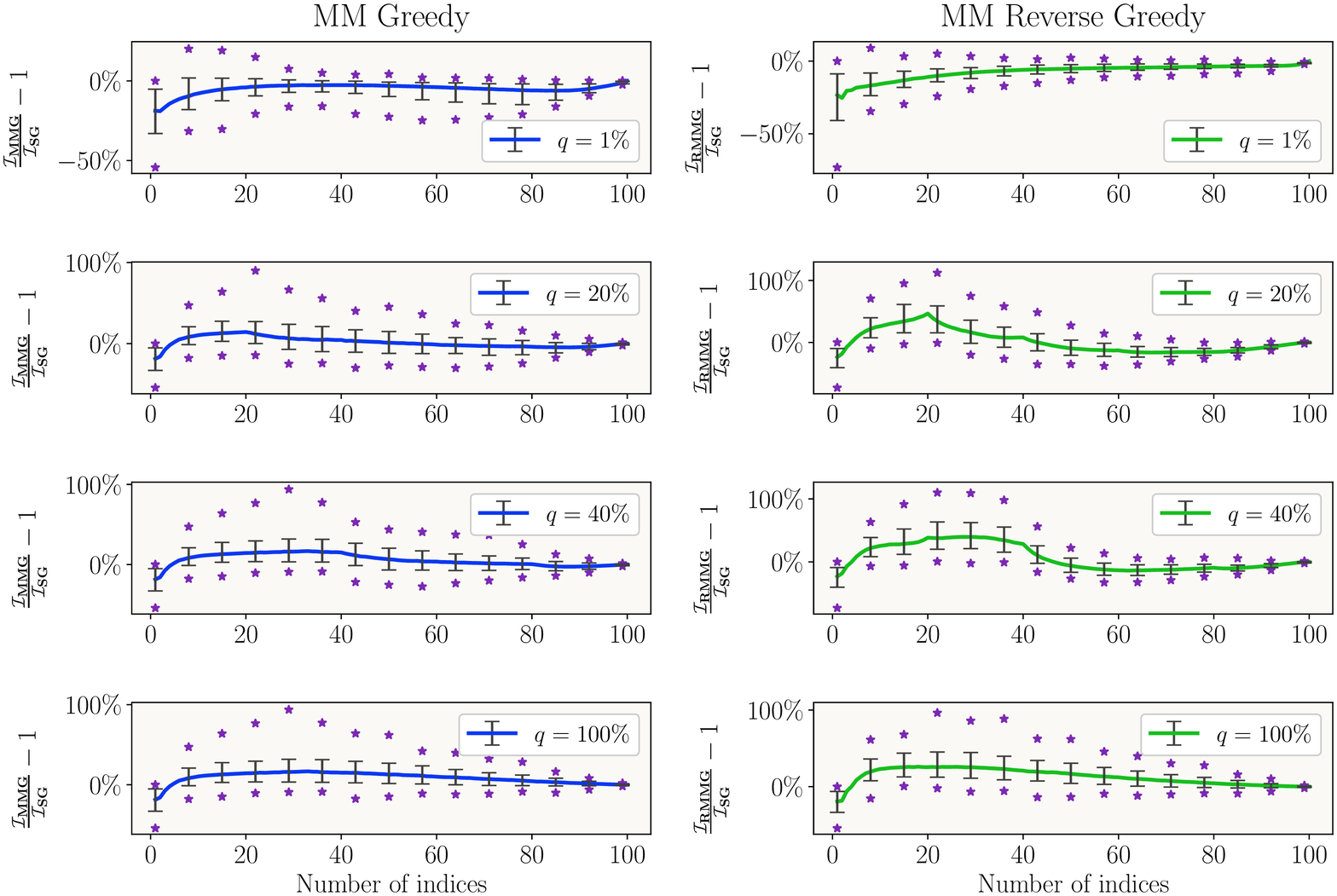}
  	 \caption{MM-based batch greedy approaches compared against the standard batch greedy heuristic for different batch sizes, across all 1000 random instances of the forward model. The solid line is the median; the whiskers capture the 10\% to 90\% interquantile range,  and the $\star$ marks the maximum and minimum mutual information captured.}
  	\label{fig:diff_MI_random_sampling_error_4x2subplot_corr_obs_error}
\end{figure}


The results we have discussed thus far are for the case of correlated but structured observation error covariance. Smaller batch sizes clearly yield more information gain, and should be preferred unless computational demands dictate otherwise. But many realistic problems have error terms with a less structured correlation (see \Cref{subsec:num_results_oscm}), or independent and identically distributed (i.i.d.)~observation errors. In the latter case, the experimental design objective is also submodular.  
In these scenarios, the use of a smaller batch size may be unwarranted, as our numerical experiments in \Cref{subsec:num_results_oscm} 
and \Cref{appendix:num_results_iid_error} indicate.
More generally, if the observation error covariance is not strongly correlated, meaning the observations are nearly conditionally independent, the advantage of a smaller batch size is diminished.

We have compiled in \Cref{appendix:num_results_exp_prior} numerical results for the case when the prior covariance is defined using a exponential kernel, while the observation error covariance is defined using a squared exponential kernel as before. Now there is significantly less prior correlation among the parameters. The results indicate, however, that the choice of prior covariance kernel is not particularly important; the performance of each heuristic is similar to this section, where we used a squared exponential kernel to define the prior covariance.

We have also numerically investigated the performance of the algorithms when the dimension of the parameters is greater than the cardinality of the candidate set of observations, i.e., $n > m$; these results are reported in \Cref{appendix:num_results_param_greater_than_obs}. Recall that in such cases the results in \Cref{prop:bound_sup_sub_modularity_ratio} provide non-trivial bounds on the sub/super--modularity ratios. The \texttt{StdGreedy} and \texttt{MMGreedy} heuristics perform similarly to the case discussed in this section. The performance of the \texttt{MMReverseGreedy} heuristic, however, appears insensitive to the batch size.

\subsection{Optimal sensor placement to improve climate models}  \label{subsec:num_results_oscm}

We now consider a problem of optimizing sensor networks for climate models (which we refer to as the ``SNCM'' problem). Given a desired cardinality, our goal is to select field observation sites that will yield the most informative  data for parameter inference. In the present application, the parameters are uncertain inputs to the land-surface component of a climate model, while the field data correspond to certain observable outputs of the same model. Our example is based on the land-surface component of the Energy Exascale Earth System Model (E3SM)  \cite{E3SM}. The latter is an ongoing effort, led by the US Department of Energy, comprising multiple model components, each with its own set of uncertain input parameters. These individual components can be coupled, and together they simulate the earth's atmosphere, ocean, land surface, and sea ice \cite{hurrell2013community}. Uncertainty in any one component can contribute to a large spread in the overall model predictions. This affects our understanding of severe climate events, their timing, and our ability to cope with the consequences. 

The \emph{simplified} E3SM land model \cite{sELM}, henceforth referred to as sELM, is a ``land model'' derived from E3SM that simulates carbon cycle processes relevant to the earth system in an efficient way. This allows for large regional ensemble simulations that would otherwise be infeasible using the complete land model.
The simulation region of our focus is the eastern part of the north American continent located between the latitudes \ang{28.25}N--\ang{48.25}N and longitudes \ang{66.25}W--\ang{96.25}W. Please see \Cref{fig:monthly_gpp_mean,fig:gpp_variance} for a depiction of the region. The simulation resolution is \ang{.5} in each direction, which corresponds to a grid of $41 \times 61$ points. Only 1642 of those grid points fall within the land area, however, and at those locations we have access to the sELM outputs. In the version of sELM we consider, there are 47 input parameters; these parameters have no spatial variability and have the same prior distribution at every location. Drawing realizations of these parameters yields a simulation ensemble with 2000 samples.  The code for sELM is publicly available \cite{sELM}, and more details about the E3SM land models can be found in  \cite{lu2019efficient,ricciuto2018impact}.

To set up the optimal experimental design problem,  we focus exclusively on one output of the sELM, the gross primary production (GPP). GPP can be understood to be a proxy for the amount of carbon flux attributable to the natural vegetation at that location. GPP is a function of the sELM input parameters and relevant meteorological quantities such as temperature. In the version of sELM we consider, the GPP is output as monthly averages for thirty years starting from the year 1980. In \Cref{fig:monthly_gpp_mean} we plot the monthly GPP averaged across the 2000 samples of the parameter ensemble and across thirty years of output history. The trends in the plot reflect expected seasonal variations, with more activity in the tropical southern regions. If we treat the GPP output at any grid point as a random variable, then its variance is affected by the uncertainty of the sELM input parameters and the temporal variation of  meteorological quantities. In \Cref{fig:gpp_variance} we plot the variance of the GPP as output from the sELM. 
 
\begin{figure}[!ht] 	
  	\centering \includegraphics[width=\textwidth,angle=0]{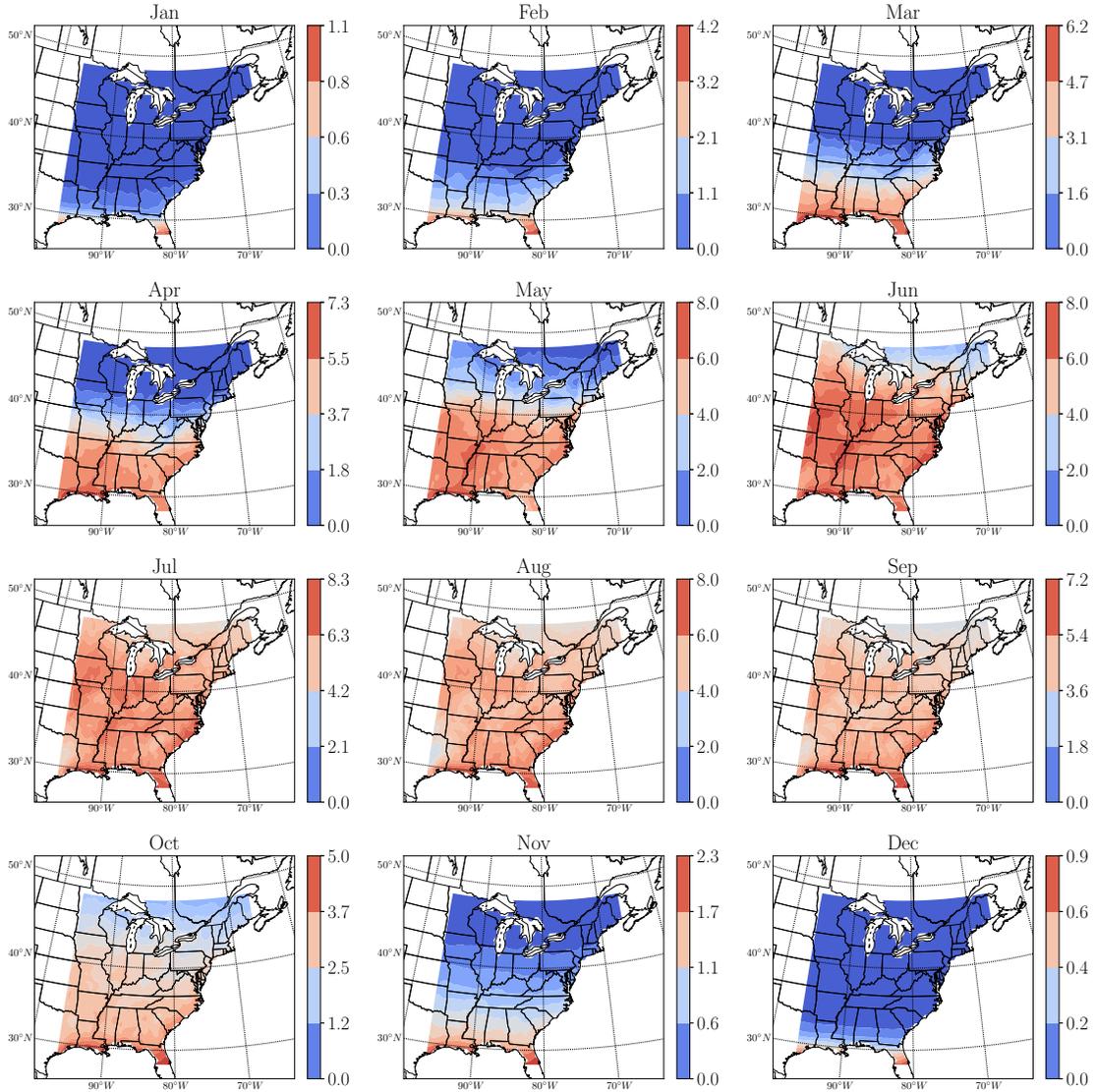}
  	 \caption{Mean GPP for each month averaged across both the parameter ensemble and temporally.}
  	\label{fig:monthly_gpp_mean}
\end{figure}

\begin{figure}[!ht]
    \begin{minipage}[t]{.62\textwidth}
              \centering \includegraphics[width=\textwidth,angle=0]{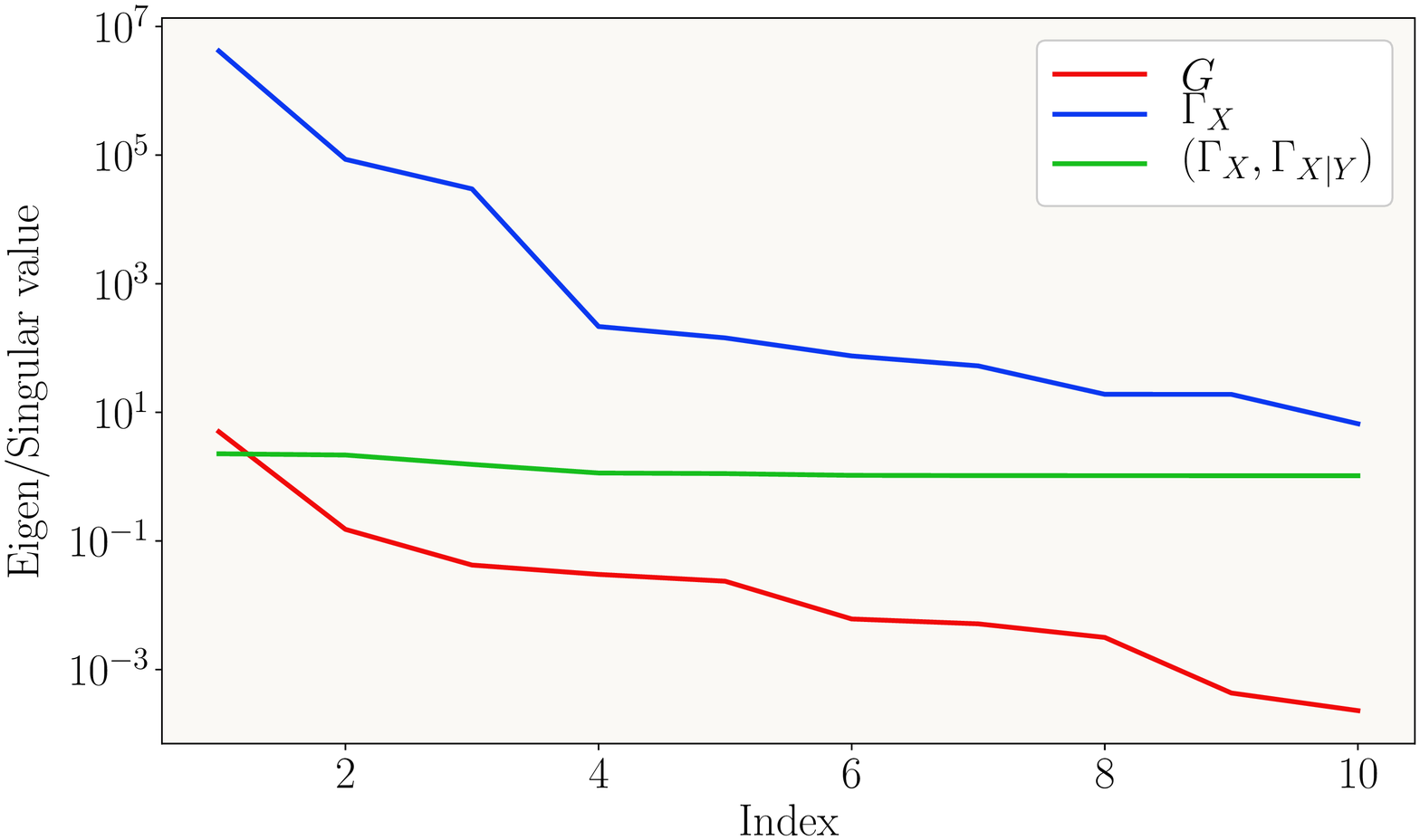}
              \subcaption{Spectrum of the linearized forward model, Gaussian prior, and the matrix pencil comprising of prior and full posterior.}
              \label{fig:spectrum_using_emp_cov}
    \end{minipage}
    \hspace{0.05\linewidth}
    \begin{minipage}[t]{.4\textwidth}
             \centering \includegraphics[width=\textwidth,angle=0]{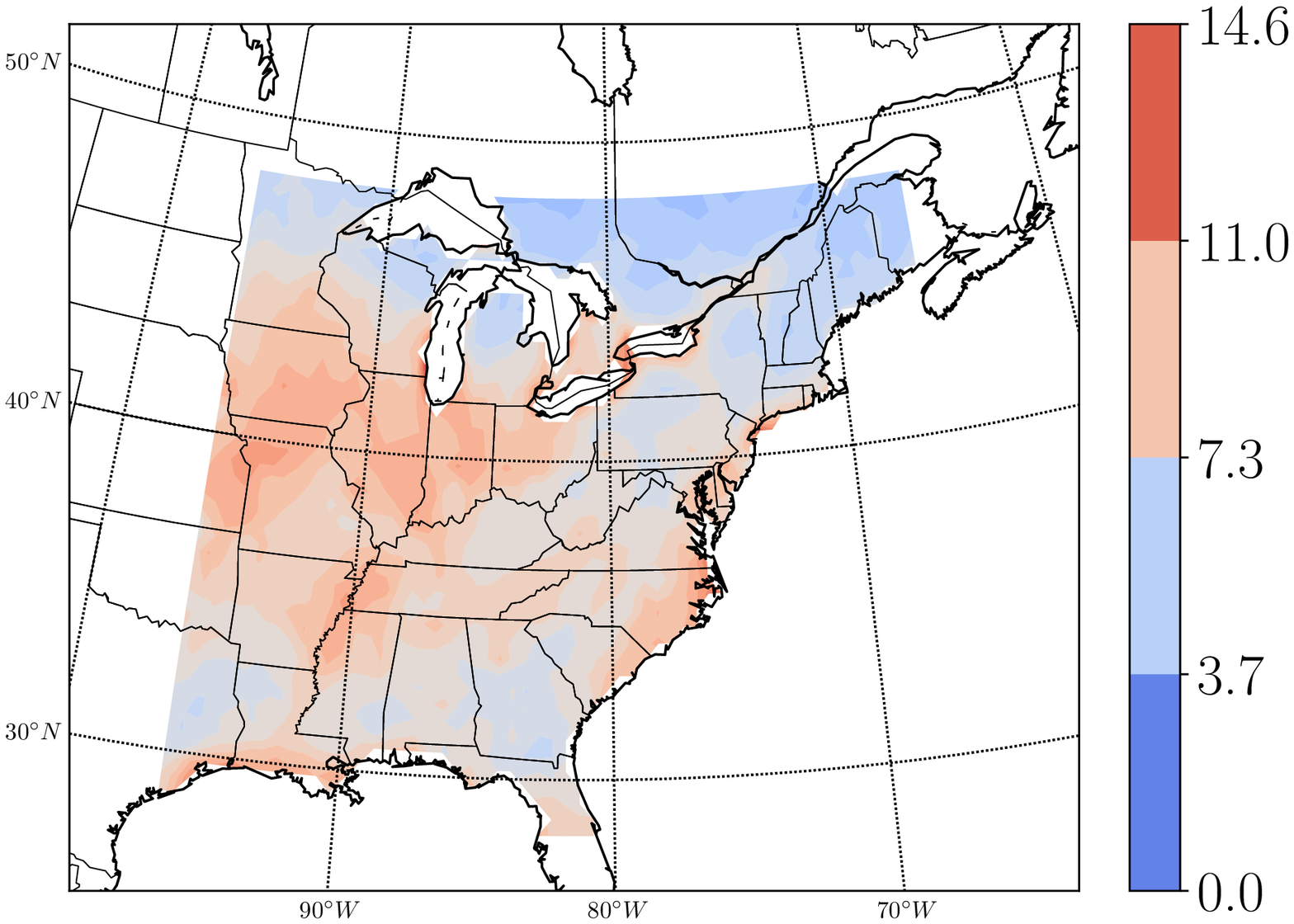}
			\subcaption{Variance of the GPP marginal visualized on the United States map.}
		  	\label{fig:gpp_variance}
    \end{minipage}  
    \label{fig:spectrum_and_variance}
    \caption{Linear operators of the SNCM problem.}
\end{figure}

%


Isolating the contribution of the model parameter uncertainty to the variance of the GPP output can be accomplished in a number of ways. We could perform a multivariate regression with covariance estimation that accounts for spatial correlation among the GPP output variables \cite{rothman2010sparse}, or alternatively use factor model approaches that are popular in econometrics \cite{fan2008high,fan2011high}. However we adopt a more straightforward technique, since our primary goal is to set up a design problem suitable for a comparative study of the batch greedy algorithms we have proposed. We simply estimate the linear relationship between the GPP output, $Y$, and sELM model parameters, $X$, using an empirical estimate of the cross-covariance $\Gamma_{Y,X}$. Incorporating this result into the setting of \Cref{subsubsec:lbip_notation_setup} assumes that the parameters have a normal prior and that the error term is independent of the parameters. In \Cref{fig:spectrum_using_emp_cov} we plot the spectrum of the linear forward operator $G$ obtained in such a manner, along with the prior covariance $\Gamma_{X}$. Observe that we have reduced the dimension of the parameters to $10$ by retaining only those that have a prior variance larger than $\mathcal{O}(1)$. The number of candidate observations, corresponding to the dimension of $Y$, is $m=1642$. The generalized eigenvalues of the matrix pencil $(\Gamma_{X},\Gamma_{X|Y})$, also shown in \Cref{fig:spectrum_using_emp_cov}, suggest that the data are only marginally informative about the parameters.

Using the derived operators, we study the performance of the previously proposed batch greedy algorithms. From \Cref{fig:batch_effect_subplot_oscm_using_emp_cov}, it is evident that decreasing the batch size does not reward us with any significant gains; furthermore, all the greedy heuristics have similar performance as indicated in \Cref{fig:compare_algorithms_4subplot_oscm_using_emp_cov}.
The difference between the relative amount of information captured by the MM greedy heuristics in comparison to the standard greedy heuristic is shown in \Cref{fig:diff_MI_4x2subplot_using_emp_cov}. The MM greedy heuristics have better gains at lower cardinality numbers (except for extremely low cardinality), but these differences return to zero once information saturates.
In \Cref{fig:batch_effect_subplot_oscm_using_emp_cov,fig:compare_algorithms_4subplot_oscm_using_emp_cov,fig:diff_MI_4x2subplot_using_emp_cov}, we have not shown the case of $q=1\%$ since it requires substantial computing time, but given the trends we expect its performance will not be any better than that of batch size $q=10\%$.


\begin{figure}[!ht] 	
  	\centering \includegraphics[width=\textwidth,angle=0]{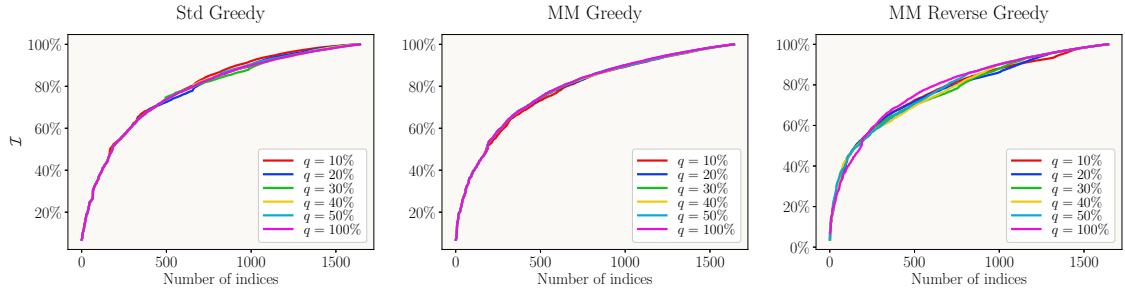}
	\caption{Performance of each greedy heuristic for different batch sizes for the SNCM problem. The batch sizes range from $q=10 \%$ to the one-shot approach with $q=100 \%$.}
  	\label{fig:batch_effect_subplot_oscm_using_emp_cov}
\end{figure}

\begin{figure}[!ht] 	
  	\centering \includegraphics[width=\textwidth,angle=0]{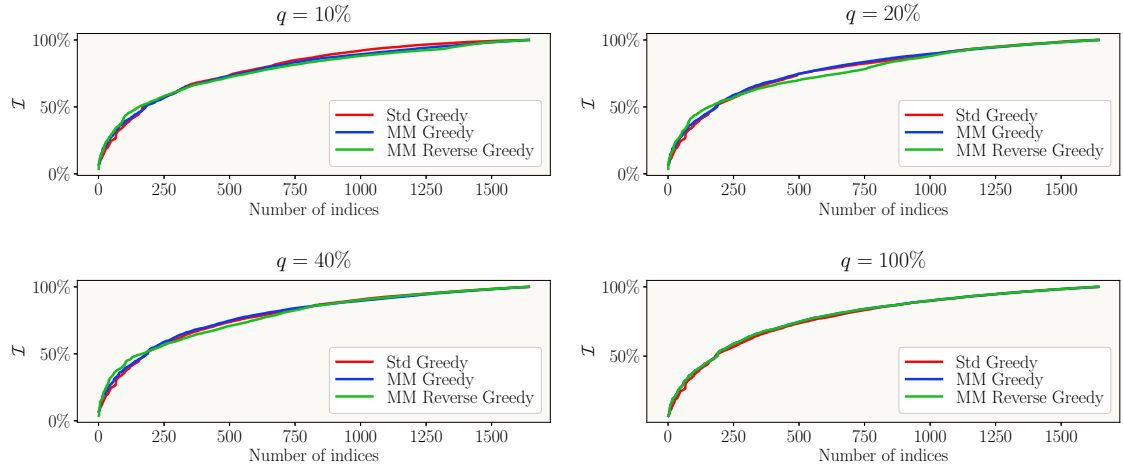}
	\caption{A comparative study of the greedy heuristics for different batch sizes for the SNCM problem.}
	  	\label{fig:compare_algorithms_4subplot_oscm_using_emp_cov}
\end{figure}

\begin{figure}[!ht] 	
  	\centering \includegraphics[width=\textwidth,angle=0]{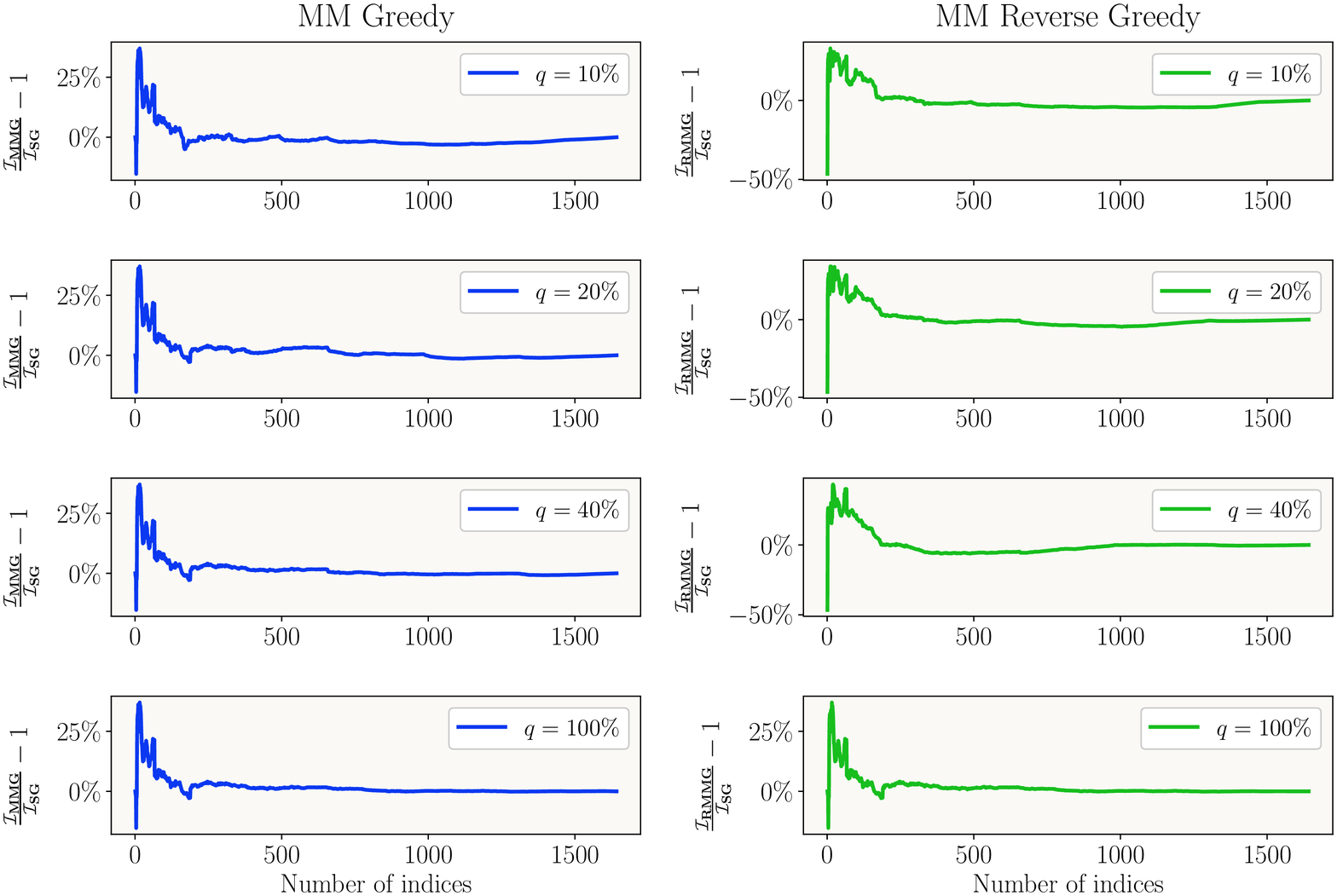}
  	 \caption{MM-based batch greedy approaches compared against the standard batch greedy heuristic for different batch sizes for the SNCM problem.}  	
  	   	\label{fig:diff_MI_4x2subplot_using_emp_cov}
\end{figure}

In \Cref{fig:new_loc_exp_design} we have marked the first 10 locations given by each of the heuristics for the case when the batch size is $10\%$ of the total cardinality. The fact that most of these locations are close to the coastal boundaries---and predominantly the southern coast---is intriguing. This phenomenon is simply a result of the sELM data. To comprehend it better, we have visualized in \Cref{fig:oscm_one_shot_ranking_visualized} the rankings of locations at the start of the standard greedy and MM greedy heuristics (i.e., the ``one-shot'' rankings). These rankings follow from either the incremental gain associated with each location (\texttt{StdGreedy}) or the initial evaluation of the modular lower bound (\texttt{MMGreedy}). While each heuristic provides a different set of locations as the solution of the design problem, the collective information gained from them may not differ significantly. Such a behavior is not unique to the SNCM problem, and is common in many scenarios where combinatorial choices have to be made.

\begin{figure}[!ht] 	
  	\centering \includegraphics[width=\textwidth,angle=0]{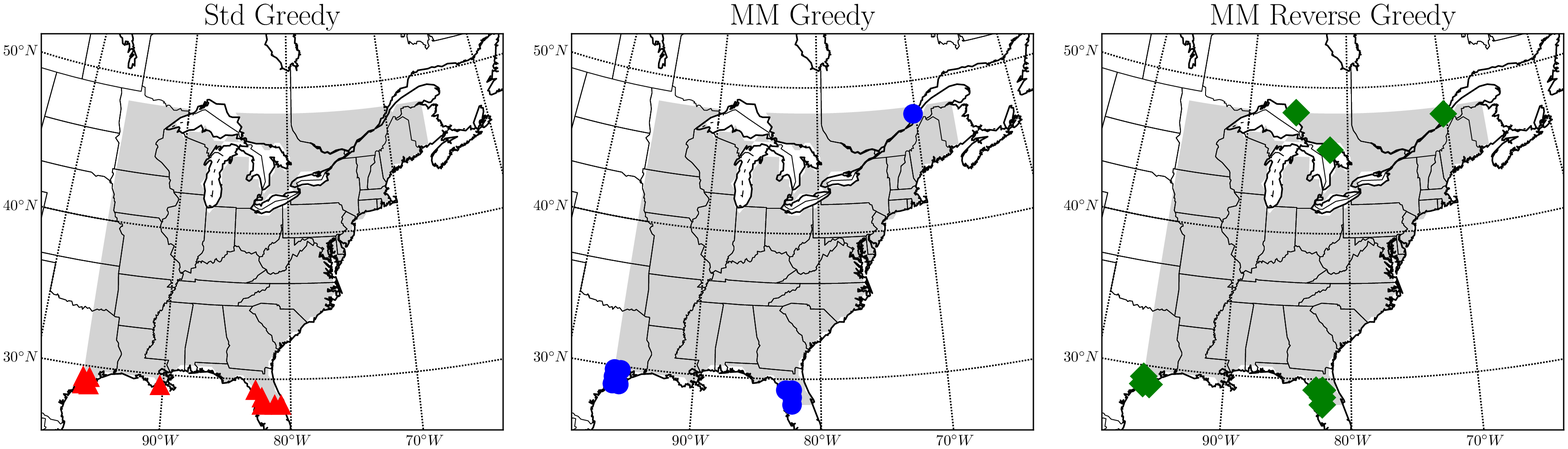}
	\caption{The first ten locations obtained by each greedy heuristic when the batch size $q$ is set to $10\%$ of the full cardinality.}  	
	  	\label{fig:new_loc_exp_design}
\end{figure}

\begin{figure}[!ht] 	
  	\centering \includegraphics[width=\textwidth,angle=0]{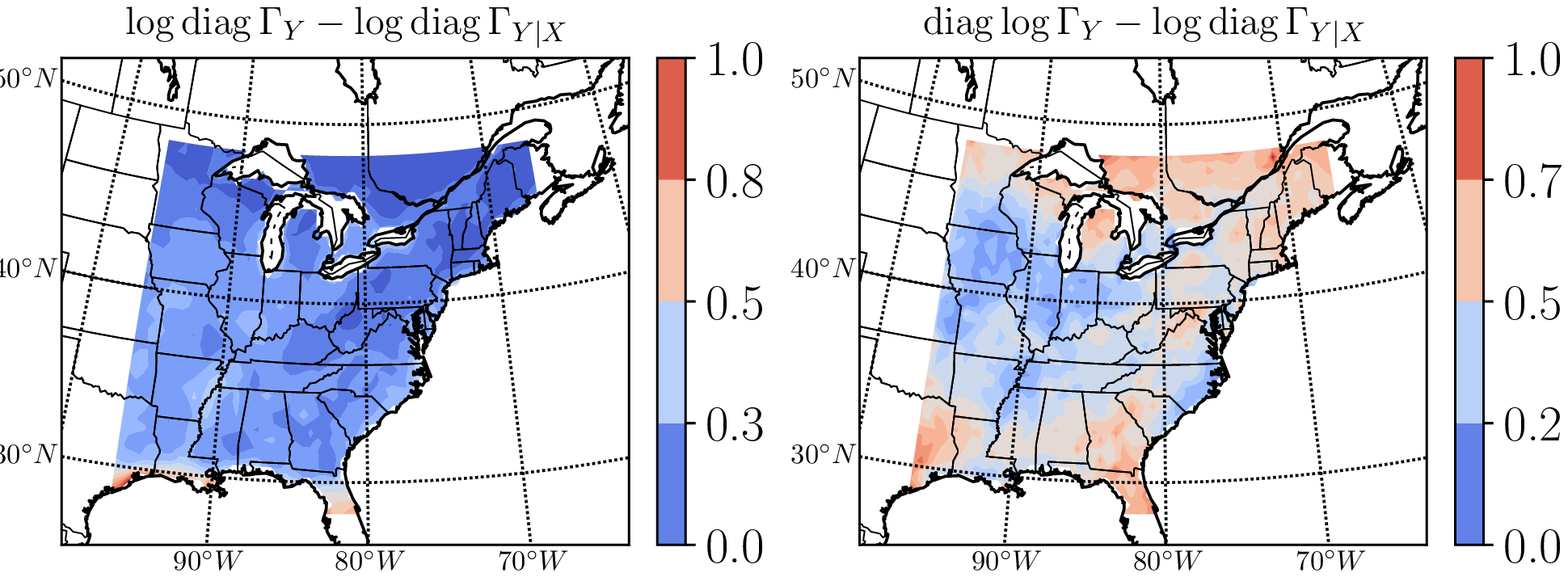}
	\caption{ (\lowerRomanNum{1})  Incremental gains of each candidate location,
	$\log \mathrm{diag} \, \Gamma_Y - \log \mathrm{diag} \, \Gamma_{Y|X}$, scaled to the interval $[0,1]$. (\lowerRomanNum{2}) The modular lower bound associated with the one-shot MM greedy approach, $\mathrm{diag} \log \, \Gamma_Y - \log \mathrm{diag} \, \Gamma_{Y|X}$, scaled to the interval $[0,1]$.}
	  	\label{fig:oscm_one_shot_ranking_visualized}
\end{figure}
\section{Discussion} \label{sec:conclusions}

This paper has investigated batch greedy heuristics for maximizing monotone non-submodular functions under cardinality constraints. 
We analyzed batch versions of the standard greedy paradigm, and of its distributed and stochastic variants.
Our theoretical guarantees for the resulting batch greedy algorithms are characterized by a combination of submodularity and supermodularity ratios. In the context of linear Bayesian optimal experimental design, we bound these parameters for the mutual information design criterion. Reinterpreting the classical greedy heuristic in the language of MM algorithms, we also argue how any good modular bound---not necessarily one based on incremental gains---can instantiate a related greedy technique. Based on those insights, we propose several novel modular bounds and algorithms for optimizing information theoretic design criteria in the context of  Bayesian experimental design.
Now we discuss some further context and potential extensions. 

Our result as expressed in \cref{eqn:batch_std_greedy_approx_guarantee} does not incorporate curvature of the function. The classical notion of curvature  (\cite{Conforti_Cornuejols_1984} and \Cref{def:total_curvature}) measures how close a submodular set function is to being modular, while the notion of generalized curvature (\cite{Bian_etal_2017} and \Cref{def:gen_curvature}) measures how close a set function is to being supermodular. In \cite{Bian_etal_2017},
the submodularity ratio 
and the generalized curvature together quantify how close a set function is to being modular. These parameters provide  approximation guarantees that refine the worst case bounds depending on the instance of the function. In our case, the product of supermodularity and submodularity ratios characterizes the modularity of a function, but as it
appears in \Cref{thm:batch_std_greedy}, this product does not always
refine the worst case bound as tightly as desired.
Incorporating curvature into our arguments should result in a more expressive approximation guarantee, and the technical path could be similar to  \cite{Conforti_Cornuejols_1984,Bian_etal_2017}.

The batch size $q$ implicitly appears in the result \cref{eqn:batch_std_greedy_approx_guarantee} through the supermodularity ratio. It can be understood to be either the uniform batch size or more generally the maximum batch size. One's computational budget should ultimately dictate the batch size; we know that a smaller maximum batch size will tend to yield better performance. An adaptive strategy to change the batch size  across different steps could thus aid the efficient utilization of  computational resources.
Such a strategy will not influence the worst case bound, but in many practical problems should improve empirical performance. Several factors should impact any adaptive strategy: the number of indices already selected, the number that remains to be selected, and the contrasts between the incremental gains corresponding to each of the remaining indices. It would make sense to measure the latter contrasts relative to the function value evaluated on the set of already chosen indices. 

The supermodularity ratio that is essential to characterize the batch greedy heuristic can be a useful theoretical construct in several other settings. Consider the work \cite{Ilev_2001_backward}, where the author analyzed the greedy descent algorithm for minimizing any non-increasing supermodular function.
The approximation guarantee was given using the \emph{steepness} of the function, which is a counterpart to curvature 
for supermodular functions.
Using the supermodularity ratio, it should be possible to analyze the case of minimizing any non-increasing set function. Such an effort would mirror the work in \cite{Bian_etal_2017} but in the context of minimizing non-increasing functions.

The modular bounds we derived for the information theoretic objectives in \Cref{sec:mm_alg_lbip} have  intriguing connections to concepts in polyhedral combinatorics. When the design objective is a submodular function, the subgradient 
vector that defines the modular bound is in general a non-extremal point in the base polytope
associated with the function.
Typically algorithms rely on enumerating  coordinates of the vertices of the polytope and iterate towards a locally optimal solution \cite{IyerJegelkaBilmes2012_mirror,IyerJegBil2013}; our approach is thus somewhat unconventional.
The existence of such a subgradient was established using the operator concave inequality (\Cref{thm:operator_concave_inequality}), and was possible since our objective involves an operator monotone function, $\log(\cdot)$, acting on a Hermitian operator (covariance of the data marginal).
Hermitian operators that characterize the {volume/diversity} of subsets arise in other situations too---for instance, as kernels of determinantal point processes.
In all such cases it is unclear if more nuanced and fundamental links exist between classical results in functional analysis and notions in combinatorial optimization.  It would be interesting to understand these connections more fully, building on what we have been able to show and exploit.

\section*{Acknowledgments}
The authors acknowledge support from the US Department of Energy, Office of Advanced Scientific Computing Research, SciDAC program; from the Air Force Office of Scientific Research, Computational Mathematics Program; and from the National Science Foundation under grant DMS-1723011. The authors are grateful to Daniel Ricciuto for help with the optimal sensor placement problem; to Stefanie Jegelka for an insightful discussion concerning submodular functions; to Jean-Christophe Bourin for a helpful correspondence clarifying a detail about operator concave inequalities; and to Arvind Saibaba for help in simplifying some linear algebra arguments. The authors would also like to thank Alessio Spantini, Ricardo Baptista, and Fengyi Li for helpful discussions.

\FloatBarrier
\bibliographystyle{siamplain}
\bibliography{references}
\newpage
\FloatBarrier
\appendix
\section{Additional background and definitions} \label{appendix:def}

\subsection{Definite matrices and generalized eigenvalue problems} \label{subsec:notation_defMat_defGEV}

\begin{definition}[L\"{o}wner ordering or the positive semi-definite ordering] \label{def:lowner_ordering} For any two Hermitian matrices $A$ and $B$, we write $A \succeq B$ if and only if $A-B$ is positive semi-definite. The positive semi-definite condition can be used to define a partial ordering on all Hermitian matrices.
\end{definition}
\begin{definition}  [Matrix pencil \cite{sun_stewart_1990}] \label{def:matrixPencil}  Given matrices $A$ and $B$, a matrix pencil is a family of matrices $A - \lambda B$, parametrized by a complex number $\lambda$.
\end{definition}
\begin{definition} [Definite pencil and definite generalized eigenvalue problem \cite{sun_stewart_1990}]\label{def:definiteGEV} 
The pair of Hermitian matrices $(A, B)$ is a definite pencil if
\begin{displaymath}
c(A,B) = \inf _{\norm{x} =1}  \lbrace x^{*}(A+\sqrt{-1}B)x \rbrace   > 0,
\end{displaymath}
where $c( A, B)$ is called the Crawford number \cite{crawford_1976} of the definite pencil $A - \lambda B$. The generalized eigenvalue problem $A\Upsilon=B\Upsilon\Sigma$  is definite if $(A, B)$ is definite.
\end{definition}
Throughout the manuscript we refer to the generalized eigenvalue problem involving matrices $A$ and $B$ through the pair $(A, B)$. We  use $\sigma_k$ to denote any non-trivial generalized eigenvalue of $(A,B)$ and $\Sigma$ to denote the diagonal matrix of eigenvalues.

\subsection{Notions from statistical information theory} \label{subsec:stat_info_theory}

\begin{definition} [Differential entropy \cite{cover_thomas_2005}]\label{def:entropy} The entropy of a continuous random variable $X$ with density $\pi_{X}$ is defined as
\begin{displaymath}
	\mathcal{H}(X) \coloneqq  - \mathbb{E}_{\pi_{X}} \log \pi_{X}(x)\, .
\end{displaymath}
\end{definition}
\begin{definition} [Condtional differential entropy \cite{cover_thomas_2005}]\label{def:cond_entropy} For two continuous random variables $X$ and $Y$, let $\pi_{X,Y}$ denote their joint density and let $\pi_{X|Y}$ denote the density of the conditional distribution of $X|Y$. Then the conditional entropy $\mathcal{H}\lr{X|Y}$ is defined as
\begin{displaymath}
	\mathcal{H}\lr{X|Y} \coloneqq  - \mathbb{E}_{\pi_{X,Y}} \log \pi_{X|Y}(x \vert y) \, .
\end{displaymath}
\end{definition}
\begin{definition} [Relative entropy / Kullback--Leibler divergence \cite{cover_thomas_2005}] \label{def:KL_divergence}
The relative entropy or Kullback--Leibler divergence between two distributions with densities $\pi$ and $\widehat{\pi}$ is defined as
\begin{displaymath}
	D_{KL} \lr{\pi \| \widehat{\pi} } \coloneqq  \mathbb{E}_{\pi} \log { \frac{\pi(x)}{\widehat{\pi}(x)} } \, .
\end{displaymath}
\end{definition}
\subsection{Submodularity and related concepts} \label{subsec:notation_submodularity}
\begin{definition} [Modular set function \cite{fujishige2005,Bach2013}] \label{def:modular_function}
A set function $F:2^{\mathscr{V}} \rightarrow \mathbb{R}$  is modular (i.e., both submodular and supermodular) if and only if there exists $s \in \mathbb{R}^{\abs{\mathscr{V}}}$ such that $F\lr{\mathscr{A}} = \sum_{k \in \mathscr{A}} s_k$. 
\end{definition}
In the literature concerning submodular functions, it is common to refer to any vector $s \in \mathbb{R}^{\abs{\mathscr{V}}}$ as the modular set function defined as $ s(\mathscr{A}) = \sum_{k \in \mathscr{A}} s_k$. This practice is particularly useful when discussing submodular and base polyhedra, or subgradients.
\begin{definition} [Submodular and base polyhedra \cite{fujishige2005,Bach2013}]  \label{def:submodular_base_polyhedra}  Let $F$ be a
submodular function such that $F(\emptyset) = 0$. The submodular polyhedron $P(F)$ and the base polyhedron $B(F)$ are defined as:
\begin{align*}
P(F) 
& \coloneqq \left \{ s \in \mathbb{R}^{m} :  \forall \mathscr{A} \subset \mathscr{V}, s(\mathscr{A}) \leq F(\mathscr{A}) \right \}, \ \text{and} \\
B(F) 
&\coloneqq \left \{ s \in \mathbb{R}^{m} :  s(\mathscr{V}) = F(\mathscr{V}), \forall \mathscr{A} \subset \mathscr{V}, s(\mathscr{A}) \leq F(\mathscr{A}) \right \}, \\
&= P(F)  \cap \left \{ s(\mathscr{V}) = F(\mathscr{V})  \right \}.
\end{align*} 
\end{definition}
\begin{remark} Analogous to the submodular polyhedron, one can also define the supermodular polyhedron for supermodular functions with the inequalities in \Cref{def:submodular_base_polyhedra} being accordingly reversed.
\end{remark}
\begin{definition} [Generalized  lower  and  upper  polyhedra \cite{Iyer_Bilmes_2015_polyhedral}]  \label{def:generalized_polyhedra}  Let $F$ be any set function, not necessarily sub- or super-modular. The generalized lower polyhedron is defined as:
\begin{displaymath}
P^{gen}_l(F) \coloneqq \left \{ \lr{s,\mathfrak{S}} : s \in \mathbb{R}^{m}, \mathfrak{S} \in \mathbb{R}, \forall \mathscr{A} \subset \mathscr{V}, s(\mathscr{A}) + \mathfrak{S} \leq F(\mathscr{A}) \right \}
\end{displaymath}
By reversing the inequality above, we can define the generalized upper polyhedron as:
\begin{displaymath}
P^{gen}_u(F) \coloneqq \left \{ \lr{s,\mathfrak{S}} : s \in \mathbb{R}^{m}, \mathfrak{S} \in \mathbb{R}, \forall \mathscr{A} \subset \mathscr{V}, s(\mathscr{A}) + \mathfrak{S} \geq F(\mathscr{A}) \right \}
\end{displaymath}
\end{definition}
\begin{definition} [Subgradients and subdifferentials of submodular functions \cite{fujishige2005}] \label{def:sub_grad_diff_SM} 
Consider a submodular function $F: \mathscr{D} \rightarrow \mathbb{R}$  on a distributive lattice $ \mathscr{D} \subseteq 2^{\mathscr{V}}$, with $\emptyset,\mathscr{V} \in  \mathscr{D}$. For  $\alpha \in  \mathbb{R}^{\mathscr{V}}$ and $\mathscr{A} \in  \mathscr{D}$, if
\begin{displaymath}
\alpha(\mathscr{B})-\alpha(\mathscr{A}) \leq F(\mathscr{B})-F(\mathscr{A})
\end{displaymath}
holds for each $\mathscr{B} \in  \mathscr{D}$,  then we call $\alpha$ a subgradient of $F$ at $\mathscr{A}$. We denote by $\partial_{F}(\mathscr{A})$  the set of all the subgradients of  $F$ at $\mathscr{A}$ and call $\partial_{F}(\mathscr{A})$ the subdifferential  of $F$ at $\mathscr{A}$.
\end{definition}
\begin{definition} [Supergradients and superdifferentials of submodular functions \cite{fujishige2005}] \label{def:super_grad_diff_SM}
Consider a submodular function $F: \mathscr{D} \rightarrow \mathbb{R}$  on a distributive lattice $ \mathscr{D} \subseteq 2^{\mathscr{V}}$, with $\emptyset,\mathscr{V} \in  \mathscr{D}$. For  $\alpha \in  \mathbb{R}^{\mathscr{V}}$ and $\mathscr{A} \in  \mathscr{D}$, if
\begin{displaymath}
\alpha(\mathscr{B})-\alpha(\mathscr{A}) \geq F(\mathscr{B})-F(\mathscr{A})
\end{displaymath}
holds for each $\mathscr{B} \in  \mathscr{D}$,  then we call $\alpha$ a supergradient of $F$ at $\mathscr{A}$. We denote by $\partial^{F} (\mathscr{A})$  the set of all the supergradients of  $F$ at $\mathscr{A}$ and call $\partial^{F}(\mathscr{A})$ the superdifferential of  $F$ at $\mathscr{A}$.
\end{definition}
\begin{definition} [Supermodular dual  \cite{fujishige2005}] \label{def:supermodular_dual} For any submodular function $F: 2^{\mathscr{V}} \rightarrow \mathbb{R}$, the function $F^{\#}(\mathscr{X}) \coloneqq  F(\mathscr{V}) - F(\mathscr{V} \setminus \mathscr{X}), \mathscr{X} \subseteq \mathscr{V}$ is referred to as its  supermodular dual, with the properties $(F^{\#})^{\#} = F$, and $B(F)=B(F^{\#})$ \cite[Lemma~2.4]{fujishige2005}.  Here $B$ is the base polytope associated with each function (\Cref{def:submodular_base_polyhedra}). 
\end{definition}

\begin{definition} [Total curvature of a non-decreasing submodular function \cite{Conforti_Cornuejols_1984,Sviridenko_etal_2017}] \label{def:total_curvature}
\begin{displaymath}
c \coloneqq \max_{\nu \in \mathscr{V}} \frac{\rho_{\nu}(\emptyset) - \rho_{\nu}(\mathscr{V}\setminus \{ \nu \}) }{\rho_{\nu}(\emptyset)} = 1 - \min_{\nu \in \mathscr{V}} \frac{\rho_{\nu}(\mathscr{V}\setminus \{ \nu \}) }{\rho_{\nu}(\emptyset)}.
\end{displaymath}
\end{definition}

\begin{definition} [Generalized curvature \cite{Bian_etal_2017}] \label{def:gen_curvature} The generalized curvature of a non-negative set function $F$ is the smallest scalar $\alpha$  s.t.,
\begin{displaymath}
\rho_{\nu} \left(  \mathscr{A} \setminus \{\nu\}  \cup \mathscr{B}  \right)
\geq 
(1-\alpha) \rho_{\nu} \left(  \mathscr{A} \setminus \{\nu\} \right), 
\qquad \forall \mathscr{A}, \mathscr{B} \subseteq \mathscr{V}, \nu \in \mathscr{A} \setminus \mathscr{B}.
\end{displaymath}
\end{definition}

\begin{definition} [Submodularity ratio from \cite{Das_Kempe_2011}] \label{def:sm_ratio_das_kempe} The submodularity ratio of a non-negative set function $F$ with respect to a set $\mathscr{V}$ and a parameter $k \geq 1$ is
\begin{displaymath}
\gamma_{\mathscr{V},k}(F) \coloneqq \min_{\mathscr{B} \subseteq \mathscr{V}, \mathscr{A}: \abs{\mathscr{A}} \leq k, \mathscr{A} \cap \mathscr{B} = \emptyset} \frac{ \sum_{\nu \in \mathscr{A}} \rho_{\nu}(\mathscr{B})}{\rho_{\mathscr{A}}(\mathscr{B})}
\end{displaymath}
\end{definition}
\begin{definition} [Submodularity ratio from \cite{Bian_etal_2017}] \label{def:sm_ratio_bian} The submodularity ratio of a non-negative set function $F$ is the largest scalar $\gamma$  s.t.,
\begin{displaymath}
\sum_{\nu \in \mathscr{A} \setminus \mathscr{B}} \rho_{\nu}(\mathscr{B}) \geq \gamma \rho_{\mathscr{A}} (\mathscr{B}),
\qquad  \forall \mathscr{A}, \mathscr{B} \subseteq \mathscr{V}.
\end{displaymath}
\end{definition}
\section{Technical results} \label{appendix:techResults}

\begin{proof} {\bfseries \Cref{prop:MI_SM_uncorrelatedCase}} \label{proof:prop_MI_SM_uncorrelatedCase}
If $Y_{i_j} | X$ are independent then the conditional entropy of  $\mathcal{P}^{\top}Y|X$ can be written as $
\mathcal{H}(\mathcal{P}^{\top}Y|X) = \sum_{j=1}^{k} \mathcal{H}(Y_{i_j} | X).$ This allows us to write the mutual information $\mathcal{I}(X;\mathcal{P}^{\top}Y)$ as follows:
\begin{equation} \label{eqn:MI_UncorrelatedCase}
 \mathcal{I}(X;\mathcal{P}^{\top}Y) = \mathcal{H}(\mathcal{P}^{\top}Y) - \mathcal{H}(\mathcal{P}^{\top}Y|X)= \mathcal{H}(\mathcal{P}^{\top}Y) - \sum_{j=1}^{k} \mathcal{H}(Y_{i_j} | X).
\end{equation}
Consider selection operators $\mathcal{P}_1 \in \mathbb{R}^{ m \times k_1}$ and $\mathcal{P}_2 \in \mathbb{R}^{ m \times k_2}$ such that $\mathscr{I}(\mathcal{P}_1) \subseteq \mathscr{I}(\mathcal{P}_2)$. Consider a canonical row $e_j$ that selects $Y_{i_j}$, with  $\text{supp}(e_j) \notin \mathscr{I}(\mathcal{P}_2)$. We define new selection operators $\widehat{\mathcal{P}}_1 \coloneqq  [\mathcal{P}_1, e_j] \in \mathbb{R}^{ m \times (k_1 + 1) }$ and  $\widehat{\mathcal{P}}_2 \coloneqq  [\mathcal{P}_2, e_j] \in \mathbb{R}^{ m \times (k_2+1)}$ by appending $e_j$ to $\mathcal{P}_1$ and $\mathcal{P}_2$ respectively. The  incremental change in mutual information from  incorporating $Y_{i_j}$ can be determined for each case. Let us  consider $\mathcal{I}(X;\widehat{\mathcal{P}}_2^{\top} Y) - \mathcal{I}(X;\mathcal{P}_2^{\top} Y) $; the expression can be adapted for $  \mathcal{I}(X;\widehat{\mathcal{P}}_1^{\top} Y) - \mathcal{I}(X;\mathcal{P}_1^{\top} Y)$ accordingly.
\begin{subequations} \label{eqn:MI_incrementalDiff}
\begin{align}
  \mathcal{I}(X;\widehat{\mathcal{P}}_2^{\top} Y) - \mathcal{I}(X;\mathcal{P}_2^{\top} Y) 
\label{eqn:MI_incrementalDiff_a}  &= \mathcal{H}(\widehat{\mathcal{P}}_2^{\top} Y) - \mathcal{H}(\mathcal{P}_2^{\top} Y) - \mathcal{H}(Y_{i_j} | X),\\
\label{eqn:MI_incrementalDiff_b}  &= \mathcal{H}( Y_{i_j} | \mathcal{P}_2^{\top} Y) - \mathcal{H}(Y_{i_j} | X).
\end{align}
\end{subequations}
To obtain \cref{eqn:MI_incrementalDiff_b} we recognize from the chain rule for entropy that $\mathcal{H}( Y_{i_j} | \mathcal{P}_2^{\top} Y)  =   \mathcal{H}(\widehat{\mathcal{P}}_2^{\top} Y) - \mathcal{H}(\mathcal{P}_2^{\top} Y)$. Since conditioning cannot increase entropy we can assert, 
\begin{align}
\mathcal{H}( Y_{i_j} | \mathcal{P}_1^{\top} Y) & \geq \mathcal{H}( Y_{i_j} | \mathcal{P}_2^{\top} Y), \\
\Rightarrow  \mathcal{I}(X;\widehat{\mathcal{P}}_1^{\top} Y) - \mathcal{I}(X;\mathcal{P}_1^{\top} Y)  & \geq   \mathcal{I}(X;\widehat{\mathcal{P}}_2^{\top} Y) - \mathcal{I}(X;\mathcal{P}_2^{\top} Y),
\end{align}
which concludes the proof.
\end{proof}

\subsection{Proofs concerning batch greedy algorithms} \label{appendix_subsec:greedy_stuff}

Below we provide a formal proof demonstrating that any function  $F$  is supermodular if and only if the supermodularity ratio $\eta = 1$. Recall that the supermodularity ratio without the cardinality parameter is defined as the largest scalar such that
\begin{displaymath}
\frac{\rho_{\mathscr{A}} (\mathscr{B})}{\sum_{\nu \in \mathscr{A} \setminus \mathscr{B}} \rho_{\nu}(\mathscr{B})} \geq \eta ,
\qquad  \forall \mathscr{A}, \mathscr{B} \subseteq \mathscr{V}.
\end{displaymath}
\begin{proof_without_of_suffix}  \label{proof:prop_supermodularity_ratio}
    Assuming $\eta=1$, we can claim the following inequalities for any set $\mathscr{B} \subseteq \mathscr{V}$, and $\{\nu_1,\nu_2\} \in \mathscr{V} \setminus \mathscr{B}$
    \begin{align}
    		F(\{\nu_1,\nu_2\} \cup \mathscr{B}) - F(\mathscr{B}) & \geq  \sum_{i=1}^{2} \lr{  F(\{\nu_i\} \cup \mathscr{B})  - F(\mathscr{B})  },\\
\Rightarrow     		F(\{\nu_1,\nu_2\} \cup \mathscr{B}) - F(\{\nu_1\} \cup \mathscr{B}) & \geq F(\{\nu_2\} \cup \mathscr{B}) - F(\mathscr{B}).
    \end{align}    
This demonstrates supermodularity  \cite[Proposition~2.3]{Bach2013} having assumed $\eta=1$. To complete the proof, we prove the proposition statement the other way around by recursively exploiting the supermodularity property. For any sets $\mathscr{A},\mathscr{B} \subseteq \mathscr{V}$ with $\abs{\mathscr{A}}\leq k, \mathscr{A} \cap \mathscr{B} = \emptyset$ consider $\nu_{k} \in \mathscr{A} $, we claim because of  supermodularity
\begin{align}
	F(\mathscr{A}\cup \mathscr{B}) - F(\mathscr{B}) 
	&\geq F(  \mathscr{A} \setminus \{\nu_{k}\} \cup\mathscr{B}) - F(\mathscr{B}) + \underbrace{F(\{\nu_{k}\} \cup \mathscr{B}) - F(\mathscr{B})}_{\rho_{\nu_k}(\mathscr{B})}.
\end{align}
Repeating the same argument but on the term $F(  \mathscr{A} \setminus \{\nu_{k}\} \cup\mathscr{B}) - F(\mathscr{B})$ we have,
\begin{align}
	F(\mathscr{A}\cup \mathscr{B}) - F(\mathscr{B}) 
	&\geq F(  \mathscr{A} \setminus \{\nu_{k},\nu_{k-1}\} \cup\mathscr{B}) - F(\mathscr{B}) + \rho_{\nu_k}(\mathscr{B}) + \rho_{\nu_{k-1}}(\mathscr{B}).
\end{align}
Continuing the process gives us the inequality
\begin{equation}
	F(\mathscr{A} \cup \mathscr{B}) - F(\mathscr{B}) 
	\geq \sum_{\nu \in \mathscr{A}} \lr{F(\{\nu\} \cup \mathscr{B}) - F(\mathscr{B})}.
\end{equation}\end{proof_without_of_suffix}

\subsubsection{Standard batch greedy algorithm} \label{appendix_subsubsec:std_batch_greedy}

\begin{proof}  {\bfseries \Cref{thm:batch_std_greedy}} \label{proof:theorem_batch_std_greedy} 
	Our arguments closely follow those of \cite{Nemhauser1978,Wolsey1982} with suitable modifications. Let $\mathscr{A}^{*}$ be a maximizer of $F$ with $k$ elements. Let $\{a_{i_1},\ldots,a_{i_{q_j}} \}$ be the $q_j$ elements selected during the $j$--th step of the greedy algorithm. If $\mathscr{A}_{j-1}$ is the set of elements after $j-1$ steps, then $\mathscr{A}_j = \mathscr{A}_{j-1} \cup \{a_{i_1},\ldots,a_{i_{q_j}} \}$. If there are $l$ steps altogether, we naturally have $q_1 + q_2 + \cdots + q_l = k$.
	For a given $j \in \{1,\ldots,l \}$, we denote by $\{b_1,\ldots,b_m\}$ the elements of $\mathscr{A}^{*} \setminus \mathscr{A}_{j-1}$ (we must have $k \geq m \geq q_j + q_{j+1} + \cdots + q_l $). We then have
\begin{subequations}	\begin{align}
		F(\mathscr{A}^{*}) 
		& \overset{a}{\leq} F(\mathscr{A}^{*} \cup \mathscr{A}_{j-1}),\\
	    & \overset{b}{\leq} F(\mathscr{A}_{j-1})  + \frac{1}{\gamma_{\mathscr{V},m}}\sum_{i=1}^{m} \lr{F(\mathscr{A}_{j-1} \cup \{ b_i \}) - F(\mathscr{A}_{j-1})}\\
    	    & \overset{c}{\leq} F(\mathscr{A}_{j-1}) + \frac{m}{q_j \gamma_{\mathscr{V},m}} \sum_{t=1}^{q_j} \lr{F(\mathscr{A}_{j-1} \cup \{ a_{i_t} \}) - F(\mathscr{A}_{j-1})}\\
    	    & \overset{d}{\leq} F(\mathscr{A}_{j-1}) + \frac{m}{q_j \eta_{\mathscr{V},q_j} \gamma_{\mathscr{V},m}} \lr{F(\mathscr{A}_{j}) - F(\mathscr{A}_{j-1})}\\
\label{eqn:proof_theorem_batch_std_greedy_1}  & \overset{e}{\leq} F(\mathscr{A}_{j-1})  + \frac{k}{q_j  \eta_{\mathscr{V},q_j} \gamma_{\mathscr{V},k}} \lr{F(\mathscr{A}_{j}) - F(\mathscr{A}_{j-1})}.
	\end{align} \end{subequations}
\begin{itemize}
	\item[a)] holds because $F$ is non-decreasing,
	\item[b)] holds by definition of the submodularity ratio (\Cref{def:sm_ratio_das_kempe}),
	\item[c)] holds because $q_j \leq m$, and the $q_j$ elements in the set $\{a_{i_1},\ldots,a_{i_{q_j}} \}$ are (by virtue of the batch greedy algorithm) those that correspond to the highest incremental gains at step $j$:
	\[ \frac{\sum_{i=1}^{m} \lr{F(\mathscr{A}_{j-1} \cup \{ b_i \}) - F(\mathscr{A}_{j-1})}}{m} \leq \frac{\sum_{t=1}^{q_j} \lr{F(\mathscr{A}_{j-1} \cup \{ a_{i_t} \}) - F(\mathscr{A}_{j-1})}}{q_j} \]
	\item[d)] holds by definition of $ \eta_{\mathscr{V},q_j}$ as the supermodularity ratio (\Cref{def:supermodularity_ratio}),
	\item[e)] holds because $m \leq k$, and $\gamma_{\mathscr{V},m} \geq \gamma_{\mathscr{V},k}$.
\end{itemize}
A simple manipulation of \cref{eqn:proof_theorem_batch_std_greedy_1} allows us to write the following,
\begin{equation}
	F(\mathscr{A}^{*})  - F(\mathscr{A}_j) \leq \lr{ 1 - \frac{q_j  \eta_{\mathscr{V},q_j} \gamma_{\mathscr{V},k} }{k}} \lr{ F(\mathscr{A}^{*})  - F(\mathscr{A}_{j-1})} .
\end{equation}
This leads to $F(\mathscr{A}^{*})  - F(\mathscr{A}_l) \leq \prod_{j=1}^{l} \lr{ 1 - \frac{q_j  \eta_{\mathscr{V},q_j} \gamma_{\mathscr{V},k}}{k}}  F(\mathscr{A}^{*})$, and the theorem statement follows.
\end{proof}

\subsubsection{Distributed batch greedy algorithm} \label{appendix_subsubsec:parallel_batch_greedy}

Our approach to prove \Cref{thm:parallel_batch_std_greedy}  {dapts the strategies in \cite{Mirzasoleiman_etal_2013} to the batch and nonsubmodular setting. 
Similar to the exposition there, we first investigate some approximation properties of the intractable but communication efficient counterpart to \Cref{alg:parallel_batch_std_greedy}. Here one  first distributes the ground set $\mathscr{V}$ to $n_p$ processes. Each process then finds the optimal solution $A^{\text{opt}}_{i,[n_p,k]}$, i.e., a set of cardinality at most $k$, that maximizes the value of $F$ in each partition. These solutions are then merged $\bigcup_i A^{\text{opt}}_{i,[n_p,k]}$, and the optimal subset $A^{\text{d-opt}}_{[n_p,k]}$ of cardinality $k$ is found in the combined set. The following lemma relates $A^{\text{d-opt}}_{[n_p,k]}$ to the combinatorial optimal solution $A^{\text{opt}}_{[k]} = \argmax_{\mathscr{B} \subset \mathscr{V}, \abs{\mathscr{B}} \leq k} F(\mathscr{B})$.

\begin{lemma} \label{lemma:parallel_greedy_helpful_lemma_1} 
Let $F$ be a non-decreasing  function with $F(\emptyset)=0$, and let $k > 0$. Then
\begin{displaymath}
F(A^{\text{d-opt}}_{[n_p,k]}) \geq \frac{\gamma_{\mathscr{V},k}}{k}  F(A^{\text{opt}}_{[k]})
\end{displaymath}
\end{lemma}
\begin{proof}  {\bfseries \Cref{lemma:parallel_greedy_helpful_lemma_1}}
Let $ A^{\text{opt}}_{[k]} = \{ \nu_1,\ldots,\nu_k \}$. Using the notion of submodularity ratio, we have $\sum_{\nu \in  A^{\text{opt}}_{[k]} } \rho_{\nu}(\emptyset) \geq \gamma_{\mathscr{V},k} \rho_{A^{\text{opt}}_{[k]}}(\emptyset)$, 
$\Rightarrow F(A^{\text{opt}}_{[k]}) \leq \frac{1}{\gamma_{\mathscr{V},k}}\sum_{\nu \in  A^{\text{opt}}_{[k]} } F(\nu)$. If $\nu^{*} = \argmax_{\nu \in  A^{\text{opt}}_{[k]}} F(\nu)$, then $ F(A^{\text{opt}}_{[k]}) \leq \frac{k}{\gamma_{\mathscr{V},k}} F(\nu^{*})$. Suppose $\nu^{*} \in \mathscr{V}_j$, the $j$--th partition  of the ground set; then we know that $F(\nu^{*}) \leq F(A^{\text{opt}}_{j,[k]}) $.  The lemma statement now follows since $ F(A^{\text{opt}}_{j,[n_p,k]}) \leq \max_{j} F(A^{\text{opt}}_{j,[n_p,k]}) \leq F(A^{\text{d-opt}}_{[n_p,k]})$.
\end{proof}

Apart from \Cref{lemma:parallel_greedy_helpful_lemma_1} we also need   a slight generalization of \Cref{thm:batch_std_greedy,eqn:batch_std_greedy_approx_guarantee} 
to prove \Cref{thm:parallel_batch_std_greedy}. The  result below is straightforward to derive and we omit the proof.
\begin{lemma} \label{lemma:parallel_greedy_helpful_lemma_2} 
Let $F$ be a non-decreasing  function with $F(\emptyset)=0$, and let $\mathscr{A}^{\text{bg}}_{[\widehat{k}]}$ be the index set of cardinality $\widehat{k}=q \widehat{l}$ returned by the batch greedy algorithm maximizing $F(\mathscr{A})$ subject to $\abs{\mathscr{A}} \leq k = ql$. Then,
\begin{displaymath}
	F(\mathscr{A}^{\text{bg}}_{[\widehat{k}]}) \geq \lr{1 - e^{- \frac{\widehat{l}}{l}\eta_{\mathscr{V},q} \gamma_{\mathscr{V},k} }} 	F(A^{\text{opt}}_{[k]}).
\end{displaymath}
Here $\gamma_{\mathscr{V},k}$ is the submodularity ratio and $\eta_{\mathscr{V},q}$ is the supermodularity ratio.
\end{lemma}

\begin{proof}  {\bfseries \Cref{thm:parallel_batch_std_greedy}} \label{proof:theorem_parallel_batch_std_greedy} 
Let $\widehat{\mathscr{A}} = \argmax_{\mathscr{B} \subset \mathscr{M}, \abs{\mathscr{B}} \leq \widehat{k}} F(\mathscr{B})$, where $ \mathscr{M} \coloneqq \bigcup_i \mathscr{A}_{i,[n_p,\widehat{k}]}^{\text{bg}}$. 
 We know that $F(\widehat{\mathscr{A}}) \geq \max_{i} F(\mathscr{A}^{\text{bg}}_{i,[n_p,\widehat{k}]}) $.
From \Cref{lemma:parallel_greedy_helpful_lemma_2}, we have
\begin{align}
	F(\mathscr{A}^{\text{bg}}_{i,[n_p,\widehat{k}]}) & \geq  \lr{1 - e^{- \eta_{\mathscr{V},q} \gamma_{\mathscr{V},k} \widehat{l} / l }} 	F(A^{\text{opt}}_{i,[n_p,k]}), \\
\Rightarrow 	F(\widehat{\mathscr{A}})  & \geq  \lr{1 - e^{- \eta_{\mathscr{V},q} \gamma_{\mathscr{V},k} \widehat{l} / l }}  \max_{i} 	F(A^{\text{opt}}_{i,[n_p,k]}).
\end{align}
From \Cref{lemma:parallel_greedy_helpful_lemma_1}, we know that $\max_{i} F(A^{\text{opt}}_{i,[n_p,k]}) \geq  \frac{\gamma_{\mathscr{V},k}}{k} F(A^{\text{opt}}_{[k]})$; therefore
\begin{equation} \label{eqn:parallel_batch_std_greedy_proof_eqn1}
 	F(\widehat{\mathscr{A}})  \geq  \lr{1 - e^{- \eta_{\mathscr{V},q} \gamma_{\mathscr{V},k} \widehat{l} / l }}  \frac{\gamma_{\mathscr{V},k}}{k} F(A^{\text{opt}}_{[k]}).
\end{equation}
Using \Cref{lemma:parallel_greedy_helpful_lemma_2} we can relate $F(\widehat{\mathscr{A}})$ and $	F(\mathscr{A}^{\text{d-bg}})$ as follows:
\begin{equation} \label{eqn:parallel_batch_std_greedy_proof_eqn2}
	F(\mathscr{A}_{[n_p,\widetilde{k}]}^{\text{d-bg}}) \geq  \lr{1 - e^{- \eta_{\mathscr{V},q} \gamma_{\mathscr{V},\widehat{k}} \widetilde{l} / \widehat{l} }} F(\widehat{\mathscr{A}}).
\end{equation}
Combining \cref{eqn:parallel_batch_std_greedy_proof_eqn1,eqn:parallel_batch_std_greedy_proof_eqn2}, we obtain the theorem statement.
\end{proof}
\subsubsection{Stochastic batch greedy algorithm} \label{appendix_subsubsec:stochastic_batch_greedy}

We first provide two lemmas that will prove useful in establishing \Cref{thm:stochastic_batch_std_greedy}. The first lemma is a bound on the probability of a certain event involving the random set drawn at each step, and the second lemma is a bound on the expected incremental gain at each step.
\begin{lemma} \label{lemma:stochastic_batch_std_greedy_prob_bound}
Given solution $\mathscr{A}_i$ after $i$--steps of the \texttt{StochasticBatchGreedy} algorithm, if $k \leq \floor{\sqrt{m/e} -1/e}$ and $\frac{m -2k}{2 e k^2} \geq  \frac{q-1}{\log^2 \frac{q}{\epsilon}}$, then for any set $\mathscr{R}$  of size $\frac{m}{k} \log\frac{q}{\epsilon}$ sampled uniformly at random from  $ \mathscr{V} \setminus \mathscr{A}_{i}$, we have
$\mathbb{P}(\abs{\mathscr{R} \cap \mathscr{A}^{*} \setminus \mathscr{A}_{i}}  \geq q)  \geq \lr{1 - \epsilon} \frac{\abs{\mathscr{A}^{*} \setminus \mathscr{A}_{i}}}{k}$.
\end{lemma}
\begin{proof_without_of_suffix}
We assume that all random quantities are conditioned on $A_{i}$. 
The set $\mathscr{R}$ consists of $s$ elements from $ \mathscr{V} \setminus \mathscr{A}_{i}$ sampled uniformly at random (w.l.o.g.\ with repetition).
Let $\widetilde{k} \coloneqq  \abs{\mathscr{A}^{*} \setminus \mathscr{A}_{i}}, \widetilde{m} \coloneqq  \abs{\mathscr{V} \setminus \mathscr{A}_{i}}$, and $p \coloneqq  \frac{\abs{\mathscr{A}^{*} \setminus \mathscr{A}_{i}}} {\abs{\mathscr{V} \setminus \mathscr{A}_{i}}} = \frac{\widetilde{k}}{\widetilde{m}}$ be the probability representing the event that any element in the random set $\mathscr{R}$ is from $\mathscr{A}^{*} \setminus \mathscr{A}_{i}$.
We now estimate the probability of  $\abs{\mathscr{R} \cap \mathscr{A}^{*} \setminus \mathscr{A}_{i}}  \geq q$ by first deriving an upper bound on the complementary event $\abs{\mathscr{R} \cap \mathscr{A}^{*} \setminus \mathscr{A}_{i}}  \leq q-1$.
\begin{subequations}	\begin{align}
\mathbb{P}(\abs{\mathscr{R} \cap \mathscr{A}^{*} \setminus \mathscr{A}_{i}}  \leq q-1) 
& \overset{a}{=} \sum^{q-1}_{i=0} \binom{\widetilde{k}}{i} (1-p)^{s-i} p^{i}, \\
& \overset{b}{\leq}  \sum^{q-1}_{i=0} \lr{\frac{ e \widetilde{k}}{i}}^{i} (1-p)^{s-i} p^{i}, \\
& \overset{c}{\leq} (1-p)^{s} \sum^{q-1}_{i=0}  \lr{\frac{e \widetilde{k} p}{1-p}}^{i}, \\
& \overset{d}{\leq} (1-p)^{s} \lr{ 1 + (q-1) \lr{\frac{e \widetilde{k}^2}{m-2k}}  }, \\
\label{eqn:prob_complementary_event_upper_bound_f} & \overset{e}{\leq} \exp(-s \frac{\widetilde{k}}{m}) \psi(\widetilde{k}).
\end{align} \end{subequations}
\begin{itemize}
	\item[b)] holds by upper bounding the binomial coefficient $\binom{\widetilde{k}}{i} \leq \lr{\frac{e \widetilde{k} }{i}}^{i}$.
	\item[c)] holds because $\lr{\frac{1}{i}}^i \leq 1 \quad \forall i \in \mathbb{Z}_{\geq 0}$.
	\item[d)] holds because of the additional assumption we impose on the cardinality constraint $ k \leq \floor{\sqrt{m/e} -1/e}$. Using the facts that $m-k+q \leq \widetilde{m}$ and $\widetilde{k} \leq k$ we can claim,
     \[ \frac{e \widetilde{k} p}{1-p} = \frac{e \widetilde{k}^2}{\widetilde{m}-\widetilde{k}} \leq \frac{e \widetilde{k}^2}{m+q-2k} \leq \frac{e \widetilde{k}^2}{m-2k} \leq \frac{e k^2}{m-2k}  \leq 1 \Leftarrow k \leq \floor{\sqrt{m/e} -1/e} \]
	\item[e)] holds because,
	\[ (1-p)^{s} = (1-\widetilde{k}/\widetilde{m})^{s} \leq  \exp(-s \frac{\widetilde{k}}{\widetilde{m}}) \leq  \exp(-s \frac{\widetilde{k}}{m}) \Leftarrow \widetilde{m} \leq m \]
	 We have defined $\psi(\widetilde{k}) \equiv \psi(\widetilde{k};q,m,k) \coloneqq  1 + (q-1) \lr{\frac{e \widetilde{k}^2}{m-2k}}$.
\end{itemize}

We now derive a lower bound for the probability of  $\abs{\mathscr{R} \cap \mathscr{A}^{*} \setminus \mathscr{A}_{i}}  \geq q$.
\begin{subequations}	\begin{align}
\mathbb{P}(\abs{\mathscr{R} \cap \mathscr{A}^{*} \setminus \mathscr{A}_{i}}  \geq q) 
& \overset{a}{=} 1 - \mathbb{P}(\abs{\mathscr{R} \cap \mathscr{A}^{*} \setminus \mathscr{A}_{i}}  \leq q-1),\\
& \overset{b}{\geq} 1 - \exp(-s \frac{\widetilde{k}}{m}) \psi(\widetilde{k}) ,\\
& \overset{c}{\geq} \lr{1 - \exp(-s \frac{k}{m}) \psi(k)} \frac{\widetilde{k}}{k},\\
& \overset{d}{\geq} \lr{1 - \epsilon} \frac{\widetilde{k}}{k}.
\end{align} \end{subequations}
\begin{itemize}
	\item[b)] follows from \cref{eqn:prob_complementary_event_upper_bound_f},
	\item[c)] For $s \geq m \sqrt{\frac{2 e (q-1)}{m-2k}}$, the function $1 - \exp(-s \frac{x}{m}) \psi(x) $ is concave in the interval $x \in [0,k]$. It can be easily verified that $ \frac{\partial^2  \exp(-s \frac{x}{m}) \psi(x) }{\partial^2 x} > 0$  if  $s \geq m \sqrt{\frac{2 e (q-1)}{m-2k}} $.
	\begin{align*}
		\Rightarrow \frac{1 - \exp(-s \frac{x}{m}) \psi(x)}{x} &\geq \frac{1 - \exp(-s \frac{k}{m}) \psi(k)}{k},\\
		\Leftrightarrow 1 - \exp(-s \frac{x}{m}) \psi(x) &\geq \frac{x}{k} \lr{1 - \exp(-s \frac{k}{m}) \psi(k)}.
	\end{align*}
 	\item[d)] 
	We choose the set $\mathscr{R}$ \textit{s.t.}\ $\abs{\mathscr{R}} \eqqcolon s =\frac{m}{k} \log\frac{q}{\epsilon}$, where $0<\epsilon<1$.
	If the batch size $q$ and tolerance $\epsilon$ are chosen \textit{s.t.}\ $\frac{m -2k}{2 e k^2} \geq  \frac{q-1}{\log^2 \frac{q}{\epsilon}}$, then
	 $\frac{m}{k} \log\frac{q}{\epsilon} \geq  m \sqrt{\frac{2 e (q-1)}{m-2k}}$, ensuring concavity of the desired function  in the previous step. For the cardinality constraint regime $k \leq \floor{\sqrt{m/e} -1/e}$, this condition on batch size and tolerance is easily satisfied. Since $\psi(\widetilde{k}) \in [1,q]$, we now have
	\begin{equation}
			1 - \exp(-s \frac{\widetilde{k}}{m}) \psi(\widetilde{k}) \geq \frac{\widetilde{k}}{k} \lr{1 - \exp(-s \frac{k}{m}) \psi(k)} \geq \frac{\widetilde{k}}{k} \lr{1 - \epsilon} .
	\end{equation}	
\end{itemize}
\end{proof_without_of_suffix}
\begin{lemma}  \label{lemma:stochastic_batch_std_greedy_exp_gain}
Given solution $\mathscr{A}_i$ after $i$--steps of the \texttt{StochasticBatchGreedy} algorithm, if the assumptions of \Cref{lemma:stochastic_batch_std_greedy_prob_bound} hold, then the expected gain $\mathbb{E}[F(\mathscr{A}_{i+1})- F(\mathscr{A}_{i})]$
 in step $i+1$ is at least $\lr{1 - \epsilon} \frac{\gamma_{\mathscr{V},k} \eta_{\mathscr{V},q} q}{k}  \mathbb{E}[\rho_{\mathscr{A}^{*} \setminus \mathscr{A}_{i}}(\mathscr{A}_{i})]$.
\end{lemma}

\begin{proof_without_of_suffix}
We assume that all random quantities are conditioned on $A_{i}$.  We bound $F(\mathscr{A}_{i+1}) - F(\mathscr{A}_{i})$ from below. Since $\mathscr{Q}_{i+1} \coloneqq  \mathscr{A}_{i+1} \setminus \mathscr{A}_{i}$ is chosen by the stochastic batch greedy rule from $\mathscr{R}$, and if $\abs{\mathscr{R} \cap \mathscr{A}^{*} \setminus \mathscr{A}_{i}} \geq q$, then $\sum_{a \in \mathscr{Q}_{i+1}} \rho_{a}(\mathscr{A}_{i})$  is at least as large as  $\sum_{a \in \mathscr{Q}_{r}} \rho_{a}(\mathscr{A}_{i})$ in expectation. 
Here the set $\mathscr{Q}_{r}$ with $\abs{\mathscr{Q}_{r}}=q$ is chosen uniformly at random $\mathscr{R} \cap \mathscr{A}^{*} \setminus \mathscr{A}_{i}$. The set $\mathscr{R}$ is equally likely to contain any each element of $\mathscr{A}^{*} \setminus \mathscr{A}_{i}$. Hence a uniform random subset $\mathscr{Q}_{r} \subseteq \mathscr{R} \cap \mathscr{A}^{*} \setminus \mathscr{A}_{i}$ is actually a uniformly random subset of $\mathscr{A}^{*} \setminus \mathscr{A}_{i}$. As a result we obtain
\begin{subequations}	\begin{align}
	\mathbb{E}[\sum_{a \in \mathscr{Q}_{i+1}} \rho_{a}(\mathscr{A}_{i})] 
&\geq \mathbb{P}(\abs{\mathscr{R} \cap \mathscr{A}^{*} \setminus \mathscr{A}_{i}}  \geq q) \frac{\sum_{\mathscr{Q}_{r} \subseteq \mathscr{A}^{*} \setminus \mathscr{A}_{i}} \sum_{a \in \mathscr{Q}_{r}} \rho_{a}(\mathscr{A}_{i})}{ \binom{\abs{\mathscr{A}^{*} \setminus \mathscr{A}_{i}}}{q}  },\\
\label{eqn:stochastic_batch_std_greedy_exp_gain_eqn_1} &\geq \lr{1 - \epsilon} \frac{\abs{\mathscr{A}^{*} \setminus \mathscr{A}_{i}}}{k}
 \frac{\binom{\abs{\mathscr{A}^{*} \setminus \mathscr{A}_{i}}}{q}\frac{q }{ \abs{\mathscr{A}^{*} \setminus \mathscr{A}_{i}}} \sum_{a \in\mathscr{A}^{*} \setminus \mathscr{A}_{i}} \rho_{a}(\mathscr{A}_{i})}{ \binom{\abs{\mathscr{A}^{*} \setminus \mathscr{A}_{i}}}{q}  },\\
\label{eqn:stochastic_batch_std_greedy_exp_gain_eqn_2} &\geq \lr{1 - \epsilon} \frac{q}{k} \sum_{a \in\mathscr{A}^{*} \setminus \mathscr{A}_{i}} \rho_{a}(\mathscr{A}_{i}).
\end{align} \end{subequations}
\cref{eqn:stochastic_batch_std_greedy_exp_gain_eqn_1} holds on account of \Cref{lemma:stochastic_batch_std_greedy_prob_bound}.  Using the notion of submodularity and supermodularity ratios we can claim,
 \begin{align}
\label{eqn:stochastic_batch_greedy_sup_mod_ratio}	\mathbb{E}[\rho_{\mathscr{Q}_{i+1}}(\mathscr{A}_{i})] 	
	& \geq  	\eta_{\mathscr{V},q} \mathbb{E}[\sum_{a \in \mathscr{Q}_{i+1}} \rho_{a}(\mathscr{A}_{i})] , \\
\label{eqn:stochastic_batch_greedy_sub_mod_ratio}    \sum_{a \in\mathscr{A}^{*} \setminus \mathscr{A}_{i}} \rho_{a}(\mathscr{A}_{i}) 		
		&\geq   \gamma_{\mathscr{V},k} \rho_{\mathscr{A}^{*} \setminus \mathscr{A}_{i}}(\mathscr{A}_{i}).
 \end{align}
Combining  \cref{eqn:stochastic_batch_std_greedy_exp_gain_eqn_2,eqn:stochastic_batch_greedy_sup_mod_ratio,eqn:stochastic_batch_greedy_sub_mod_ratio} we have
\begin{equation}
	\mathbb{E}[\rho_{\mathscr{Q}_{i+1}}(\mathscr{A}_{i})]
	\geq \lr{1 - \epsilon} \frac{\gamma_{\mathscr{V},k} \eta_{\mathscr{V},q}  q}{k}  \rho_{\mathscr{A}^{*} \setminus \mathscr{A}_{i}}(\mathscr{A}_{i}).
\end{equation}
By taking expectation over all possible realizations of $\mathscr{A}_i$ we obtain the lemma.
\end{proof_without_of_suffix}

\begin{proof}  {\bfseries \Cref{thm:stochastic_batch_std_greedy}} \label{proof:theorem_stochastic_batch_std_greedy} 
Let $\mathscr{A}^{*}$ be a maximizer of $F$ with $k$ elements, i.e., $\mathscr{A}^{*} \coloneqq  \argmax_{\mathscr{B} \subset \mathscr{V}, \abs{\mathscr{B}} \leq k} F(\mathscr{B})$.  By recursively using \Cref{lemma:stochastic_batch_std_greedy_exp_gain}  for $i = 1, \ldots, l$, where $l= k/q$, and knowing that $ F(\emptyset) = 0$, we obtain:
\begin{align}
\mathbb{E}[F(\mathscr{A}_{l})] 
&	\geq F(\mathscr{A}^{*}) - \lr{1 - \lr{1 - \epsilon} \frac{\gamma_{\mathscr{V},k} \eta_{\mathscr{V},q}}{l}}^l \lr{ F(\mathscr{A}^{*}) -F(\emptyset)},\\
&	\geq \lr{1 - e^{- \lr{1 - \epsilon} \gamma_{\mathscr{V},k} \eta_{\mathscr{V},q}}} F(\mathscr{A}^{*}).
\end{align}
\end{proof}


\subsubsection{MM batch greedy algorithm} \label{appendix_subsubsec:MM_batch_greedy}

\begin{proof} {\bfseries \Cref{thm:greedy_ascent}} \label{proof:theorem_greedy_ascent}
	The arguments are similar to the proof of \Cref{thm:batch_std_greedy}. Let $\mathscr{A}^{*}$ be a maximizer of $F$ with $k$ elements. Let $\{a_{i_1},\ldots,a_{i_{q_j}} \}$ be the $q_j$ elements selected during the $j$--th step of the greedy algorithm. If $\mathscr{A}_{j-1}$ is the set of elements after $j-1$ steps, then $\mathscr{A}_j = \mathscr{A}_{j-1} \cup \{a_{i_1},\ldots,a_{i_{q_j}} \}$. If there are $l$--steps all together, we naturally have $q_1 + q_2 + \cdots + q_l = k$.
	For a given $j \in \{1,\ldots,l \}$, we denote by $\{b_1,\ldots,b_m\}$ the elements of $\mathscr{A}^{*} \setminus \mathscr{A}_{j-1}$ (we must have $k \geq m \geq q_j + q_{j+1} + \cdots + q_l $). We then have,
\begin{subequations} \begin{align}
	F(\mathscr{A}^{*}) 
	& \overset{a}{\leq} F(\mathscr{A}^{*} \cup \mathscr{A}_{j-1}),\\
	&\overset{b}{=}  F(\mathscr{A}_{j-1})  +  \rho_{\mathscr{A}^{*}} \lr{ \mathscr{A}_{j-1}} , \\
	&\overset{c}{\leq}  F(\mathscr{A}_{j-1})  +   \delta_j +  \mathcal{M}_{\uparrow}  \left[  \rho_{\mathscr{A}^{*}} \lr{ \mathscr{A}_{j-1}} \right], \\
	&\overset{d}{\leq}  F(\mathscr{A}_{j-1})  +   \delta_j +  \frac{k}{q_j} \mathcal{M}_{\uparrow}  \left[  \rho_{\mathscr{A}_{j}} \lr{ \mathscr{A}_{j-1}} \right], \\
	&\overset{e}{\leq}  F(\mathscr{A}_{j-1})  +   \delta_j +  \frac{k}{q_j} 	\rho_{\mathscr{A}_{j}} \lr{ \mathscr{A}_{j-1}}, \\
\label{eqn:proof_theorem_greedy_ascent_1} 	&\overset{f}{\leq}  F(\mathscr{A}_{j-1}) + \frac{k \lr{1+\tau_j}}{q_j} \rho_{\mathscr{A}_{j}} \lr{ \mathscr{A}_{j-1}}.
\end{align} \end{subequations}
\begin{itemize}
	\item[a)] holds because $F$ is non-decreasing,
	\item[b)] is an identity since $F(\mathscr{A}^{*} \cup \mathscr{A}_{j-1}) - F(\mathscr{A}_{j-1}) \eqqcolon  \rho_{\mathscr{A}^{*}} \lr{ \mathscr{A}_{j-1}}$ is the incremental gain,
	\item[c)] holds by defining $\delta_j$ as the smallest scalar satisfying,
	\begin{equation}
       \rho_{\mathscr{A}^{*}} \lr{ \mathscr{A}_{j-1}}  \leq   \delta_j +  \mathcal{M}_{\uparrow}  \left[  \rho_{\mathscr{A}^{*}} \lr{ \mathscr{A}_{j-1}} \right],
	\end{equation}
	\item[d)] holds because $q_j \leq m \leq k$, and the $q_j$ elements in the set $\{a : a \in \mathscr{A}_j \setminus \mathscr{A}_{j-1}\}$ maximize $\mathcal{M}_{\uparrow}  \left[  \rho_{a} \lr{ \mathscr{A}_{j-1}} \right]$, and hence
	\[ \frac{\mathcal{M}_{\uparrow}  \left[  \rho_{\mathscr{A}^{*}} \lr{ \mathscr{A}_{j-1}} \right]}{m} \leq  \frac{\mathcal{M}_{\uparrow}  \left[  \rho_{\mathscr{A}_{j}} \lr{ \mathscr{A}_{j-1}} \right]}{q_j}, \]
     \item[e)] holds since $\mathcal{M}_{\uparrow}  \left[  \rho_{\mathscr{A}_{j}} \lr{ \mathscr{A}_{j-1}} \right] \leq  \rho_{\mathscr{A}_{j}} \lr{ \mathscr{A}_{j-1}}$,
		\item[f)] holds by defining  $\tau_j$ as the smallest scalar satisfying
	\begin{equation}
	 \delta_j \leq \frac{ \tau_j k \rho_{\mathscr{A}_{j}} \lr{ \mathscr{A}_{j-1}}}{q_j}.
	\end{equation}
\end{itemize}
	$\delta_j$ quantifies the slackness of the modular lower bound, while $\tau_j$ does so relative to the incremental gain at the current step. A simple manipulation of \cref{eqn:proof_theorem_greedy_ascent_1} allows us to write the following:
\begin{equation}
	F(\mathscr{A}^{*})  - F(\mathscr{A}_j) \leq \lr{ 1 - \frac{q_j}{k \lr{1+\tau_j}}} \lr{ F(\mathscr{A}^{*})  - F(\mathscr{A}_{j-1})} .
\end{equation}
This leads to $F(\mathscr{A}^{*})  - F(\mathscr{A}_l) \leq \prod_{j=1}^{l} \lr{ 1 - \frac{q_j}{k \lr{1+\tau_j}}} F(\mathscr{A}^{*})$, and the theorem statement follows.
\end{proof}	
	
\subsection{Proofs concerning linear Bayesian experimental design criterion} \label{appendix_subsec:lbip}

\begin{proof} {\bfseries \Cref{prop:eigen_prob_equivalence}} \label{proof:prop_eigen_prob_equivalence}
Let $(\sigma_j, U_j)$  be any generalized eigenvalue-eigenvector pair of the definite pair $(\Gamma_{Y}-\Gamma_{Y|X},\Gamma_{Y|X})$.
Since  $\Gamma_{Y}-\Gamma_{Y|X}\succeq 0$ and $\Gamma_{Y|X} \succ 0$, we can claim $\sigma_j \geq 0$. It is evident that $(1+\sigma_j, U_j)$ is the corresponding generalized eigenvalue-eigenvector pair of the definite pair $(\Gamma_{Y},\Gamma_{Y|X})$. To 
establish \cref{enumerate:prop_eigen_prob_equivalence_1} and \cref{enumerate:prop_eigen_prob_equivalence_2}, we need to show that $\sigma_j$ and $1+\sigma_j$ are the corresponding generalized eigenvalues of the definite pair $(\Gamma_{X}- \Gamma_{X|Y},\Gamma_{X|Y})$ and $(\Gamma_{X},\Gamma_{X|Y})$, respectively. \Cref{enumerate:prop_eigen_prob_equivalence_1} $\Leftrightarrow$ \cref{enumerate:prop_eigen_prob_equivalence_2} then holds automatically.

Using the specified linear statistical model \cref{eqn:linearFwdModel}, we can deduce $\Gamma_{Y}-\Gamma_{Y|X} = G \Gamma_{X} G^{\top}$. For the generalized eigenvalue-eigenvector pair $(\sigma_j, U_j)$, we have the relationship
\begin{equation} \label{eqn:eigProb_Yspace}
	G \Gamma_{X} G^{\top} U_j = \Gamma_{Y|X} U_j \sigma_j.
\end{equation}
Multiplying \cref{eqn:eigProb_Yspace} by $G^{\top} \Gamma_{Y|X}^{-1}$, we have
\begin{equation} \label{eqn:eigProb_Xspace_interim}
	G^{\top} \Gamma_{Y|X}^{-1} G \Gamma_{X} G^{\top} U_j = G^{\top} U_j \sigma_j.
\end{equation}
Define $\widetilde{V}_j\coloneqq \frac{1}{\alpha_j} \Gamma_{X} G^{\top} U_j$, where $\alpha_j$ is some scaling parameter to obtain the desired orthogonality, and rewrite \cref{eqn:eigProb_Xspace_interim} as
\begin{equation} \label{eqn:eigProb_Xspace}
	G^{\top} \Gamma_{Y|X}^{-1} G  \widetilde{V}_j = \Gamma_{X}^{-1} \widetilde{V}_j \sigma_j.
\end{equation}
 From \cref{eqn:eigProb_Xspace}, we claim $(\sigma_j,\widetilde{V}_j)$ is a generalized eigenvalue-eigenvector pair of  $(G^{\top} \Gamma_{Y|X}^{-1} G, \Gamma_{X}^{-1})$. Consequently, $(1+\sigma_j,\widetilde{V}_j)$ is the corresponding pair of  $(\Gamma_{X}^{-1}+G^{\top} \Gamma_{Y|X}^{-1} G, \Gamma_{X}^{-1}) \eqqcolon  (\Gamma_{X|Y}^{-1}, \Gamma_{X}^{-1}) $. By definition the definite pairs $(\Gamma_{X|Y}^{-1}, \Gamma_{X}^{-1}) $ and  $(\Gamma_{X},\Gamma_{X|Y})$ have the same spectrum. Further, if $(1+\sigma_j,\widetilde{V}_j)$ is a generalized eigenvalue-eigenvector pair of  $(\Gamma_{X|Y}^{-1}, \Gamma_{X}^{-1}) $, then  $(1+\sigma_j,\Gamma_{X}^{-1}\widetilde{V}_j)$ is the corresponding pair of $(\Gamma_{X},\Gamma_{X|Y})$. This completes the proof for  \cref{enumerate:prop_eigen_prob_equivalence_2}. We could have alternatively defined $\Gamma_{X|Y}^{-1}\widetilde{V}_j$ as the generalized eigenvector without any bearing on the spectrum. The choice we make is motivated by the orthogonality of eigenvectors we desire.
 
Let us define $V_j \coloneqq  \Gamma_{X}^{-1}\widetilde{V}_j$. Then \cref{enumerate:prop_eigen_prob_equivalence_1} follows by recognizing that $(\sigma_j,V_j)$ is the corresponding generalized eigenvalue-eigenvector pair of  $(\Gamma_{X}-\Gamma_{X|Y},\Gamma_{X|Y})$. 
\end{proof}


\begin{proposition}\label{prop:expSymKLDiv}
Let $X \in \mathbb{R}^{n}$ and $Y \in \mathbb{R}^{m}$ be jointly Gaussian random variables as defined in \Cref{subsubsec:lbip_notation_setup}. Let $\mathcal{P} \in \mathbb{R}^{m \times k }$ be a selection operator such that $Y_{\mathcal{P}} \coloneqq  \mathcal{P}^{\top}Y = \left[ Y_{i_1},\ldots,Y_{i_{k}} \right] ^{\top}$. The expected symmetrized Kullback--Leiber divergence, $\mathbb{E}_{\pi_{Y_{\mathcal{P}}}} \left[ D^{sym}_{\text{KL}} (\pi_{X|Y_{\mathcal{P}}}, \pi_{X} ) \right]$, between the prior  $\pi_{X}$ and the posterior $\pi_{X|Y_{\mathcal{P}}}$
 is equal to the sum of the generalized eigenvalues $\widehat{\sigma}_j$ of the definite pair $(\Gamma_{Y_{\mathcal{P}}} - \Gamma_{Y_{\mathcal{P}}|X},\Gamma_{Y_{\mathcal{P}}|X})$ or equivalently $(\Gamma_{X} - \Gamma_{X|Y_{\mathcal{P}} },\Gamma_{X|Y_{\mathcal{P}} })$, i.e., 
\begin{displaymath}
		\mathbb{E}_{\pi_{Y_{\mathcal{P}}}} \left[ D^{sym}_{\text{KL}} (\pi_{X|Y_{\mathcal{P}}}, \pi_{X} ) \right] \coloneqq  \mathbb{E}_{\pi_{Y_{\mathcal{P}}}} \left[ D_{\text{KL}}(\pi_{{X|Y_{\mathcal{P}} }} \| \pi_{{X}} ) + D_{\text{KL}}(\pi_{{X}} \| \pi_{{X|Y_{\mathcal{P}} }} )\right]  =  \sum_j   \widehat{\sigma}_{j}.
\end{displaymath}
\end{proposition}

\begin{proof}  {\bfseries \Cref{prop:expSymKLDiv} }  \label{proof:prop_expSymKLDiv}
Since the prior and posterior are normal distributions, we can write the expected symmetrized KL divergence analytically,
\begin{multline} \label{eqn:expSymDKL}
	\mathbb{E}_{\pi_{Y_{\mathcal{P}}}} \left[ D^{sym}_{\text{KL}} (\pi_{X|Y_{\mathcal{P}}}, \pi_{X} ) \right]
     = \frac{1}{2} \Bigl( \tr{\Gamma_{X}^{-1}\Gamma_{X|Y_{\mathcal{P}}}}   +	\mathbb{E} \left[ \mu_{X|Y_{\mathcal{P}}}^{\top }\Gamma_{X}^{-1}\mu_{X|Y_{\mathcal{P}}} \right]  \\
      + \tr{\Gamma_{X|Y_{\mathcal{P}}}^{-1}\Gamma_{X}} +\mathbb{E}\left[ \mu_{X|Y_{\mathcal{P}}}^{\top }\Gamma_{X|Y_{\mathcal{P}}}^{-1}\mu_{X|Y_{\mathcal{P}}} \right] - 2n  \Bigr).
\end{multline}
Above we have assumed without any loss of generality that the prior has zero mean. The posterior mean $\mu_{X|Y_{\mathcal{P}}} \coloneqq  \Gamma_{X|Y_{\mathcal{P}}} G^{\top}_{\mathcal{P}} \Gamma_{Y_{\mathcal{P}}|X}^{-1} Y_{\mathcal{P}}$ is a function of the actual realization of data, unlike the covariance operator $\Gamma_{X|Y_{\mathcal{P}}}$ which is independent of data. 

Simplifying the above expression is easier when we evaluate $\mathbb{E}_{\pi_{Y_{\mathcal{P}}}} \left[ \mu_{X|Y_{\mathcal{P}}} \mu_{X|Y_{\mathcal{P}}}^{\top } \right]$.
\begin{subequations}\begin{align}
	\mathbb{E}_{\pi_{Y_{\mathcal{P}}}} \left[ \mu_{X|Y_{\mathcal{P}}} \mu_{X|Y_{\mathcal{P}}}^{\top } \right]
	&= \Gamma_{X|Y_{\mathcal{P}}} G^{\top}_{\mathcal{P}} \Gamma_{Y_{\mathcal{P}}|X}^{-1}  \mathbb{E} \left[ Y_{\mathcal{P}} Y_{\mathcal{P}}^{\top} \right]  \Gamma_{Y_{\mathcal{P}}|X}^{-1} G_{\mathcal{P}} \Gamma_{X|Y_{\mathcal{P}}}, \\
	&\overset{(i)}{=}  \Gamma_{X|Y_{\mathcal{P}}} G^{\top}_{\mathcal{P}} \Gamma_{Y_{\mathcal{P}}|X}^{-1}  \lr{ \Gamma_{Y_{\mathcal{P}}|X}  + G_{\mathcal{P}} \Gamma_{X} G^{\top}_{\mathcal{P}}} \Gamma_{Y_{\mathcal{P}}|X}^{-1} G_{\mathcal{P}} \Gamma_{X|Y_{\mathcal{P}}}, \\
	&=  \Gamma_{X|Y_{\mathcal{P}}}  \lr{ \Gamma_{X}^{-1}  + G^{\top}_{\mathcal{P}} \Gamma_{Y_{\mathcal{P}}|X}^{-1} G_{\mathcal{P}}} \Gamma_{X} G^{\top}_{\mathcal{P}}\Gamma_{Y_{\mathcal{P}}|X}^{-1} G_{\mathcal{P}} \Gamma_{X|Y_{\mathcal{P}}}, \\
	&\overset{(ii)}{=}  \Gamma_{X|Y_{\mathcal{P}}}  \Gamma_{X|Y_{\mathcal{P}}}^{-1}  \Gamma_{X} \lr{ \Gamma_{X|Y_{\mathcal{P}}}^{-1} - \Gamma_{X} } \Gamma_{X|Y_{\mathcal{P}}}, \\
	&=   \Gamma_{X} - \Gamma_{X|Y_{\mathcal{P}}}.
\end{align} \end{subequations}
In the above set of equations, $(i)$ holds since $\Gamma_{Y_{\mathcal{P}}|X}  + G_{\mathcal{P}} \Gamma_{X} G^{\top}_{\mathcal{P}} = \mathbb{E}_{\pi_{Y_{\mathcal{P}}}} \left[ Y_{\mathcal{P}} Y_{\mathcal{P}}^{\top} \right] \eqqcolon  \Gamma_{Y_{\mathcal{P}}}$ is the marginal of the data, and $(ii)$ holds since $ \Gamma_{X}^{-1}  + G^{\top}_{\mathcal{P}} \Gamma_{Y_{\mathcal{P}}|X}^{-1} G_{\mathcal{P}} \eqqcolon  \Gamma_{X|Y_{\mathcal{P}}}^{-1}$ is the posterior precision.

Using the cyclic property of trace and linearity of the expectation operator it is now easy to see that,
\begin{align}
\label{eqn:expPostMeanPriorPrecision}\mathbb{E}_{\pi_{Y_{\mathcal{P}}}} \left[ \mu_{X|Y_{\mathcal{P}}}^{\top }\Gamma_{X}^{-1}\mu_{X|Y_{\mathcal{P}}} \right]
&= \tr{\Gamma_{X}^{-1} \lr{\Gamma_{X} - \Gamma_{X|Y_{\mathcal{P}}}}} = n - \tr{\Gamma_{X}^{-1}\Gamma_{X|Y_{\mathcal{P}}}} , \\
\label{eqn:expPostMeanPostPrecision} \mathbb{E}_{\pi_{Y_{\mathcal{P}}}} \left[ \mu_{X|Y_{\mathcal{P}}}^{\top }\Gamma_{X|Y_{\mathcal{P}}}^{-1}\mu_{X|Y_{\mathcal{P}}} \right]
&= \tr{\Gamma_{X|Y_{\mathcal{P}}}^{-1} \lr{\Gamma_{X} - \Gamma_{X|Y_{\mathcal{P}}}}} =  \tr{\Gamma_{X|Y_{\mathcal{P}}}^{-1} \Gamma_{X}} -n.
\end{align}

Using \cref{eqn:expPostMeanPriorPrecision} and \cref{eqn:expPostMeanPostPrecision} in \cref{eqn:expSymDKL} we have,
\begin{equation}
\mathbb{E}_{\pi_{Y_{\mathcal{P}}}} \left[ D^{sym}_{\text{KL}} (\pi_{X|Y_{\mathcal{P}}}, \pi_{X} ) \right]
=  \tr{\Gamma_{X|Y_{\mathcal{P}}}^{-1}\Gamma_{X}}  - n.
\end{equation}
The proof is completed by recognizing that $ \tr{\Gamma_{X|Y_{\mathcal{P}}}^{-1}\Gamma_{X}}$ is the sum of generalized eigenvalues of the definite pair $\lr{\Gamma_{X},\Gamma_{X|Y_{\mathcal{P}}}}$, which we know from \Cref{prop:eigen_prob_equivalence}  is $n + \sum_j \widehat{\sigma}_j $, where $\widehat{\sigma}_j$ are the generalized eigenvalues of the definite pair $(\Gamma_{X} - \Gamma_{X|Y_{\mathcal{P}} },\Gamma_{X|Y_{\mathcal{P}} })$ or equivalently $(\Gamma_{Y_{\mathcal{P}}} - \Gamma_{Y_{\mathcal{P}}|X},\Gamma_{Y_{\mathcal{P}}|X})$.

\end{proof}

\subsubsection{Bounds on sub/sup--modularity ratios}

We adapt an eigenvalue interlacing result for definite pairs which will be useful in proving \Cref{prop:bound_sup_sub_modularity_ratio}.
\begin{theorem}\textbf{\rm \bf (Cauchy Interlacing Theorem \cite{horn_johnson_2012,bhatia1997}, \cite[Theorem~2.3]{Kressner2014}, \cite[Theorem~2.1]{KovacStriko_Veselic_1995})}
\label{thm:cauchy_interlacing}
Let $A,B \in \mathbb{R}^{n \times n}$ with $A \succeq 0 ,B \succ 0$, and let $\gamma_1 \geq \gamma_2 \geq \cdots  \geq \gamma_{\widehat{n}} > 0, \ \widehat{n} \leq n$, be the eigenvalues of $(A,B)$. 
For any $Z \in \mathbb{R}^{n \times p}$, with  $p \leq n$ and full column-rank, let $\mu_1 \geq \mu_2 \geq \cdots  \geq \mu_{\widehat{p}} > 0, \ \widehat{p} \leq p,$ be the eigenvalues of $(Z^{\top}AZ,Z^{\top}BZ)$. Then:
\begin{displaymath}
\gamma_{n-p+k} 
\leq \mu_k 
\leq \gamma_k , \qquad k = 1,\ldots,\widehat{p}.
\end{displaymath}
If additionally  $A \succ 0$, then $\widehat{n} = n$ and $\widehat{p} = p$.
\end{theorem}

\begin{proof}  {\bfseries \Cref{prop:bound_sup_sub_modularity_ratio}} \label{proof:prop_bound_sup_sub_modularity_ratio}
In \cite{Bian_etal_2017}  the authors provide a bound on the submodularity ratio for the Bayesian A-optimality design criterion. Our arguments use that proof as a template, but are modified to account for the mutual information design criterion. In addition we seek a bound for the supermodularity ratio, which intriguingly is the same.

In order to  bound the \emph{submodularity} ratio we need to lower bound, 
\begin{displaymath}
	\frac{\sum_{\nu \in \mathscr{A} \setminus \mathscr{B}} \rho_{\nu}(\mathscr{B})}{\rho_{\mathscr{A}} (\mathscr{B})} = \frac{\sum_{\nu \in \mathscr{A} \setminus \mathscr{B}}  F(\nu \cup\mathscr{B}) -F(\mathscr{B})}{ F(\mathscr{A}\cup\mathscr{B}) -F(\mathscr{B})}, \qquad \forall \mathscr{A}, \mathscr{B} \subseteq \mathscr{V}.
\end{displaymath}
Similarly, in order to  bound the \emph{supermodularity} ratio we need to lower bound,
\begin{displaymath}
	\frac{\rho_{\mathscr{A}} (\mathscr{B})}{\sum_{\nu \in \mathscr{A} \setminus \mathscr{B}} \rho_{\nu}(\mathscr{B})} = \frac{ F(\mathscr{A}\cup\mathscr{B}) -F(\mathscr{B})}{\sum_{\nu \in \mathscr{A} \setminus \mathscr{B}}  F(\nu \cup\mathscr{B}) -F(\mathscr{B})}, \qquad \forall \mathscr{A}, \mathscr{B} \subseteq \mathscr{V}.
\end{displaymath}
It is clear that we can achieve the above tasks by finding lower and upper bounds for  $F(\mathscr{A}\cup\mathscr{B}) -F(\mathscr{B})$ and
$\sum_{\nu \in \mathscr{A} \setminus \mathscr{B}}  F(\nu \cup\mathscr{B}) -F(\mathscr{B})$ respectively. 

First we introduce some notation to aid the proof.  For any index set $\mathscr{A}$ with corresponding selection operator $\mathcal{P}_{\mathscr{A}}$, let $\zeta_j^{\mathscr{A}}$ represent the eigenvalue of the compressed pencil $\lr{\mathcal{P}_{\mathscr{A}}^{\top} \Gamma_Y \mathcal{P}_{\mathscr{A}}, \mathcal{P}_{\mathscr{A}}^{\top} \Gamma_{Y|X}\mathcal{P}_{\mathscr{A}}}$.
 $\zeta_j$ \emph{without any superscript} will continue to represent any eigenvalue of the pencil $\lr{\Gamma_Y, \Gamma_{Y|X}}$. Since we consider mutual information as the objective, $F(\mathscr{A})$ in terms of the eigenvalues $\zeta_j^{\mathscr{A}}$ is simply (\cref{eqn:mutualInfoGaussians_using_logDet,eqn:mutualInfoGaussians_reduced_data}),
\begin{equation}
	F(\mathscr{A}) = \frac{1}{2} \log \prod_{j=1}^{\abs{\mathscr{A}}} \zeta_j^{\mathscr{A}}.
\end{equation}

We now state a simple corollary of \Cref{thm:cauchy_interlacing}. We omit the proof since it is quite trivial. For any two index sets $\mathscr{A}, \mathscr{B}$  we have,
\begin{equation} \label{eqn:corr_cauchy_interlacing}
	\zeta_j^{\mathscr{A} \cup \mathscr{B}} \geq \zeta_j^{\mathscr{B}} \geq \zeta^{\mathscr{A}  \cup \mathscr{B}}_{\abs{\mathscr{A} \cup \mathscr{B}} - \abs{\mathscr{B}} + j}, \qquad j=1,\ldots,\abs{\mathscr{B}}.
\end{equation}
As in \Cref{thm:cauchy_interlacing}, in the above statement we have ordered the eigenvalues such that $\zeta_1^{\mathscr{B}} \geq \cdots \geq \zeta_{\abs{\mathscr{B}}}^{\mathscr{B}}$ and $\zeta_1^{\mathscr{A} \cup \mathscr{B}} \geq \cdots \geq \zeta_{\abs{\mathscr{A} \cup \mathscr{B}}}^{\mathscr{A}\cup \mathscr{B}}$. This simple result proves handy in bounding $ F(\mathscr{A} \cup \mathscr{B}) -F(\mathscr{B})$.
\begin{subequations} \label{eqn:prop_bound_sup_sub_modularity_ratio_helpful_bound1} \begin{align}
 F(\mathscr{A} \cup \mathscr{B}) -F(\mathscr{B})
 	&=  \frac{1}{2} \log \prod_{j=1}^{\abs{\mathscr{A} \cup \mathscr{B}}}  \zeta_j^{\mathscr{A}  \cup\mathscr{B}} - \frac{1}{2} \log \prod_{j=1}^{\abs{\mathscr{B}}} \zeta_j^{\mathscr{B}}, \\
 	&=  \frac{1}{2} \log \prod_{j=1}^{\abs{\mathscr{A} \cup \mathscr{B}} - \abs{\mathscr{B}}} \zeta_j^{\mathscr{A}  \cup\mathscr{B}} + \frac{1}{2} \log \prod_{j=1}^{\abs{\mathscr{B}}}  \frac{ \zeta_{\abs{\mathscr{A} \cup \mathscr{B}} - \abs{\mathscr{B}} + j}^{\mathscr{A}  \cup\mathscr{B}}}{\zeta_j^{\mathscr{B}}}, \\
 	&\leq \frac{1}{2}  \log \prod_{j=1}^{\abs{\mathscr{A} \setminus \mathscr{B}}} \zeta_j^{\mathscr{A}  \cup\mathscr{B}}, \quad \because \quad \text{\cref{eqn:corr_cauchy_interlacing}} \Rightarrow  \frac{ \zeta_{\abs{\mathscr{A} \cup \mathscr{B}} - \abs{\mathscr{B}} + j}^{\mathscr{A}  \cup\mathscr{B}}}{\zeta_j^{\mathscr{B}}} \leq 1, \\
 	&\leq \frac{1}{2} \log \lr{ \zeta_1^{\mathscr{A}  \cup\mathscr{B}} }^{\abs{\mathscr{A} \setminus \mathscr{B}}},  \quad \because \quad  \zeta_1^{\mathscr{A}  \cup\mathscr{B}} \geq  \zeta_j^{\mathscr{A}  \cup\mathscr{B}} \quad \forall j,   \\
 	&\leq \frac{\abs{\mathscr{A} \setminus \mathscr{B}}}{2} \log  \zeta_1^{\mathscr{A}  \cup\mathscr{B}}.
\end{align}  \end{subequations}
\begin{subequations} \label{eqn:prop_bound_sup_sub_modularity_ratio_helpful_bound2} \begin{align}
 F(\mathscr{A} \cup \mathscr{B}) -F(\mathscr{B})
 	&=  \frac{1}{2} \log \prod_{j=1}^{\abs{\mathscr{A} \cup \mathscr{B}} - \abs{\mathscr{B}}} \zeta_{\abs{\mathscr{B}} + j}^{\mathscr{A}  \cup\mathscr{B}} + \frac{1}{2} \log \prod_{j=1}^{\abs{\mathscr{B}}}  \frac{ \zeta_{j}^{\mathscr{A}  \cup\mathscr{B}}}{\zeta_j^{\mathscr{B}}}, \\
 	&\geq \frac{1}{2}  \log \prod_{j=1}^{\abs{\mathscr{A} \setminus \mathscr{B}}} \zeta_{\abs{\mathscr{B}} + j}^{\mathscr{A}  \cup\mathscr{B}}, \quad \because \quad \text{\cref{eqn:corr_cauchy_interlacing}} \Rightarrow   \frac{ \zeta_{j}^{\mathscr{A}  \cup\mathscr{B}}}{\zeta_j^{\mathscr{B}}} \geq 1,  \\
 	&\geq \frac{1}{2} \log \lr{ \zeta_{\abs{\mathscr{A} \cup \mathscr{B}}}^{\mathscr{A}  \cup\mathscr{B}}}^{\abs{\mathscr{A}  \setminus \mathscr{B}}},  \quad \because \quad \zeta^{\mathscr{A}  \cup\mathscr{B}}_{\abs{\mathscr{A} \cup \mathscr{B}}} \leq \zeta^{\mathscr{A}  \cup\mathscr{B}}_{\abs{\mathscr{B}}+j} \quad \forall j, \\
 	&\geq \frac{\abs{\mathscr{A} \setminus \mathscr{B}}}{2} \log  \zeta_{\abs{\mathscr{A} \cup \mathscr{B}}}^{\mathscr{A}  \cup\mathscr{B}}.
\end{align}  \end{subequations}

Specializing  \cref{eqn:corr_cauchy_interlacing} to the case when the set $\mathscr{A}$ is a singleton set consisting of the element $\nu$ we have the following result,
\begin{equation} \label{eqn:corr_cauchy_interlacing_special_case}
	\zeta_j^{\nu \cup \mathscr{B}} \geq \zeta_j^{\mathscr{B}} \geq \zeta^{\nu \cup \mathscr{B}}_{j+1}, \qquad j=1,\ldots,\abs{\mathscr{B}}.
\end{equation}
Using \cref{eqn:corr_cauchy_interlacing_special_case}  we can bound the term $\sum_{\nu \in \mathscr{A} \setminus \mathscr{B}}  F(\nu \cup\mathscr{B}) -F(\mathscr{B})$ using similar arguments as  in \cref{eqn:prop_bound_sup_sub_modularity_ratio_helpful_bound1,eqn:prop_bound_sup_sub_modularity_ratio_helpful_bound2}.
\begin{subequations} \label{eqn:prop_bound_sup_sub_modularity_ratio_helpful_bound3} \begin{align}
	\sum_{\nu \in \mathscr{A} \setminus \mathscr{B}}  F(\nu \cup\mathscr{B}) -F(\mathscr{B})
	&= \sum_{\nu \in \mathscr{A} \setminus \mathscr{B}} \lr{ \frac{1}{2} \log \prod_{j=1}^{\abs{\mathscr{B}}+1} \zeta_j^{\nu \cup\mathscr{B}} - \frac{1}{2} \log \prod_{j=1}^{\abs{\mathscr{B}}} \zeta_j^{\mathscr{B}} }, \\
	&= \sum_{\nu \in \mathscr{A} \setminus \mathscr{B}} \lr{ \frac{1}{2} \log  \zeta_{\abs{\mathscr{B}}+1}^{\nu \cup \mathscr{B}} +  \frac{1}{2} \log \prod_{j=1}^{\abs{\mathscr{B}}} \frac{\zeta_j^{\nu \cup\mathscr{B}}}{\zeta_j^{\mathscr{B}} } }, \\
	&\geq   \frac{\abs{\mathscr{A} \setminus \mathscr{B}}}{2}  \log  \zeta_{\abs{\mathscr{B}}+1}^{\nu^{\dagger} \cup \mathscr{B}}, \quad \because \quad  \frac{\zeta_j^{\nu \cup\mathscr{B}}}{\zeta_j^{\mathscr{B}}} \geq 1,
\end{align}  \end{subequations}
where $\nu^{\dagger} = \argmin_{\nu \in \mathscr{A} \setminus \mathscr{B}}  \zeta_{\abs{\mathscr{B}}+1}^{\nu \cup \mathscr{B}}$.
\begin{subequations}  \label{eqn:prop_bound_sup_sub_modularity_ratio_helpful_bound4} \begin{align}
	\sum_{\nu \in \mathscr{A} \setminus \mathscr{B}}  F(\nu \cup\mathscr{B}) -F(\mathscr{B})
	&= \sum_{\nu \in \mathscr{A} \setminus \mathscr{B}} \lr{ \frac{1}{2} \log  \zeta_{1}^{\nu \cup \mathscr{B}} +  \frac{1}{2} \log \prod_{j=1}^{\abs{\mathscr{B}}} \frac{\zeta_{j+1}^{\nu \cup\mathscr{B}}}{\zeta_j^{\mathscr{B}} } }, \\
	&\leq   \frac{\abs{\mathscr{A} \setminus \mathscr{B}}}{2}  \log  \zeta_{1}^{\nu^{\ddagger} \cup \mathscr{B}},  \quad \because \quad  \frac{\zeta_{j+1}^{\nu \cup\mathscr{B}}}{\zeta_j^{\mathscr{B}}} \leq 1,
\end{align}  \end{subequations}
where $\nu^{\ddagger} = \argmax_{\nu \in \mathscr{A} \setminus \mathscr{B}}  \zeta_{1}^{\nu \cup \mathscr{B}}$.
Gathering the results from \cref{eqn:prop_bound_sup_sub_modularity_ratio_helpful_bound1,eqn:prop_bound_sup_sub_modularity_ratio_helpful_bound2,eqn:prop_bound_sup_sub_modularity_ratio_helpful_bound3,eqn:prop_bound_sup_sub_modularity_ratio_helpful_bound4} we have,
\begin{alignat}{2}
\label{eqn:prop_bound_sup_sub_modularity_ratio_helpful_bound5}	 \frac{\abs{\mathscr{A} \setminus \mathscr{B}}}{2}  \log  \zeta_{\abs{\mathscr{B}}+1}^{\nu^{\dagger} \cup \mathscr{B}} 
	& 	\leq  \sum_{\nu \in \mathscr{A} \setminus \mathscr{B}}  F(\nu \cup\mathscr{B}) -F(\mathscr{B}) 
	&& \leq \frac{\abs{\mathscr{A} \setminus \mathscr{B}}}{2}  \log  \zeta_{1}^{\nu^{\ddagger} \cup \mathscr{B}},\\
\label{eqn:prop_bound_sup_sub_modularity_ratio_helpful_bound6}	 \frac{\abs{\mathscr{A} \setminus \mathscr{B}}}{2}  \log  \zeta_{\abs{\mathscr{A} \cup \mathscr{B}}}^{\mathscr{A}  \cup\mathscr{B}} 
	&  \leq \quad  F(\mathscr{A} \cup \mathscr{B}) -F(\mathscr{B}) 
	&& \leq \frac{\abs{\mathscr{A} \setminus \mathscr{B}}}{2}  \log  \zeta_1^{\mathscr{A}  \cup\mathscr{B}} .	
\end{alignat}

Using \cref{eqn:prop_bound_sup_sub_modularity_ratio_helpful_bound5,eqn:prop_bound_sup_sub_modularity_ratio_helpful_bound6} we have,
\begin{equation}
	 \frac{\sum_{\nu \in \mathscr{A} \setminus \mathscr{B}}  F(\nu \cup\mathscr{B}) -F(\mathscr{B})}{ F(\mathscr{A}\cup\mathscr{B}) -F(\mathscr{B})}
	 \geq \frac{ \log  \zeta_{\abs{\mathscr{B}}+1}^{\nu^{\dagger} \cup \mathscr{B}} }{\log  \zeta_1^{\mathscr{A}  \cup\mathscr{B}}} \geq \frac{\log \zeta_{\text{min}}}{\log \zeta_{\text{max}}}.
\end{equation}
\begin{equation}
	\frac{ F(\mathscr{A}\cup\mathscr{B}) -F(\mathscr{B})}{\sum_{\nu \in \mathscr{A} \setminus \mathscr{B}}  F(\nu \cup\mathscr{B}) -F(\mathscr{B})}
    \geq \frac{\log  \zeta_{\abs{\mathscr{A} \cup \mathscr{B}}}^{\mathscr{A}  \cup\mathscr{B}} }{\log  \zeta_{1}^{\nu^{\ddagger} \cup \mathscr{B}}} \geq  \frac{\log \zeta_{\text{min}}}{\log \zeta_{\text{max}}}.
\end{equation}
\end{proof}

\subsection{Proofs concerning MM algorithms for linear Bayesian experimental design} \label{appendix_subsec:lbip_mm}

This subsection is devoted to proving \Cref{prop:MBound_MI} and its corollaries.  We begin by stating certain standard results in linear algebra, subsequently deducing useful corollaries, and some helpful lemmas from them, all of which contribute to the argument of the main result.

\begin{theorem} [Hadamard's inequality \cite{horn_johnson_2012}] \label{thm:hadamard_inequality} Let $A \in \mathbb{R}^{m \times m}$ be any positive definite Hermitian matrix, with $a_{ij}$ representing its individual entries. Then,
\begin{displaymath}
	\det A \leq a_{11} \cdots a_{mm}
\end{displaymath}
\end{theorem}
\begin{corollary} \label{corr:logDetInequality_hadamard}  Let $A \in \mathbb{R}^{m \times m}$ be any positive definite Hermitian matrix, and $\mathcal{P} \in  \mathbb{R}^{m \times k}, k < m$ be a selection operator (\Cref{def:selectionOperator}), then:
\begin{displaymath}
	\log \lr{ \det \lr{ \mathcal{P}^{\top} A \mathcal{P}}} \leq  \tr{\mathcal{P}^{\top} \log \lr{\mathrm{diag}\lr{A}}  \mathcal{P}}
\end{displaymath}
\end{corollary}
\begin{proof_without_of_suffix}
The corollary follows immediately from \Cref{thm:hadamard_inequality}.
\end{proof_without_of_suffix}
\begin{remark} The $\log \det$ of a principal submatrix is a submodular function with respect to the indices defining the submatrices \cite{Gantmacher_Krein_1960,Kotelyanskii_1950,Fan_1967,Fan_1968,Kelmans_1983,Johnson_1985}. We know that submodularity implies subadditivity for nonnegative functions, thus the modular upper bound in the statement of \Cref{corr:logDetInequality_hadamard} can be understood in that sense as well.
\end{remark}


\begin{theorem} [Operator concave inequality \cite{Chandler_1957,Hansen1980,HansenPedersen1982,hansen_pedersen_2003,Bourin_2006}] \label{thm:operator_concave_inequality}
Let $A$ be any Hermitian operator defined on a Hilbert space, and $K$ be an isometry on a subspace of the Hilbert space. Denote    the compression of $A$ by $K$ as $K^{*} A K$, then for any operator concave function $\phi(\cdot)$ we have that,
\begin{displaymath}
 K^{*} \phi(A) K	 \preceq \phi (K^{*} A K)
\end{displaymath}
\end{theorem}


\begin{theorem} [Matrix versions of the Cauchy and Kantorovich inequalities \cite{Marshall_Olkin_1990,MondPecaric1994,Bourin_2005}] \label{thm:Cauchy_Kantorovich_inequality} Let $A \in \mathbb{R}^{m \times m}$ be any positive definite Hermitian matrix, with eigenvalues contained in the interval $[\lambda_{s},\lambda_{l}] \subset \mathbb{R}_{> 0}$. If $K$ is an isometry then:
\begin{displaymath}
\frac{4 \lambda_{s} \lambda_{l}}{\lr{\lambda_{s} + \lambda_{l}}^{2}}  K^{*} A^{-1} K
\preceq \lr{K^{*} A K}^{-1} 
\preceq  K^{*} A^{-1} K  
\end{displaymath}
\end{theorem}
As pointed out in \cite{Marshall_Olkin_1990}, if $K$ is merely a selection operator, then the upper bound of \Cref{thm:Cauchy_Kantorovich_inequality} reduces to a standard result in linear algebra, first established in \cite{Chollet_1982} and now a common theorem in texts \cite[Theorem 7.7.15.]{horn_johnson_2012}; the inverse of a principal submatrix of any definite matrix is less than or equal to the corresponding submatrix of the inverse.


\begin{corollary} \label{corr:logDetInequality} Let $A \in \mathbb{R}^{m \times m}$ be any positive definite Hermitian matrix with an eigenvalue interval $[\lambda_{s},\lambda_{l}]$ as prescribed in \Cref{thm:Cauchy_Kantorovich_inequality}. Define the scalar parameter $\varrho \coloneqq  \frac{4 \lambda_{s} \lambda_{l}}{\lr{\lambda_{s} + \lambda_{l}}^{2}}$. If $K \in  \mathbb{R}^{m \times k}, k < m$ is an isometry, $K^{*}K = I_{k}$, then:
\begin{displaymath}
 \tr{ K^{*} \log \lr{ A} K} \leq \log \det \lr{K^{*} A K} \leq  \tr{ K^{*} \lr{ \log \lr{A} - I_{m}\log\lr{\varrho} } K}
\end{displaymath}
\end{corollary}
\begin{proof} {\bfseries \Cref{corr:logDetInequality}} \label{proof:corr_logDetInequality}
Since $A \succ 0$, its compression by $K$, $K^{*} A K \succ 0$. Hence we have the identity $\tr{\log \lr{K^{*} A K}} =  \log \det \lr{K^{*} A K}$. Since $\log$ is a concave function, we claim using \Cref{thm:operator_concave_inequality} $\log \lr{K^{*} A K} \succeq K^{*} \log \lr{ A} K$, and this implies the lower bound.

To prove the upper bound, we begin by using the identity
\begin{equation}
	\log \det \lr{K^{*} A K} = -\log \det \lr{K^{*} A K}^{-1}  = -\tr{ \log  \lr{K^{*} A K}^{-1}}.
\end{equation}
Using the lower bound in \Cref{thm:Cauchy_Kantorovich_inequality} we have,
\begin{align} 
 \log \det \lr{ \varrho K^{*} A^{-1} K} &\leq \log \det\lr{K^{*} A K}^{-1}.\\
  \label{eqn:upper_bound_proof_1}
\Rightarrow	-\tr{\log  \lr{K^{*} A K}^{-1}}   &\leq     -\tr{\log \lr{ K^{*} A^{-1} K}} - \tr{I_{k} \log\lr{\varrho}}.
\end{align}
 $ K^{*} A^{-1} K$ is a compression by $K$ of $A^{-1} \succ 0$. Using  \Cref{thm:operator_concave_inequality} we claim,
 \begin{equation}
 			 \log \lr{K^{*} A^{-1} K}  \succeq K^{*} \log \lr{A^{-1}} K = -K^{*} \log \lr{A} K.
 \end{equation}
	\begin{equation} \label{eqn:upper_bound_proof_2}
\Rightarrow -\tr{\log \lr{K^{*} A^{-1} K}}  \leq \tr{K^{*} \log \lr{A} K}.	
	\end{equation}
Using \cref{eqn:upper_bound_proof_2} in \cref{eqn:upper_bound_proof_1} we have,
 \begin{equation}
	-\tr{\log  \lr{K^{*} A K}^{-1}}   \leq     \tr{K^{*} \log \lr{A} K} - \tr{I_{k} \log\lr{\varrho}}.
 \end{equation}
The upper bound in the corollary statement is now immediately obtained by  linearity of the trace operator and the isometry property of $K$.
\end{proof}


\begin{proof} {\bfseries \Cref{prop:MBound_MI}} \label{proof:prop_MBound_MI}
Equation \cref{eqn:mutualInfoGaussians_reduced_data} provides an expression for the mutual information $\mathcal{I}(  X;Y_{\mathcal{P}})$,
\begin{displaymath}
	   \mathcal{I}(  X; Y_{\mathcal{P}})  =  \frac{1}{2} \log \lr{ \frac{\det \lr{ \mathcal{P}^{\top} \Gamma_{Y}  \mathcal{P}}}{\det \lr{ \mathcal{P}^{\top}\Gamma_{Y|X}  \mathcal{P}}} }.
\end{displaymath}
The proposition statement and the alternative modular bounds in \cref{eqn:alt_MBound_MI} are easily obtained by using \Cref{corr:logDetInequality_hadamard,corr:logDetInequality} on the two $\log\det$ terms.
\end{proof}


\begin{proof} {\bfseries \Cref{corr:MI_deviation_from_SM}} \label{proof:corr_MI_deviation_from_SM}
\Cref{eqn:corr_MI_deviation_from_SM_eq1} is simply a restatement of the lower bound in  \Cref{prop:MBound_MI}. The equivalence between \cref{eqn:corr_MI_deviation_from_SM_eq1} and \cref{eqn:corr_MI_deviation_from_SM_eq2} is trivial. Rewriting \cref{eqn:corr_MI_deviation_from_SM_eq2} with $Y \setminus Y_\mathcal{P}$ as the argument, and subtracting that term from $\mathcal{I}(  X;Y)$ we obtain  \cref{eqn:corr_MI_deviation_from_SM_eq1}. In a similar way we can obtain  \cref{eqn:corr_MI_deviation_from_SM_eq2} from  \cref{eqn:corr_MI_deviation_from_SM_eq1}.
\end{proof}


\begin{proof} {\bfseries \Cref{corr:MI_SM_case_bounding_BP} } \label{proof:corr_MI_SM_case_bounding_BP}
If $Y_{i_k} | X$ are independent, then we know $\Gamma_{Y|X}$ is diagonal, and hence $\log \mathrm{diag} \lr{\Gamma_{Y|X}} = \log \lr{\Gamma_{Y|X}}$. Thus the corollary statement is apparent as a consequence  of \Cref{corr:MI_deviation_from_SM}.
\end{proof}


\begin{proof} {\bfseries \Cref{prop:min_MI_loss_upper_bound}} \label{proof:prop_min_MI_loss_upper_bound}
Observe that
\begin{equation}
		\mathcal{I}(X;Y \setminus Y_{\mathcal{P}} |Y_{\mathcal{P}_1}) - \mathcal{I}(X;Y \setminus Y_{\mathcal{P}}, Y_{\mathcal{P}_1}) 
	=
	\mathcal{I}(Y_{\mathcal{P}_1}; Y \setminus Y_{\mathcal{P}},Y_{\mathcal{P}_1} |X) - \mathcal{I}(Y_{\mathcal{P}_1};Y \setminus Y_{\mathcal{P}},Y_{\mathcal{P}_1}).
\end{equation}
This implies
\begin{equation} \label{eqn:proof_prop_min_MI_loss_upper_bound_1}
   \mathcal{I}(X;Y \setminus Y_{\mathcal{P}} |Y_{\mathcal{P}_1}) 	-\mathcal{I}(X;Y \setminus Y_{\mathcal{P}}, Y_{\mathcal{P}_1})  \leq  \mathcal{I}(Y_{\mathcal{P}_1}; Y \setminus Y_{\mathcal{P}},Y_{\mathcal{P}_1} |X).
\end{equation}
Using \cref{eqn:proof_prop_min_MI_loss_upper_bound_1}, we have
\begin{equation}
 \mathcal{I}(X;Y|Y_{\mathcal{P}_1}) - \mathcal{I}(X;Y \setminus Y_{\mathcal{P}_1}, Y_{\mathcal{P}})  \leq 
  \mathcal{I}(X;Y|Y_{\mathcal{P}_1}) - \mathcal{I}(X;Y \setminus Y_{\mathcal{P}} | Y_{\mathcal{P}_1}) + \mathcal{I}(Y_{\mathcal{P}_1}; Y \setminus Y_{\mathcal{P}},Y_{\mathcal{P}_1} |X).
\end{equation}
Consider the term $\mathcal{I}(X;Y|Y_{\mathcal{P}_1}) - \mathcal{I}(X;Y \setminus Y_{\mathcal{P}} | Y_{\mathcal{P}_1})$; adapting the result in \Cref{corr:MI_deviation_from_SM} we have the following bound:
\begin{align} \label{eqn:proof_prop_min_MI_loss_upper_bound_2}
	  \mathcal{I}(X;Y|Y_{\mathcal{P}_1}) - \mathcal{I}(X;Y \setminus Y_{\mathcal{P}} | Y_{\mathcal{P}_1}) 
\nonumber	  & \leq \tr{ \widehat{\mathcal{P}}^{\top} \lr{ \log \lr{\Gamma_{Y|Y_{\mathcal{P}_1}}} - \log \lr{\Gamma_{Y|X,Y_{\mathcal{P}_1}}} }\widehat{\mathcal{P}}} \\
& \quad +  \tr{ \widehat{\mathcal{P}}^{c}{}^{\top} \lr{ \log \mathrm{diag} \lr{\Gamma_{Y|X,Y_{\mathcal{P}_1}}} - \log \lr{\Gamma_{Y|X,Y_{\mathcal{P}_1}}} }\widehat{\mathcal{P}}^{c}} .
\end{align}
Now consider the term $ \mathcal{I}(Y_{\mathcal{P}_1}; Y \setminus Y_{\mathcal{P}},Y_{\mathcal{P}_1} |X)$. Using \Cref{corr:logDetInequality_hadamard,corr:logDetInequality} we have the following bound:
\begin{align} \label{eqn:proof_prop_min_MI_loss_upper_bound_3}
 \mathcal{I}(Y_{\mathcal{P}_1}; Y \setminus Y_{\mathcal{P}},Y_{\mathcal{P}_1} |X) \leq  \tr{ \widehat{\mathcal{P}}^{c}{}^{\top} \lr{ \log \mathrm{diag} \lr{\Gamma_{Y \setminus Y_{\mathcal{P}_1}|X}} - \log \lr{\Gamma_{Y|X,Y_{\mathcal{P}_1}}} }\widehat{\mathcal{P}}^{c}}.
\end{align}
Adding \cref{eqn:proof_prop_min_MI_loss_upper_bound_1,eqn:proof_prop_min_MI_loss_upper_bound_2}, we get the proposition statement.
\end{proof}

\section{Maximizing information gain versus minimizing information loss} \label{subsec:max_info_gain_vs_min_info_loss}
When solving the design problem using sequential algorithms, is maximizing information gain necessarily  better than minimizing information loss? Are there circumstances when one ought to be preferred over the other?

These are natural questions that arise when contrasting the two approaches. Here we discuss nuances of the two formulations, using the Bayesian experimental design problem of \Cref{subsec:BIP_and_exp_design} as a template. Recall that the goal of this design problem is to select a subset of observations $Y_{\mathcal{P}}$ that best informs the parameters $X$.
The fundamental difference between the two approaches is most evident at the onset of each procedure.
When maximizing information gain, we first compare independent contributions given by $\mathcal{I}\lr{X;Y_i}$. When minimizing information loss, on the other hand, we first assess the relative magnitudes of
 $\mathcal{I}\lr{X;Y} - \mathcal{I}\lr{X;Y \setminus Y_i}$; here the amount of collective information provided by the observations $Y \setminus Y_i$ is the guiding factor. In the latter case, note how the relative importance of each observation $Y_i$ is dictated by its absence, in contrast to the former.
These differences suggest two ideas:
 \begin{enumerate}[\itshape i)\upshape] 
 \item If the dominating contributor to the total mutual information  $\mathcal{I}\lr{X;Y}$ is the aggregate of  \emph{independent} contributions $\mathcal{I}\lr{X;Y_i}$, then maximizing information gain seems the better approach.
 \item Conversely, if interaction among the observations $Y_i$ is the dominating contributor to $\mathcal{I}\lr{X;Y}$, then minimizing information loss would appear as a better choice.
 \end{enumerate}

The above remarks are not entirely new in and of themselves. Several authors have discussed similar notions in the context of seeking symmetric multivariate correlation measures \cite{Watanabe_1960,Han_1978}, and analyzing the polymatroidal dependence structure of a set of random variables \cite{Han_1975,Fujishige_1978}.

If a subset of observations has already been selected/discarded, one would simply redefine the terms by appropriately conditioning on that set. The main caveat of course is the size of the conditioning set. If a substantial number of observations have already been chosen, then the strategy of maximizing information gain also implicitly accounts for any favorable interactions among the observations. Thus the desired cardinality constraint and any prior knowledge about the strength of the interaction can help in the choice of strategy.

From an optimization perspective, there are clear distinctions between the two approaches.
Maximizing information gain corresponds to maximizing a non-decreasing set function, and when the observations are conditionally independent the function is submodular. 
Minimizing information loss corresponds to minimizing a non-decreasing set function, and when the observations are conditionally independent the function is supermodular. 
The constrained maximization of non-decreasing set functions is well understood theoretically, especially when the function is submodular. Results for the minimization of non-decreasing set functions, however, are fewer and less satisfactory. Even in the restricted setting of the non-decreasing set function being supermodular, the analysis of the greedy algorithm is more involved \cite{Ilev_2001_backward}, and the approximation guarantee is given using \emph{steepness} of the function. For the more general case of minimizing a set function which is not necessarily sub/super-modular, there are no approximation guarantees of which we are aware.
Note that in \cite{Ilev_2001_backward} the result was given for any non-increasing supermodular function. We are dealing with a non-decreasing function since the function argument ($Y \setminus Y_{\mathcal{P}}$) involves the complement of the index set, $\mathscr{I}_{\mathcal{P}}$, that we seek. The result in \cite{Ilev_2001_backward} is then applicable with suitable modifications.

\section{Additional numerical results} \label{appendix:extra_discussion_and_num_results}

\subsection{Structured random example with i.i.d.\ observation error} \label{appendix:num_results_iid_error}

Consider the linear inverse problem as described in  \Cref{subsec:num_results_random}, but with  i.i.d.\  observation error. In \Cref{fig:spectrum_iid_obs_error,fig:batch_effect_subplot_iid_obs_error,fig:compare_algorithms_4subplot_iid_obs_error,fig:random_sampling_error_4x3subplot_iid_obs_error,fig:diff_MI_random_sampling_error_4x2subplot_iid_obs_error}
we present numerical investigations analogous to those in \Cref{subsec:num_results_random}.

\begin{figure}[h!] 	
  	\centering \includegraphics[width=\textwidth,angle=0]{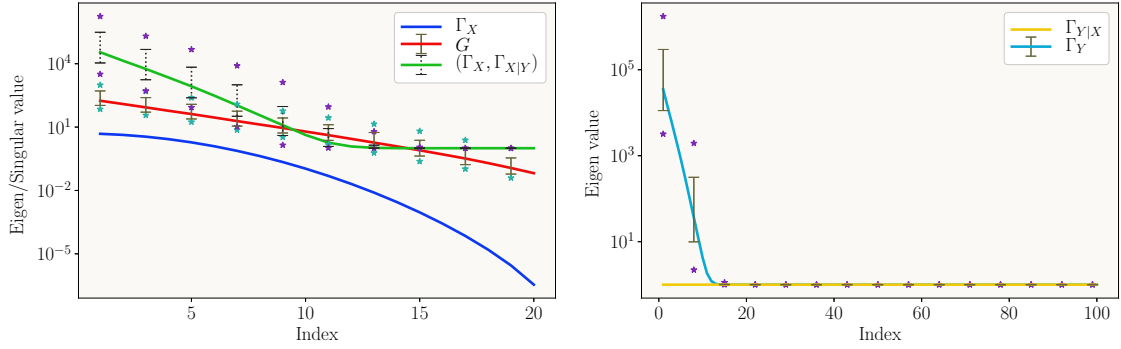}
  	 \caption{Spectrum of the relevant operators of the inverse problem with i.i.d.\  observation error. The solid line is the median across 1000 random instances of the forward model. The whiskers capture the interquantile range (10\% to 90\%),  and the $\star$ marks the maximum and minimum eigen/singular value. The prior and observation error covariances are fixed and not random.}
  	\label{fig:spectrum_iid_obs_error}
\end{figure}

\begin{figure}[h!] 	
  	\centering \includegraphics[width=\textwidth,angle=0]{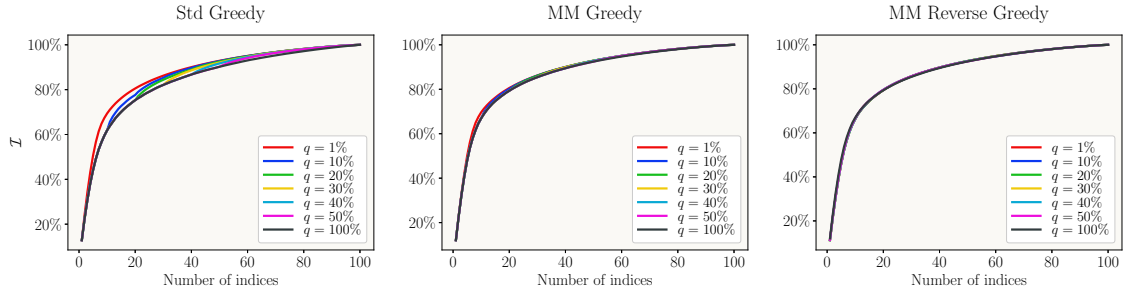}
	\caption{Performance of each greedy heuristic for different batch sizes for the case of inverse problem with i.i.d.\  observation error. The batch size ranges from single index selection to one shot approach. The solid line is the median across 1000 random instances for the forward model.}
  	\label{fig:batch_effect_subplot_iid_obs_error}
\end{figure}

\begin{figure}[h!] 	
  	\centering \includegraphics[width=\textwidth,angle=0]{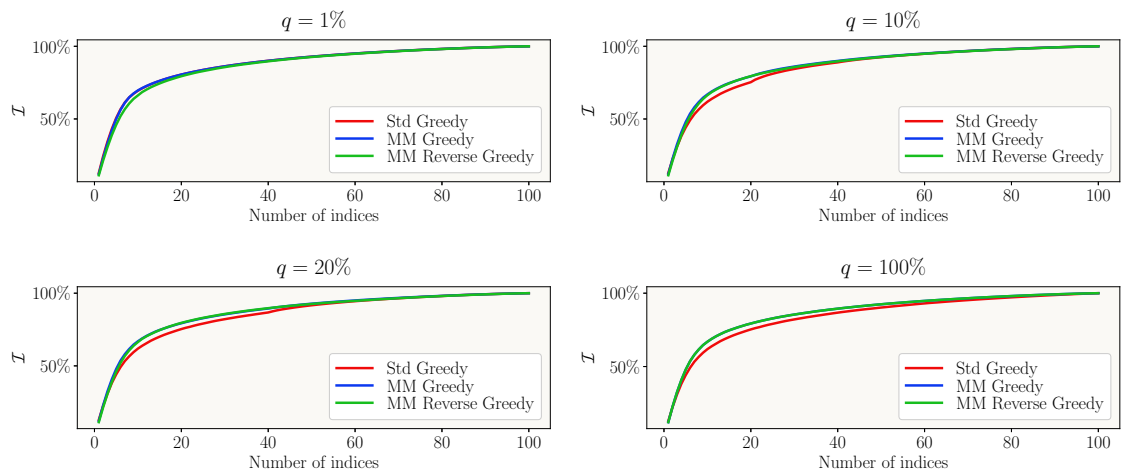}
	\caption{A comparative study of the greedy heuristics for four different batch sizes for the case of inverse problem with i.i.d.\  observation error. The solid line is the median across 1000 random instances for the forward model.}
  	\label{fig:compare_algorithms_4subplot_iid_obs_error}
\end{figure}

\begin{figure}[h!] 	
  	\centering \includegraphics[width=\textwidth,angle=0]{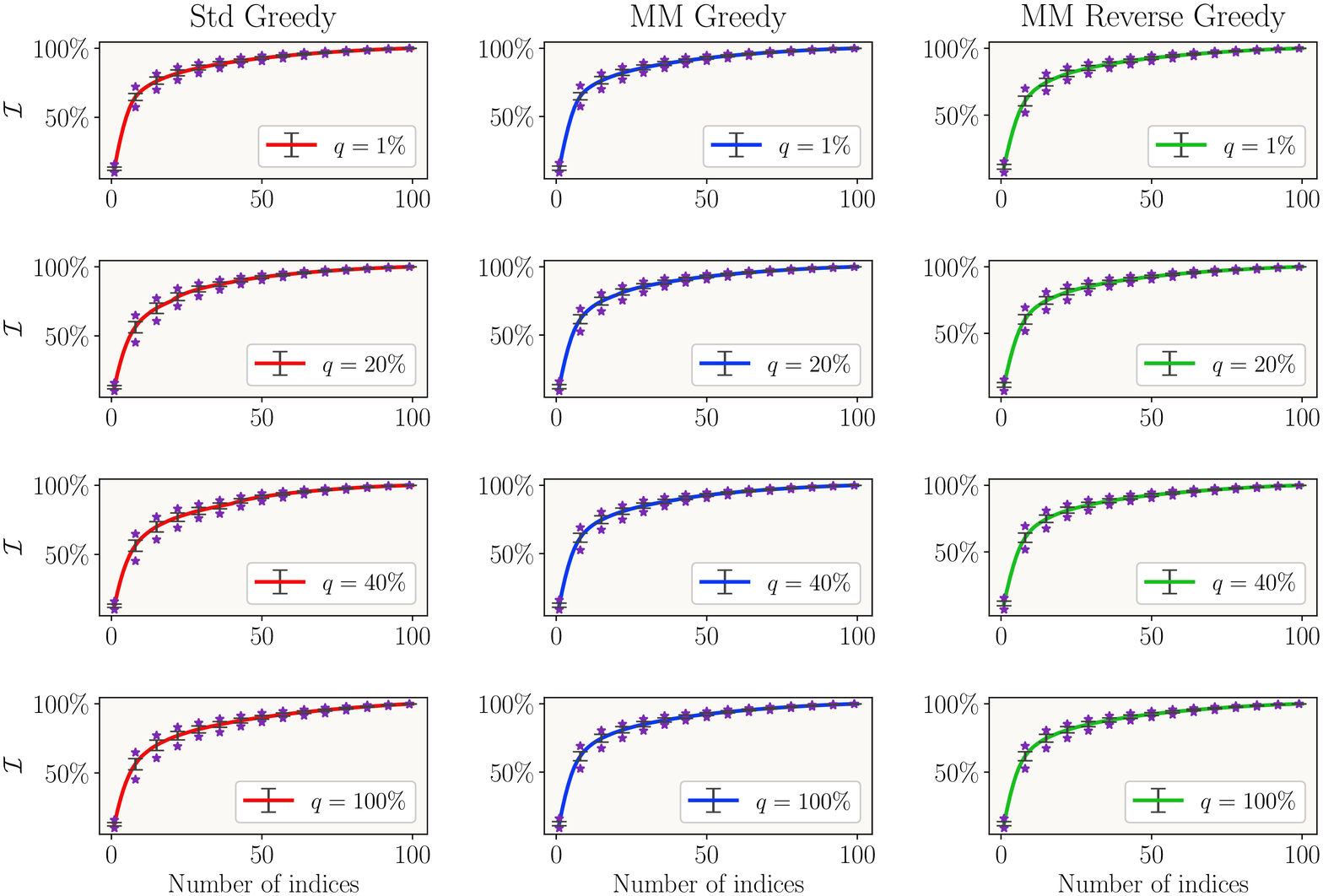}
	\caption{Performance of the greedy heuristics for different batch sizes across all 1000 random instances of the forward model for the case of inverse problem with i.i.d.\  observation error. The solid line is the median; the whiskers capture the interquantile range (10\% to 90\%),  and the $\star$ marks the maximum and minimum mutual information captured.}
  	\label{fig:random_sampling_error_4x3subplot_iid_obs_error}
\end{figure}

\begin{figure}[h!] 	
  	\centering \includegraphics[width=\textwidth,angle=0]{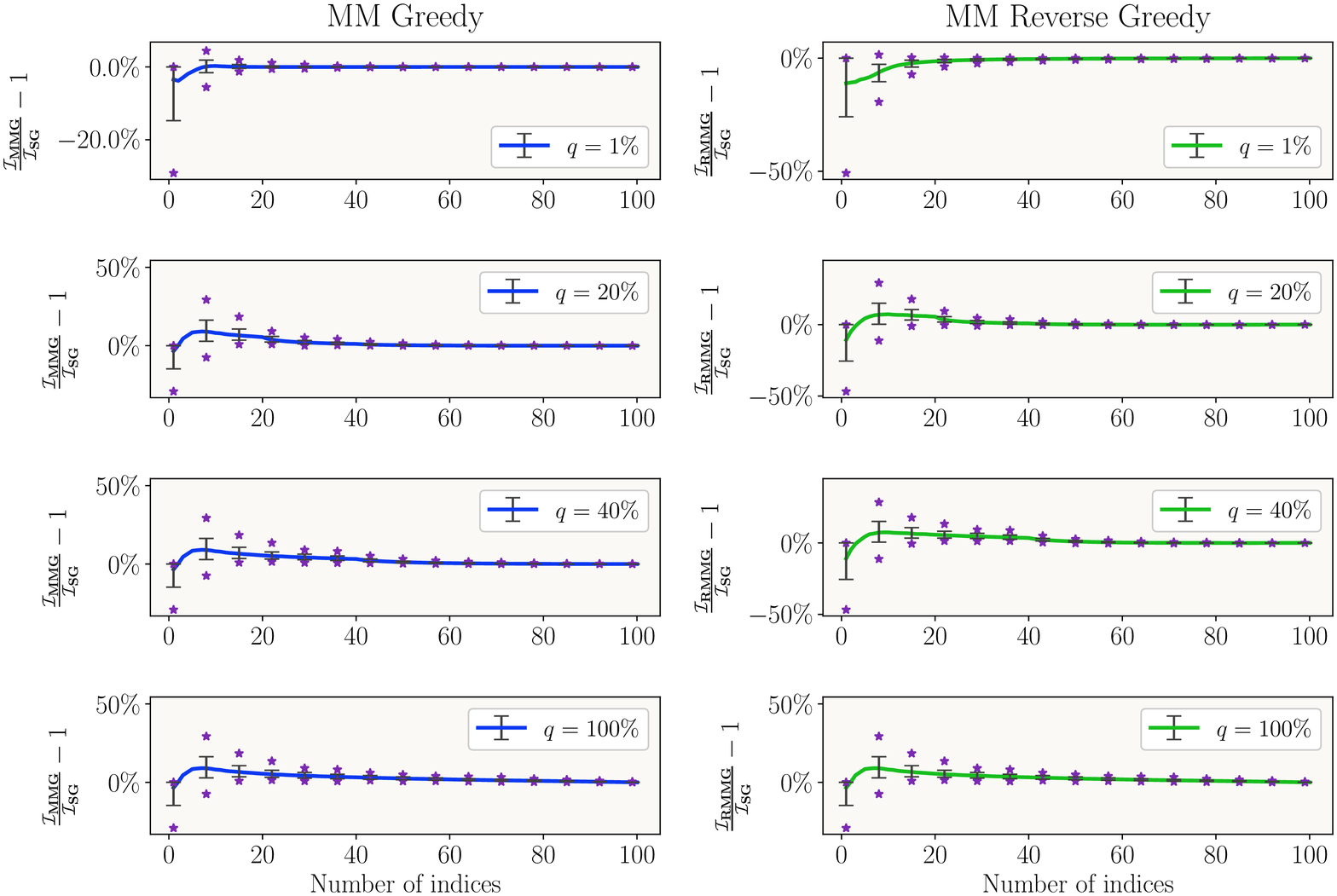}
  	 \caption{MM based batch greedy approaches compared against the standard batch greedy heuristic for different batch sizes across all 1000 random instances of the forward model for the case of inverse problem with i.i.d.\  observation error. The solid line is the median; the whiskers capture the interquantile range (10\% to 90\%),  and the $\star$ marks the maximum and minimum mutual information captured.}
  	\label{fig:diff_MI_random_sampling_error_4x2subplot_iid_obs_error}
\end{figure}

\subsection{Structured random example with exponential prior covariance} \label{appendix:num_results_exp_prior}

Consider the linear inverse problem as described in  \Cref{subsec:num_results_random}, but with exponential kernel used to define the prior on the inference parameters. In \Cref{fig:spectrum_corr_obs_error_exp_prior,fig:batch_effect_subplot_corr_obs_error_exp_prior,fig:compare_algorithms_4subplot_corr_obs_error_exp_prior,fig:random_sampling_error_4x3subplot_corr_obs_error_exp_prior,fig:diff_MI_random_sampling_error_4x2subplot_corr_obs_error_exp_prior}
we present numerical investigations analogous to those in \Cref{subsec:num_results_random}.

\begin{figure}[h!] 	
  	\centering \includegraphics[width=\textwidth,angle=0]{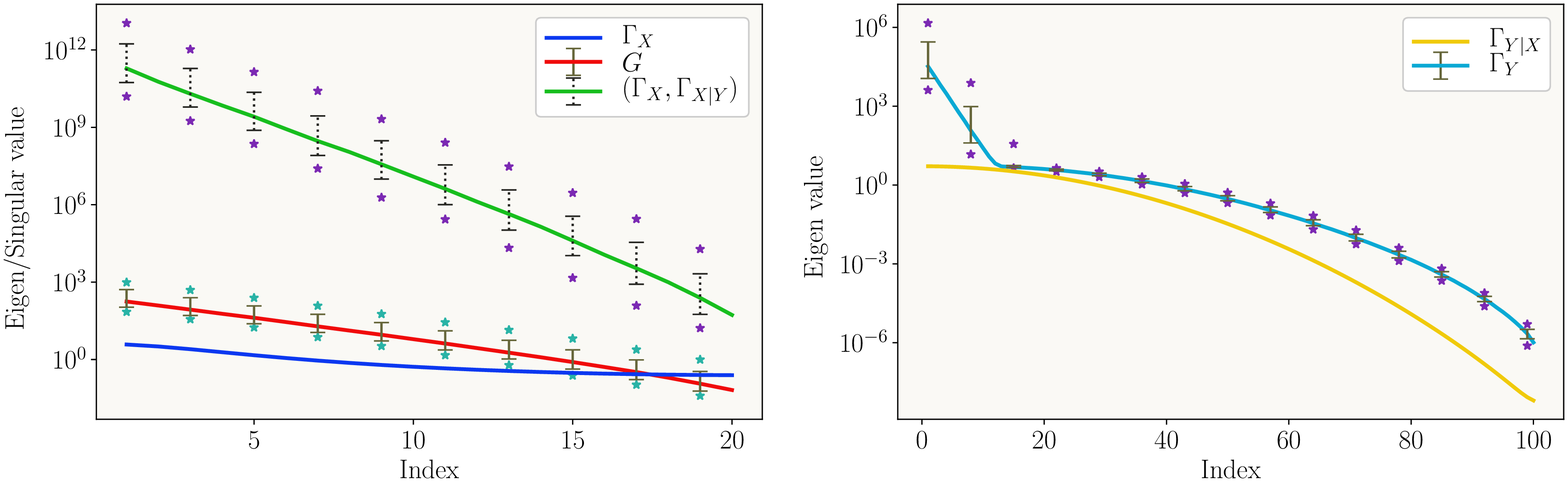}
  	 \caption{Spectrum of the relevant operators of the inverse problem with exponential prior covariance. The solid line is the median across 1000 random instances of the forward model. The whiskers capture the interquantile range (10\% to 90\%),  and the $\star$ marks the maximum and minimum eigen/singular value. The prior and observation error covariances are fixed and not random.}
  	\label{fig:spectrum_corr_obs_error_exp_prior}
\end{figure}

\begin{figure}[h!] 	
  	\centering \includegraphics[width=\textwidth,angle=0]{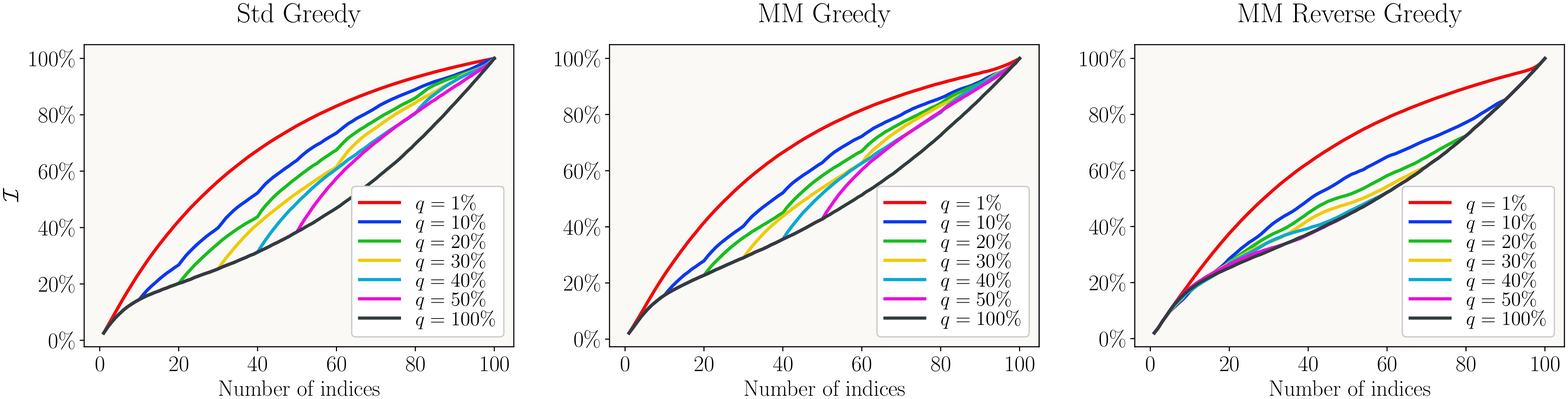}
	\caption{Performance of each greedy heuristic for different batch sizes for the case of inverse problem with exponential prior covariance. The batch size ranges from single index selection to one shot approach. The solid line is the median across 1000 random instances for the forward model.}
  	\label{fig:batch_effect_subplot_corr_obs_error_exp_prior}
\end{figure}

\begin{figure}[h!] 	
  	\centering \includegraphics[width=\textwidth,angle=0]{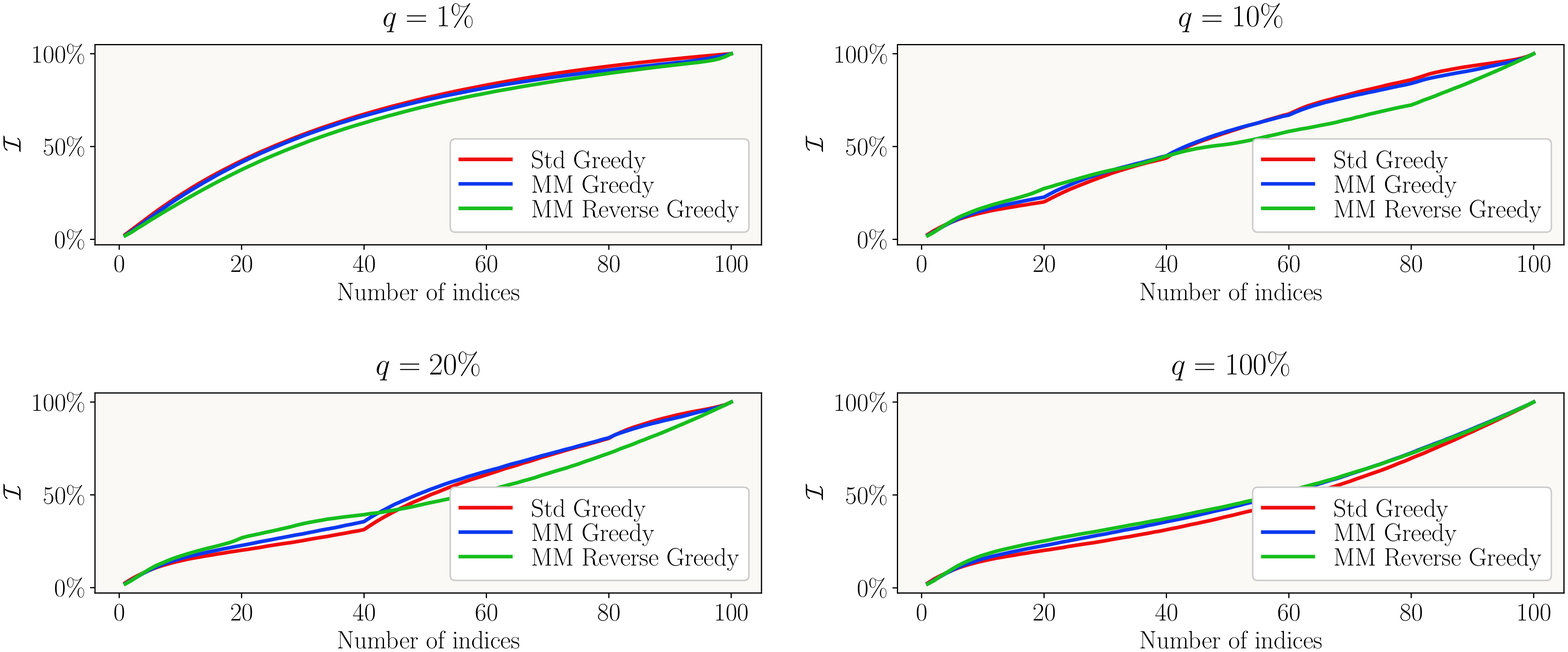}
	\caption{A comparative study of the greedy heuristics for four different batch sizes for the case of inverse problem with exponential prior covariance. The solid line is the median across 1000 random instances for the forward model.}
  	\label{fig:compare_algorithms_4subplot_corr_obs_error_exp_prior}
\end{figure}

\begin{figure}[h!] 	
  	\centering \includegraphics[width=\textwidth,angle=0]{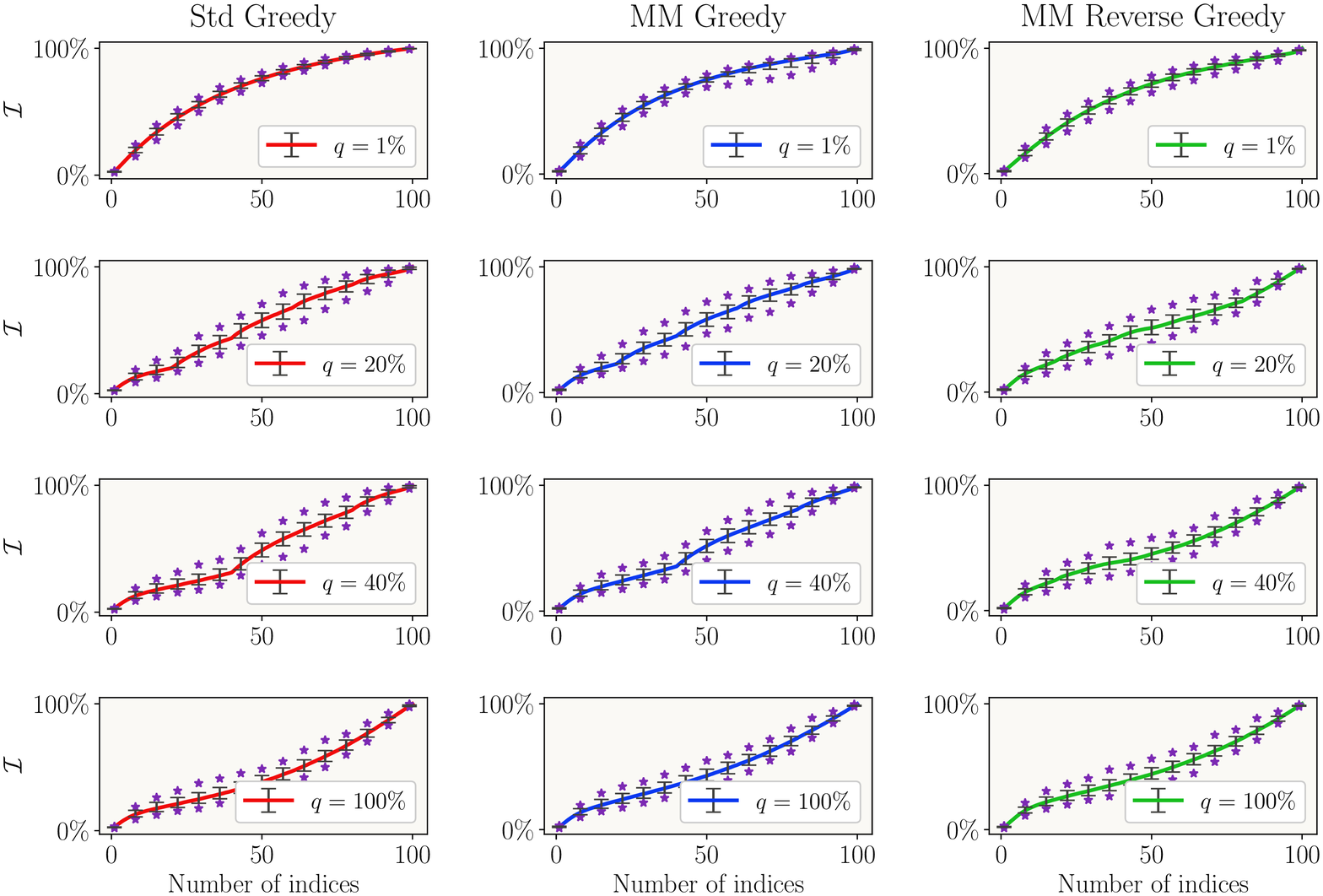}
	\caption{Performance of the greedy heuristics for different batch sizes across all 1000 random instances of the forward model for the case of inverse problem with exponential prior covariance. The solid line is the median; the whiskers capture the interquantile range (10\% to 90\%),  and the $\star$ marks the maximum and minimum mutual information captured.}
  	\label{fig:random_sampling_error_4x3subplot_corr_obs_error_exp_prior}
\end{figure}

\begin{figure}[h!] 	
  	\centering \includegraphics[width=\textwidth,angle=0]{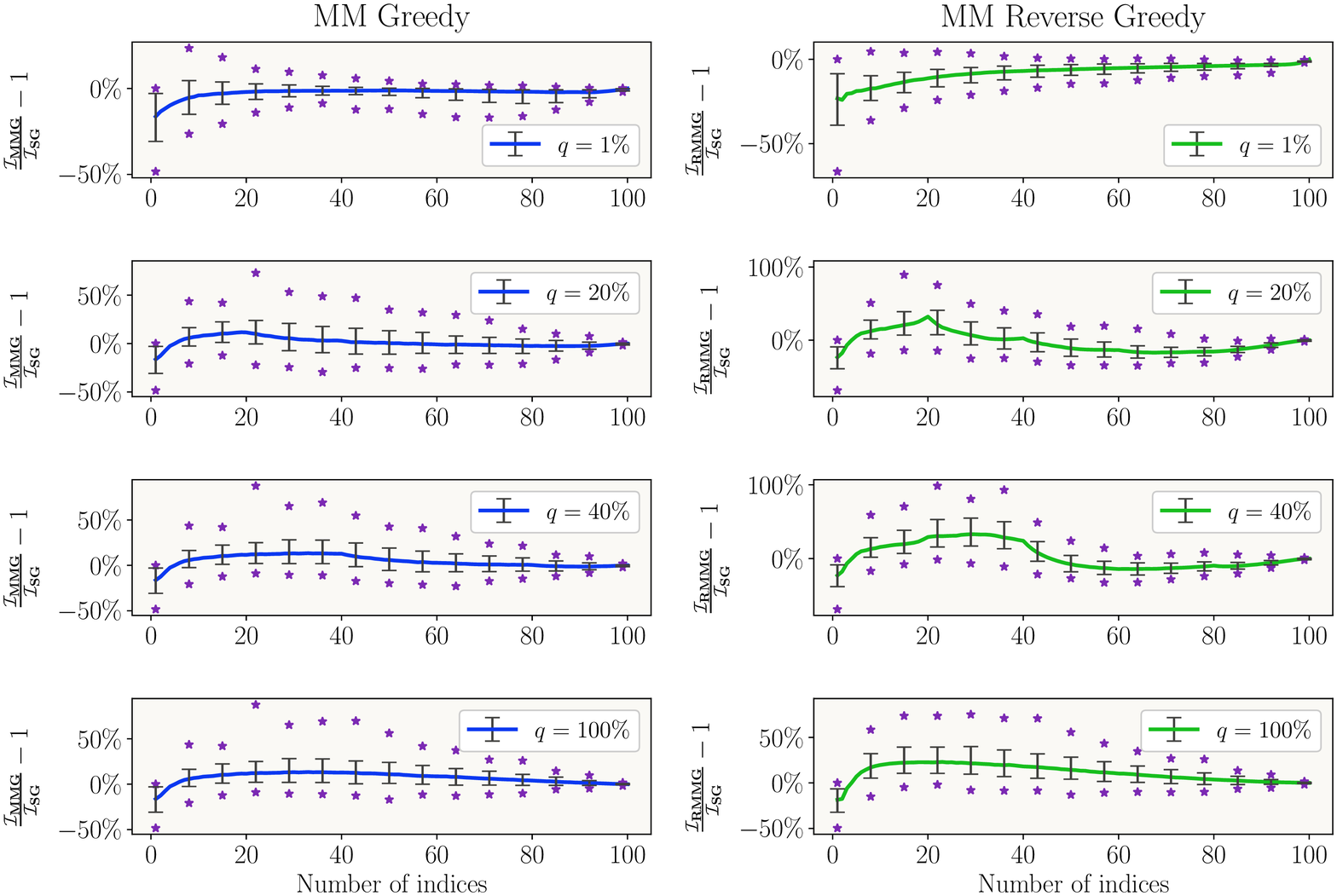}
  	 \caption{MM based batch greedy approaches compared against the standard batch greedy heuristic for different batch sizes across all 1000 random instances of the forward model for the case of inverse problem with exponential prior covariance. The solid line is the median; the whiskers capture the interquantile range (10\% to 90\%),  and the $\star$ marks the maximum and minimum mutual information captured.}
  	\label{fig:diff_MI_random_sampling_error_4x2subplot_corr_obs_error_exp_prior}
\end{figure}

\subsection{Structured random example with dimension of inference parameters greater than observations} \label{appendix:num_results_param_greater_than_obs}

We  consider a linear inverse problem with random linear forward model similar to the one described in \Cref{subsec:num_results_random}, but now the dimension of parameters is set to $100$, $X \in \mathbb{R}^{100}$, while cardinality of the candidate set of observations is fixed at $50$, $Y \in \mathbb{R}^{50}$.
We fix the prior, $\Gamma_X$ and observation error, $\Gamma_{Y|X}$, covariances using squared exponential kernel with correlation length $0.021$ and $0.042$ respectively. In \Cref{fig:spectrum_param_number_greater_than_obs,fig:batch_effect_subplot_param_number_greater_than_obs,fig:compare_algorithms_4subplot_param_number_greater_than_obs,fig:random_sampling_error_4x3subplot_param_number_greater_than_obs,fig:diff_MI_random_sampling_error_4x2subplot_param_number_greater_than_obs}
we present numerical investigations analogous to those in \Cref{subsec:num_results_random}.

\begin{figure}[h!] 	
  	\centering \includegraphics[width=\textwidth,angle=0]{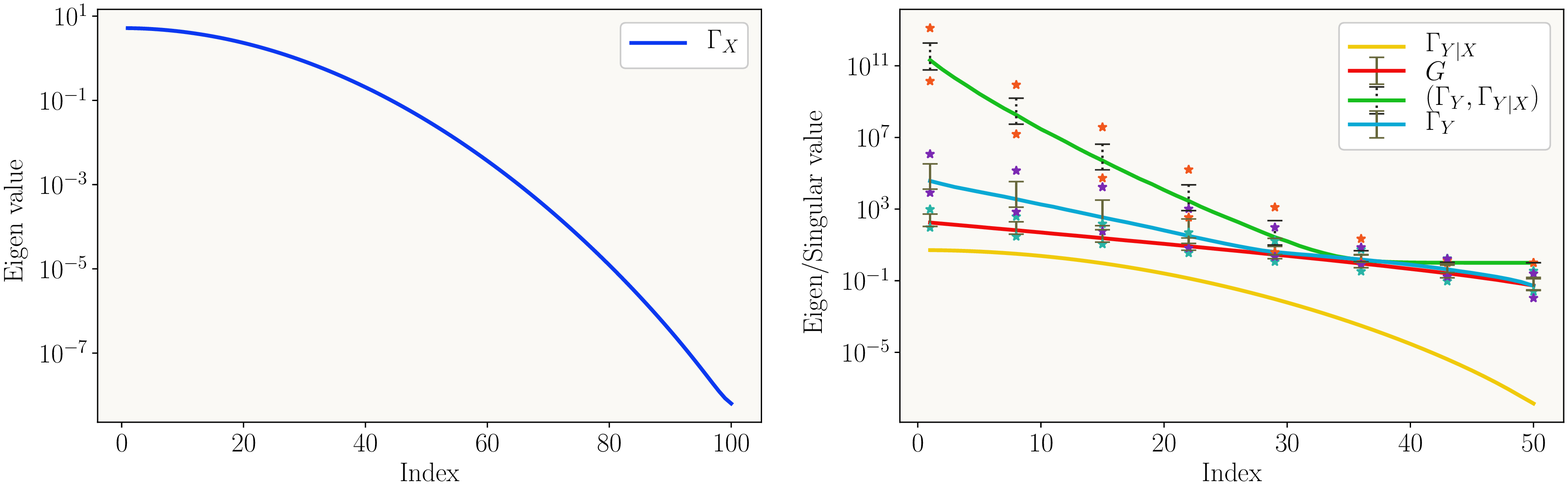}
  	 \caption{Spectrum of the relevant operators of the inverse problem with number of parameters greater than observations. The solid line is the median across 1000 random instances of the forward model. The whiskers capture the interquantile range (10\% to 90\%),  and the $\star$ marks the maximum and minimum eigen/singular value. The prior and observation error covariances are fixed and not random.}
  	\label{fig:spectrum_param_number_greater_than_obs}
\end{figure}

\begin{figure}[h!] 	
  	\centering \includegraphics[width=\textwidth,angle=0]{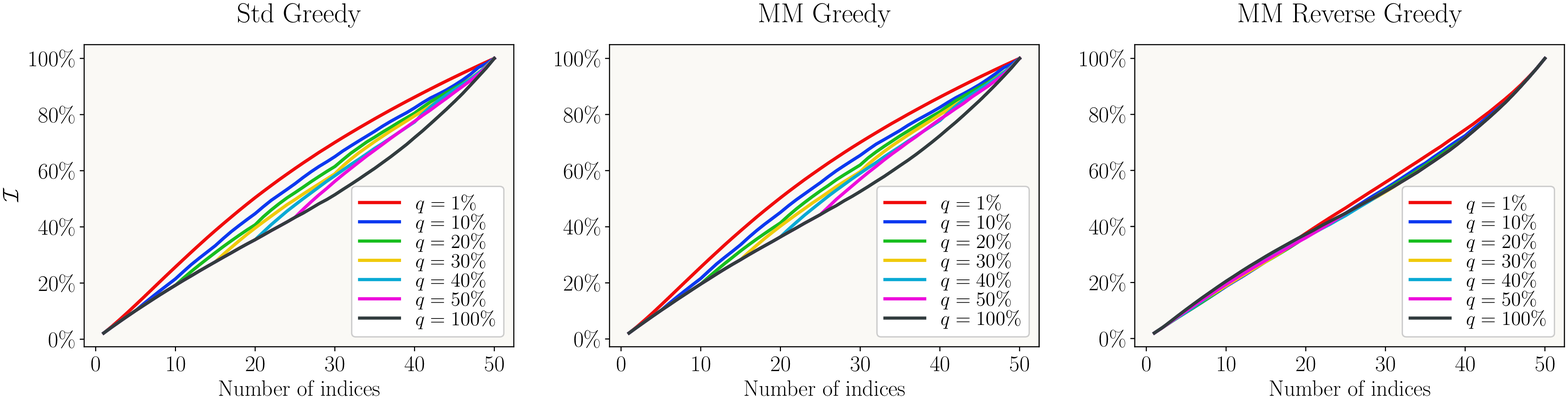}
	\caption{Performance of each greedy heuristic for different batch sizes for the case of inverse problem with number of parameters greater than observations. The batch size ranges from single index selection to one shot approach. The solid line is the median across 1000 random instances for the forward model.}
  	\label{fig:batch_effect_subplot_param_number_greater_than_obs}
\end{figure}

\begin{figure}[h!] 	
  	\centering \includegraphics[width=\textwidth,angle=0]{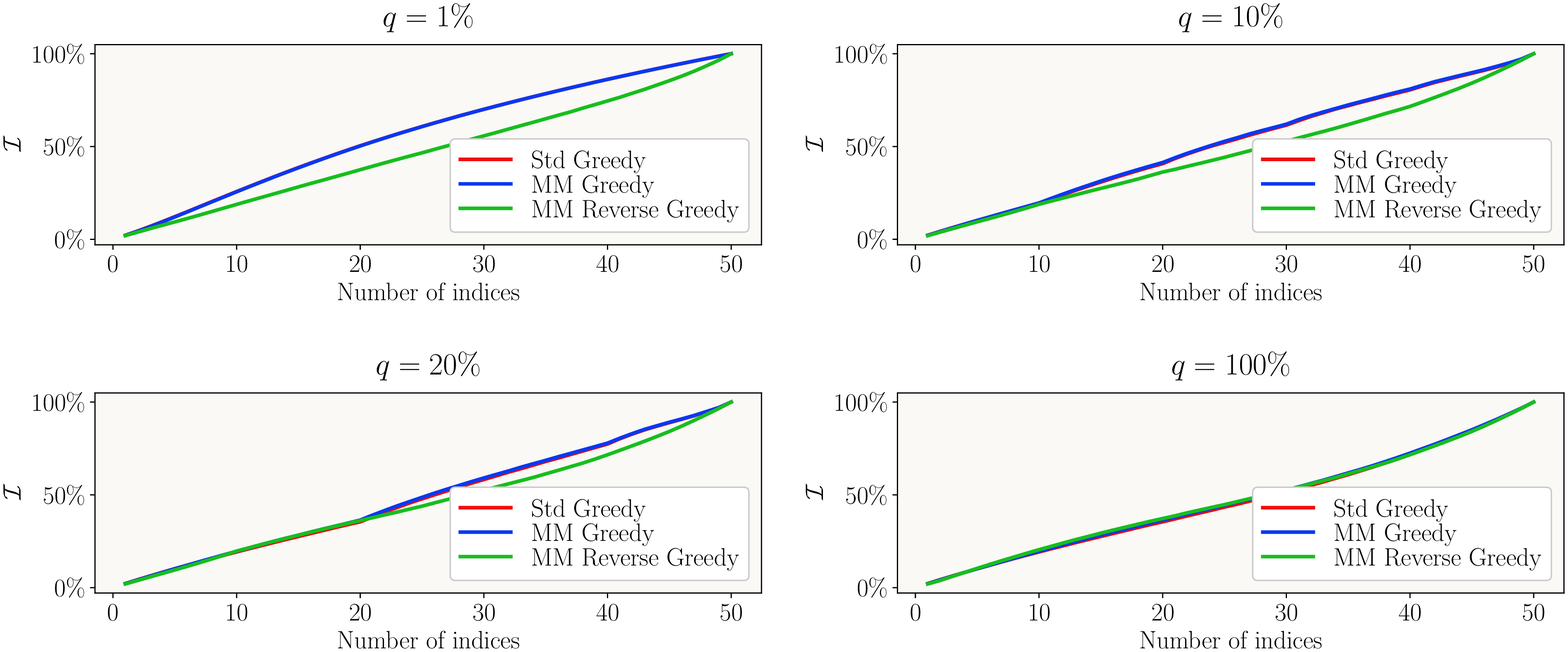}
	\caption{A comparative study of the greedy heuristics for four different batch sizes for the case of inverse problem with number of parameters greater than observations. The solid line is the median across 1000 random instances for the forward model.}
  	\label{fig:compare_algorithms_4subplot_param_number_greater_than_obs}
\end{figure}

\begin{figure}[h!] 	
  	\centering \includegraphics[width=\textwidth,angle=0]{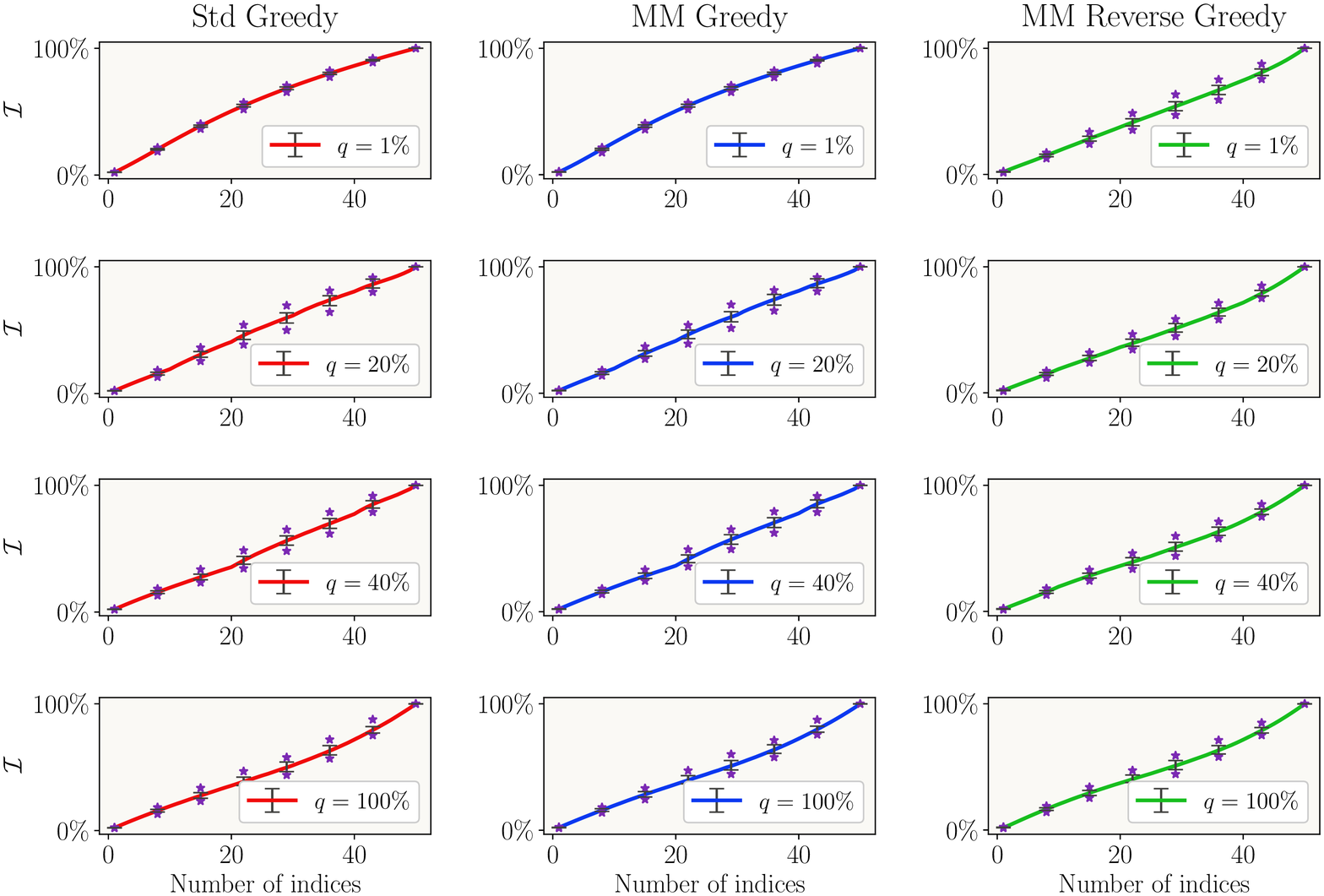}
	\caption{Performance of the greedy heuristics for different batch sizes across all 1000 random instances of the forward model for the case of inverse problem with number of parameters greater than observations. The solid line is the median; the whiskers capture the interquantile range (10\% to 90\%),  and the $\star$ marks the maximum and minimum mutual information captured.}
  	\label{fig:random_sampling_error_4x3subplot_param_number_greater_than_obs}
\end{figure}

\begin{figure}[h!] 	
  	\centering \includegraphics[width=\textwidth,angle=0]{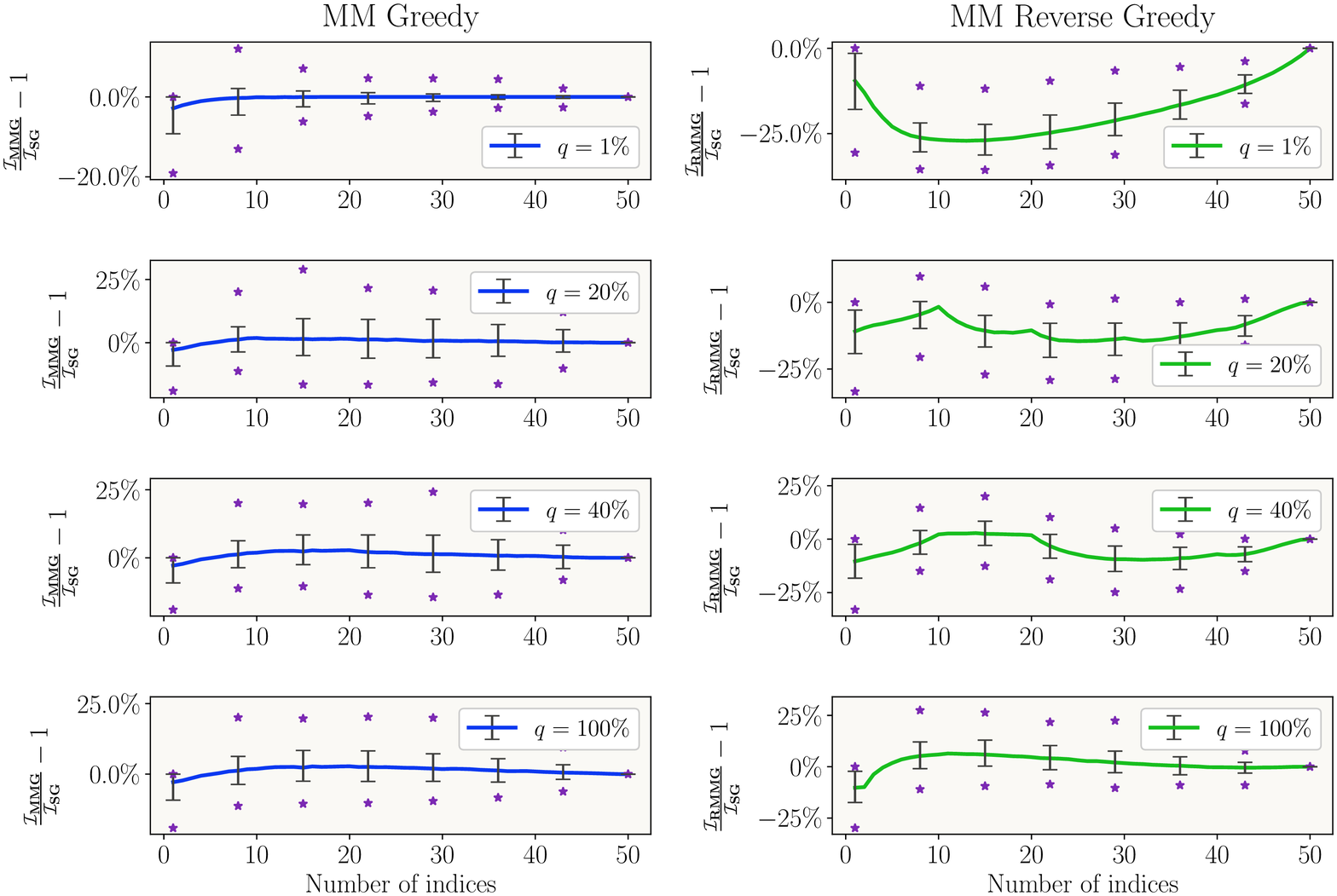}
  	 \caption{MM based batch greedy approaches compared against the standard batch greedy heuristic for different batch sizes across all 1000 random instances of the forward model for the case of inverse problem with number of parameters greater than observations. The solid line is the median; the whiskers capture the interquantile range (10\% to 90\%),  and the $\star$ marks the maximum and minimum mutual information captured.}
  	\label{fig:diff_MI_random_sampling_error_4x2subplot_param_number_greater_than_obs}
\end{figure}


\end{document}